\documentclass[nospthms]{svjour3}

\RequirePackage{fix-cm}
\smartqed  %
\journalname{}

\usepackage{amsmath}
\usepackage{amsthm}
\usepackage{amsfonts}
\usepackage{amssymb}
\usepackage{graphicx}
\usepackage[export]{adjustbox}
\graphicspath{{PaperFigs/}}
\usepackage[mathscr]{eucal}
\usepackage{enumerate}
\usepackage{cases}
\usepackage{pgf,tikz}\usetikzlibrary{matrix, calc, arrows,decorations.pathreplacing}
\usepackage{pgfplots}
\usepackage{color}
\usepackage{csquotes}
\usepackage{pdfcolfoot}  %
\usepackage{cite} %

\usepackage{algorithmic}
\usepackage[pdftex,hidelinks]{hyperref}

\definecolor{dhcol}{rgb}{0,0.5,0}

\definecolor{sccol}{rgb}{0,0,0.5}

\definecolor{amcol}{rgb}{0.5,0,0}

\usepackage{soul}

\DeclareMathOperator{\supp}{supp}

\DeclareMathOperator{\diam}{diam}
\DeclareMathOperator{\dist}{dist}

\begin{document}

\title{Boundary element methods for acoustic scattering by fractal screens}
\author{Simon N. Chandler-Wilde \and David P. Hewett \and Andrea Moiola \and Jeanne Besson}
\institute{
Simon N. Chandler-Wilde \at Department of Mathematics and Statistics, University of Reading, Whiteknights PO Box 220, Reading RG6 6AX, UK\\
\email{s.n.chandler-wilde@reading.ac.uk}
\and
David P. Hewett \at Department of Mathematics, University College London, Gower Street, London, WC1E 6BT, UK\\
\email{d.hewett@ucl.ac.uk}
\and
Andrea Moiola \at Dipartimento di Matematica, Universit\`a degli Studi di Pavia, Via Ferrata, 27100 Pavia, Italia\\
\email{andrea.moiola@unipv.it}
\and
Jeanne Besson \at ENSTA Paris, 828 Boulevard des Mar\'echaux, 91120 Palaiseau, France\\
\& Universit\'e Paris-Sud, 15 Rue Georges Clemenceau, 91400 Orsay, France\\
\email{jeanne.besson@ensta-paris.fr}
}

\date{Received: date / Accepted: date}

\maketitle

\newcommand{\rf}[1]{(\ref{#1})}
\newcommand{\mmbox}[1]{\fbox{\ensuremath{\displaystyle{ #1 }}}}	
\newcommand{\hs}[1]{\hspace{#1mm}}
\newcommand{\vs}[1]{\vspace{#1mm}}
\newcommand{\ri}{{\mathrm{i}}}
\newcommand{\re}{{\mathrm{e}}}
\newcommand{\rd}{\mathrm{d}}
\newcommand{\R}{\mathbb{R}}
\newcommand{\Q}{\mathbb{Q}}
\newcommand{\N}{\mathbb{N}}
\newcommand{\Z}{\mathbb{Z}}
\newcommand{\C}{\mathbb{C}}
\newcommand{\K}{{\mathbb{K}}}
\newcommand{\cA}{\mathcal{A}}
\newcommand{\cB}{\mathcal{B}}
\newcommand{\cC}{\mathcal{C}}
\newcommand{\cS}{\mathcal{S}}
\newcommand{\cD}{\mathcal{D}}
\newcommand{\cH}{\mathcal{H}}
\newcommand{\cI}{\mathcal{I}}
\newcommand{\cItilde}{\tilde{\mathcal{I}}}
\newcommand{\cIhat}{\hat{\mathcal{I}}}
\newcommand{\cIcheck}{\check{\mathcal{I}}}
\newcommand{\cIstar}{{\mathcal{I}^*}}
\newcommand{\cJ}{\mathcal{J}}
\newcommand{\cM}{\mathcal{M}}
\newcommand{\cP}{\mathcal{P}}
\newcommand{\cV}{{\mathcal V}}
\newcommand{\cW}{{\mathcal W}}
\newcommand{\scrD}{\mathscr{D}}
\newcommand{\scrS}{\mathscr{S}}
\newcommand{\scrJ}{\mathscr{J}}
\newcommand{\sD}{\mathsf{D}}
\newcommand{\sN}{\mathsf{N}}
\newcommand{\sS}{\mathsf{S}}
 \newcommand{\sT}{\mathsf{T}}
 \newcommand{\sH}{\mathsf{H}}
 \newcommand{\sI}{\mathsf{I}}
\newcommand{\bs}[1]{\mathbf{#1}}
\newcommand{\bb}{\mathbf{b}}
\newcommand{\bd}{\mathbf{d}}
\newcommand{\bn}{\mathbf{n}}
\newcommand{\bp}{\mathbf{p}}
\newcommand{\bP}{\mathbf{P}}
\newcommand{\bv}{\mathbf{v}}
\newcommand{\bx}{\mathbf{x}}
\newcommand{\by}{\mathbf{y}}
\newcommand{\bz}{{\mathbf{z}}}
\newcommand{\bxi}{\boldsymbol{\xi}}
\newcommand{\boldeta}{\boldsymbol{\eta}}	
\newcommand{\ts}{\tilde{s}}
\newcommand{\tGamma}{{\tilde{\Gamma}}}
 \newcommand{\tbx}{\tilde{\bx}}
 \newcommand{\tbd}{\tilde{\bd}}
 \newcommand{\txi}{\xi}
\newcommand{\done}[2]{\dfrac{d {#1}}{d {#2}}}
\newcommand{\donet}[2]{\frac{d {#1}}{d {#2}}}
\newcommand{\pdone}[2]{\dfrac{\partial {#1}}{\partial {#2}}}
\newcommand{\pdonet}[2]{\frac{\partial {#1}}{\partial {#2}}}
\newcommand{\pdonetext}[2]{\partial {#1}/\partial {#2}}
\newcommand{\pdtwo}[2]{\dfrac{\partial^2 {#1}}{\partial {#2}^2}}
\newcommand{\pdtwot}[2]{\frac{\partial^2 {#1}}{\partial {#2}^2}}
\newcommand{\pdtwomix}[3]{\dfrac{\partial^2 {#1}}{\partial {#2}\partial {#3}}}
\newcommand{\pdtwomixt}[3]{\frac{\partial^2 {#1}}{\partial {#2}\partial {#3}}}
\newcommand{\bnabla}{\boldsymbol{\nabla}}
\newcommand{\dive}{\boldsymbol{\nabla}\cdot}
\newcommand{\curl}{\boldsymbol{\nabla}\times}
\newcommand{\Phixy}{\Phi(\bx,\by)}
\newcommand{\PhiOxy}{\Phi_0(\bx,\by)}
\newcommand{\dxPhixy}{\pdone{\Phi}{n(\bx)}(\bx,\by)}
\newcommand{\dyPhixy}{\pdone{\Phi}{n(\by)}(\bx,\by)}
\newcommand{\dxPhiOxy}{\pdone{\Phi_0}{n(\bx)}(\bx,\by)}
\newcommand{\dyPhiOxy}{\pdone{\Phi_0}{n(\by)}(\bx,\by)}
\newcommand{\eps}{\varepsilon}
\newcommand{\real}[1]{{\rm Re}\left[#1\right]} 
\newcommand{\im}[1]{{\rm Im}\left[#1\right]}
\newcommand{\ol}[1]{\overline{#1}}
\newcommand{\ord}[1]{\mathcal{O}\left(#1\right)}
\newcommand{\oord}[1]{o\left(#1\right)}
\newcommand{\Ord}[1]{\Theta\left(#1\right)}
\newcommand{\hsnorm}[1]{||#1||_{H^{s}(\bs{R})}}
\newcommand{\hnorm}[1]{||#1||_{\tilde{H}^{-1/2}((0,1))}}
\newcommand{\norm}[2]{\left\|#1\right\|_{#2}}
\newcommand{\normt}[2]{\|#1\|_{#2}}
\newcommand{\on}[1]{\Vert{#1} \Vert_{1}}
\newcommand{\tn}[1]{\Vert{#1} \Vert_{2}}
\newcommand{\xt}{\mathbf{x},t}
\newcommand{\PhiF}{\Phi_{\rm freq}}
\newcommand{\cone}{{c_{j}^\pm}}
\newcommand{\ctwo}{{c_{2,j}^\pm}}
\newcommand{\cthree}{{c_{3,j}^\pm}}
\newtheorem{thm}{Theorem}[section]
\newtheorem{lem}[thm]{Lemma}
\newtheorem{defn}[thm]{Definition}
\newtheorem{prop}[thm]{Proposition}
\newtheorem{cor}[thm]{Corollary}
\newtheorem{rem}[thm]{Remark}
\newtheorem{conj}[thm]{Conjecture}
\newtheorem{ass}[thm]{Assumption}
\newcommand{\tH}{\widetilde{H}}
\newcommand{\Hze}{H_{\rm ze}} 	
\newcommand{\uze}{u_{\rm ze}}		
\newcommand{\dimH}{{\rm dim_H}}
\newcommand{\dimB}{{\rm dim_B}}
\newcommand{\IntClosOm}{\mathrm{int}(\overline{\Omega})}
\newcommand{\IntClosOmOne}{\mathrm{int}(\overline{\Omega_1})}
\newcommand{\IntClosOmTwo}{\mathrm{int}(\overline{\Omega_2})}
\newcommand{\Ccomp}{C^{\rm comp}}
\newcommand{\tCcomp}{\tilde{C}^{\rm comp}}
\newcommand{\uC}{\underline{C}}
\newcommand{\utC}{\underline{\tilde{C}}}
\newcommand{\oC}{\overline{C}}
\newcommand{\otC}{\overline{\tilde{C}}}
\newcommand{\capcomp}{{\rm cap}^{\rm comp}}
\newcommand{\Capcomp}{{\rm Cap}^{\rm comp}}
\newcommand{\tcapcomp}{\widetilde{{\rm cap}}^{\rm comp}}
\newcommand{\tCapcomp}{\widetilde{{\rm Cap}}^{\rm comp}}
\newcommand{\hcapcomp}{\widehat{{\rm cap}}^{\rm comp}}
\newcommand{\hCapcomp}{\widehat{{\rm Cap}}^{\rm comp}}
\newcommand{\tcap}{\widetilde{{\rm cap}}}
\newcommand{\tCap}{\widetilde{{\rm Cap}}}
\newcommand{\ccap}{{\rm cap}}
\newcommand{\ucap}{\underline{\rm cap}}
\newcommand{\uCap}{\underline{\rm Cap}}
\newcommand{\cCap}{{\rm Cap}}
\newcommand{\ocap}{\overline{\rm cap}}
\newcommand{\oCap}{\overline{\rm Cap}}
\DeclareRobustCommand
{\mathringbig}[1]{\accentset{\smash{\raisebox{-0.1ex}{$\scriptstyle\circ$}}}{#1}\rule{0pt}{2.3ex}}
\newcommand{\cirH}{\mathringbig{H}}
\newcommand{\cirHs}{\mathringbig{H}{}^s}
\newcommand{\cirHt}{\mathringbig{H}{}^t}
\newcommand{\cirHm}{\mathringbig{H}{}^m}
\newcommand{\cirHzero}{\mathringbig{H}{}^0}
\newcommand{\deO}{{\partial\Omega}}
\newcommand{\OO}{{(\Omega)}}
\newcommand{\Rn}{{(\R^n)}}
\newcommand{\Id}{{\mathrm{Id}}}
\newcommand{\gap}{\mathrm{Gap}}
\newcommand{\ggap}{\mathrm{gap}}
\newcommand{\isom}{{\xrightarrow{\sim}}}
\newcommand{\half}{{1/2}}
\newcommand{\mhalf}{{-1/2}}
\newcommand{\inter}{{\mathrm{int}}}

\newcommand{\Hsp}{H^{s,p}}
\newcommand{\Htq}{H^{t,q}}
\newcommand{\tHsp}{{{\widetilde H}^{s,p}}}
\newcommand{\SP}{\ensuremath{(s,p)}}
\newcommand{\Xsp}{X^{s,p}}

\newcommand{\dd}{{d}}\newcommand{\pp}{{p_*}}

\newcommand{\Rnn}{\R^{n_1+n_2}}
\newcommand{\Tr}{{\mathrm{Tr}}}
\newcommand{\sO}{\mathsf{O}}
\newcommand{\sC}{\mathsf{C}}
\newcommand{\sA}{\mathsf{A}}
\newcommand{\sM}{\mathsf{M}}
\newcommand{\sF}{\mathsf{F}}
\newcommand{\sG}{\mathsf{G}}
\newcommand{\mS}{\Gamma}
\newcommand{\omS}{{\overline{\mS}}}
\newcommand{\sumpm}[1]{\{\!\!\{#1\}\!\!\}}
\newcommand{\GammajFat}{{\Gamma_j^{\delta}}}
\newcommand{\phijBEM}{{\phi_j^h}}
\newcommand{\VjBEM}{{V_j^h}}
\newcommand{\FAT}{^{\text{\sc fat}}}

\begin{abstract}
We study boundary element methods for time-harmonic %
scattering in $\R^n$ ($n=2,3$) by a %
fractal planar screen, assumed to be a non-empty bounded subset $\Gamma$ of the hyperplane $\Gamma_\infty=\R^{n-1}\times \{0\}$. We consider two distinct cases: (i) $\Gamma$ is a relatively open subset of $\Gamma_\infty$ with fractal boundary (e.g.\ the interior of the Koch snowflake in the case $n=3$); (ii) $\Gamma$ is a compact fractal subset of $\Gamma_\infty$
with empty interior (e.g.\ the Sierpinski triangle in the case $n=3$). In both cases our numerical simulation strategy involves approximating the fractal screen $\Gamma$ by a sequence of smoother ``prefractal'' screens, for which we compute the scattered field using boundary element methods that discretise the associated first kind boundary integral equations.
We prove sufficient conditions on the mesh sizes guaranteeing convergence to the limiting fractal solution, using the
framework of Mosco convergence. We also provide numerical examples illustrating our theoretical results.
\end{abstract}

\section{Introduction}
The scattering of acoustic waves by screens (or ``cracks'' in the elasticity literature) is a classical topic in physics, applied mathematics and scientific computing.
The basic scattering problem involves an incident wave propagating in $\R^n$ ($n=2$ or $3$), striking a screen $\Gamma$, assumed to be a bounded subset (typically, a relatively open subset) of some $(n-1)$-dimensional submanifold of $\R^n$, and producing a scattered field which radiates outward to infinity.
In a homogeneous background medium, the scattering problem can be reformulated as a boundary integral equation (BIE) on the screen, as described in e.g.\ \cite{Du:83,StWe84,WeSt:90,stephan87,Ha-Du:90}, %
and numerical solutions can then be computed using the boundary element method (BEM), as in e.g.\ \cite{StWe84,WeSt:90,stephan87,holm1996hp}. %
The classical work cited above has since been extended in many directions, e.g.\ to the electromagnetic case \cite{BuCh:03,bespalov2010convergence}, to ``multi-screens'' (the union of multiple screens intersecting non-trivially) \cite{ClHi:13}, and to hybrid numerical-asymptotic approximation spaces \cite{ScreenBEM} for high-frequency problems.

For simplicity we focus on the case where the screen is flat, and the underlying $(n-1)$-dimensional submanifold is the hyperplane $\Gamma_\infty:=\R^{n-1}\times \{0\}\subset \R^n$.
Restricted to this setting, existing studies all assume (either explicitly or implicitly) that the screen $\Gamma\subset \Gamma_\infty$ is a (relatively) open set with smooth (or piecewise smooth) boundary.
In our recent paper \cite{ScreenPaper} we derived well-posed boundary value problem (BVP) and BIE formulations for sound-soft and sound-hard acoustic scattering by \textit{arbitrary} bounded screens $\Gamma\subset \Gamma_\infty$, including cases where $\Gamma$ or $\partial\Gamma$ has fractal structure.
In this paper we consider the numerical solution of these BVPs/BIEs using the BEM.

Our focus is on scattering by two general classes of fractal screens\footnote{While our focus in this paper is on fractal screens, our main results (for instance, Theorems \ref{prop:DiscreteOpen} and \ref{prop:DiscreteCompact}) do not require fractality (neither self-similarity nor non-integer fractal dimension), but apply to any non-smooth screen approximated by a sequence of smoother screens.}:
\begin{enumerate}[(i)]
\item $\Gamma$ is a bounded, relatively open subset of $\Gamma_\infty$ with fractal boundary, for instance the interior of the Koch snowflake in the case $n=3$;
\item $\Gamma$ is a compact fractal subset of $\Gamma_\infty$ with empty relative interior, for instance the Sierpinski triangle in the case $n=3$.
\end{enumerate}
In both cases our general approach to analysis and numerical simulation is to approximate the fractal screen $\Gamma$ by a sequence of smoother ``prefractal'' screens $\Gamma_j$, $j\in\N_0$, for which BVP/BIE well-posedness and BEM approximation is classical.
The key question we address in this paper is:
\begin{quote}
\textit{Given a fractal screen $\Gamma$, how should the prefractals $\Gamma_j$ and the corresponding BEM discretisations be chosen so as to ensure convergence of the numerical solutions on $\Gamma_j$ to the limiting solution on $\Gamma$?}
\end{quote}

In this paper we focus exclusively on sound-soft screens, on which homogeneous Dirichlet boundary conditions are imposed.
Our decision to restrict attention to this case was made partly to make numerical simulations as simple as possible,
since one has only to discretise the weakly-singular single-layer boundary integral operator.
But the sound-soft case is also particularly interesting from a physical point of view, as there is a strong dependence of the scattering properties on the fractal dimension of the screen (see Proposition~\ref{prop:CantorDust} and the numerical results in \S\S\ref{sec:NumCantorSet}--\ref{sec:NumCantorDust}).
We leave for future work the application of the techniques developed in this paper to the numerical simulation of the well-posed BVP/BIE formulations for scattering by fractal sound-hard and impedance screens presented in \cite{CWHewettWaves2017,ScreenPaper,HewettBannisterWaves2019}. (See \cite{CWHewettWaves2017} for simulations of diffraction through fractal apertures in sound-hard screens, equivalent, by Babinet's Principle \cite{Bouwkamp:54}, to the sound-soft screen problem that we focus on in this paper.)%

\paragraph{Our main results and their novelty.} The main focus of this paper is to address the ``key question'' above, proving results (Theorems \ref{prop:DiscreteOpen} and \ref{prop:DiscreteCompact}) that specify, for each of the classes (i) and (ii), how to choose a sequence of prefractals and their BEM discretisations so as to achieve convergence of the resulting numerical scheme to the limiting solution on $\Gamma$. While BEM simulations have been carried out previously on sequences of prefractals (see Jones et al.~\cite{jones1994fast} in the context of potential theory and Panagiotopulos and Panagouli \cite{panagiotopoulos1996bem} in elasticity),  prior to the results in this paper there does not appear to exist any analysis to justify the convergence of such simulations, that the sequence of numerical solutions converges to the desired limiting solution on $\Gamma$.   %

Our focus throughout the paper is on particular BIEs for sound-soft fractal screens, but we expect the methods and arguments that we introduce to be much more widely applicable.  Indeed, our analysis is based on a variational formulation of the BIEs in terms of (complex-valued) continuous sesquilinear forms, which allows the question of prefractal to fractal convergence to be rephrased in terms of the Mosco %
convergence\footnote{Mosco convergence (Definition \ref{def:mosco} below) is a standard notion in the study of variational inequalities, closely related to Gamma convergence. Introduced by U.\ Mosco in \cite{Mosco69} (almost exactly 50 years ago), it has been applied by a number of authors to the study of PDEs on sequences of domains, see e.g.\ \cite{BuVa2000,Daners:03,MenegattiRondi13,ArLa:17} and the references therein.
It has been mainly used (e.g.\ the references just cited) in relation to convergence at a continuous rather than a discrete level (i.e. in the context of mathematical analysis rather than numerical analysis). But it is also relevant to proving convergence of numerical methods, as was illustrated already in \cite[Chapt.~3]{Mosco1971} in the context of numerical methods for variational inequalities.}
of the discrete BEM subspaces to the fractional Sobolev space in which the limiting fractal solution lives.
The methods that we develop, to reduce proof of convergence of the numerical solution to the Mosco convergence of the BEM subspaces (Lemma \ref{lem:dec3}), and to prove Mosco convergence of the BEM subspaces, are potentially widely applicable to Galerkin discretisations of other integral or differential equations posed on rough (not necessarily fractal) domains that are approximated by more regular sets. For example, the proof of Theorem \ref{prop:DiscreteCompact} depends only on a characterisation of Mosco convergence (Lemma~\ref{lem:MoscoSuff}), quantitative bounds on the norms of mollification operators on scales of Sobolev spaces (Appendix \ref{sec:Mollification}), and a quantitative extension of standard piecewise constant finite element approximation theory (Lemma \ref{lem:PWc}), all of which should be widely applicable.

A feature of our numerical analysis based on Mosco convergence is that our discrete BEM subspaces need not be subspaces of the Hilbert space
 in which the solution on $\Gamma$ lies; indeed, this is crucial whenever the prefractals are not subsets of the limiting fractal $\Gamma$. Our Lemma \ref{lem:dec3}, which applies in such cases (and is proved in slightly more generality than we need for the results in \S\ref{sec:Convergence}, anticipating wider application) can be seen as a replacement in this circumstance for the standard C\'ea lemma and its generalisation to compact perturbations (see, e.g.~\cite[Thms.~ 8.10, 8.11]{Steinbach}).

While our main results, Theorems \ref{prop:DiscreteOpen} and \ref{prop:DiscreteCompact}, are quite general in terms of geometry of the screen $\Gamma$ and its approximating prefractal sequence, we pay particular attention in our theory, examples, and numerical simulations to cases where $\Gamma$ (or the boundary of $\Gamma$) is the fixed point of an iterated function system (IFS) (e.g.\ Corollaries \ref{cor:ifs} and \ref{cor:dsetconv}). In particular our examples in \S\ref{sec:examples} and \S\ref{sec:Numerics} include the cases where $\Gamma$ is (when $n=2$) a Cantor set and (when $n=3$) a Cantor dust, the Sierpinski triangle, or the interior of a Koch snowflake.  These are cases where the BIEs we wish to solve are posed either on fractal sets or on rough domains with fractal boundaries. Unsurprisingly,  subtle and interesting properties of fractional Sobolev spaces and integral operators on rough sets, explored recently in \cite{ChaHewMoi:13,InterpolationCWHM,HewMoi:15,CoercScreen2,caetano2018}, are crucial to our arguments throughout.

A novel feature of our BEM and its analysis is that convergence can be achieved in regimes where each boundary element contains many disjoint components of a prefractal (e.g.\ Corollary \ref{cor:dsetconv}(ii) and Figure \ref{fig:PreConvexMeshes}). To justify this we need an extension, that applies in such cases, of standard piecewise constant approximation theory in scales of Sobolev spaces with explicit constants. We supply this in Lemma \ref{lem:PWc}.

\paragraph{Applications and motivations.}
Wave scattering by fractal structures is relevant for numerous applications, since
fractals provide a natural mathematical framework for modelling the multiscale roughness of many natural and man-made scatterers. We highlight in particular the propagation of acoustic and electromagnetic waves in dendritic structures like the human lung in medical science \cite{achdou2007transparent,JoSe:11}, and
the scattering of electromagnetic waves by snowflakes, ice crystals and other atmospheric particles in climate science \cite{Baran:2009,tyynela2011radar,ceolato2013spectral,StWe:15}.
 But particularly close to the fractal screen scattering problems that we focus on in this paper are configurations arising in the design of fractal antennas for electromagnetic wave transmission/reception (see e.g.\ \cite{PBaRomPouCar:98,WeGa:03,GhSiKa:14}) and fractal piezoelectric ultrasound transducers (see e.g.\ \cite{MuWa:11,AlMu:15,algehyne2018analysis,fang2018broadband})  (fractal structures being attractive in these contexts because of the possibility of wideband performance), and configurations that arise in the study of fractal aperture problems in laser physics \cite{jaggard1998cantor,horvath2010koch,verma2013robustness,Ch:16}.
The current study into acoustic scattering by fractal screens represents a first step towards the rigorous numerical analysis of integral equation methods for the study of such challenging problems involving fractal scatterers.

\paragraph{Related literature.}
One of the first studies of wave scattering by fractals appears to be M. Berry's 1979 paper \cite{Berry79} on scattering by ``random phase screens'', in which Berry coins the term ``diffractal'' to describe waves that have undergone interactions with fractal structures. The difficult problem of studying high frequency asymptotics for fractal scattering problems was investigated by Sleeman and Hua in \cite{sleeman1992analogue,hua2001high}.
Concerning the study of PDE problems on fractal domains more generally, U. Mosco notes in \cite{mosco2013analysis} that \textit{``\ldots introducing fractal constructions into the classic theory of PDEs opens a vast new field of study, both theoretically and numerically''}, but also that \textit{``this new field has been only scratched''}.
In addition to papers cited above, research that has started to explore this field of study includes work on finite element method approximations of heat transmission across fractal interfaces \cite{Bagnerini13,capitanelli2017fem,cefalo2013heat,cefalo2014optimal,lancia2012numerical} and $H^1$ extension problems \cite{evans2012finite};
Dirichlet-to-Neumann (Poincar\'e-Steklov) operators on domains with fractal boundaries \cite{arfi2017dirichlet}; and finite difference \cite{neuberger2006computing} and conformal mapping \cite{banjai2017poincar} approaches to the computation of Laplace eigenfunctions on fractal domains.

\paragraph{Structure of the paper.}
In \S\ref{sec:Prelim} we collect some important Hilbert and Sobolev space results that will be used throughout the paper. The main new result here is Lemma \ref{lem:dec3}, which proves convergence of solutions of variational problems for ``compactly perturbed coercive'' sesquilinear forms on a Mosco-convergent sequence of closed subspaces.
The compactly perturbed coercive setting is more general than we need for the particular problem under consideration in this paper, since for flat sound-soft screens the first-kind formulation is coercive (strongly elliptic) \cite{CoercScreen2}.
However, it is included here to lay the foundations for future investigations into other, closely related problems, such as curved sound-soft screens (as studied e.g.\ in \cite{StWe84,WeSt:90}), and impedance screens (as studied e.g.\ in  \cite{kress2003integral,ben2013application,HewettBannisterWaves2019,BaGiHe:20}).
In \S\ref{sec:BVPs} we state well-posed BVPs and BIEs for open and closed screens, refining results from \cite{ScreenPaper}, and in \S\ref{sec:Conv} we prove convergence of solutions on prefractal screens to solutions on limiting fractal screens using the Mosco framework; in particular we prove for the first time convergence in cases where the prefractal sequence is not monotonic, and it holds neither that $\Gamma\subset \Gamma_j$ for all $j$, nor that $\Gamma_j\subset \Gamma$ for all $j$. In \S\ref{sec:Convergence} we study numerical discretisations based on piecewise constant BEM approximations on prefractals, determining conditions under which the BEM solution converges to the limiting fractal solution in the joint limit of prefractal and mesh refinement. In \S\ref{sec:examples} we present examples of the kind of fractal screens  we have in mind, and in \S\ref{sec:Numerics} we provide numerical results which illustrate our theoretical predictions through a number of concrete examples. We also include in this section preliminary numerical investigations into physical questions such as how the fractal dimension of a screen affects the magnitude of the resulting scattered field. In \S\ref{sec:conclude} we make a brief conclusion and list some of the many intriguing open problems.

\section{Preliminaries}\label{sec:Prelim}
\subsection{Dual space realisations}\label{sec:Hilbert}
We say that a linear isomorphism between Hil\-bert spaces is \emph{unitary} if it preserves the inner product (equivalently, if it is isometric \cite[Prop.~5.2]{Conway}).
If $H$ is a complex Hilbert space (as all the Hilbert spaces in this paper are)
 by its {\em dual space} $H^*$ we mean, following Kato \cite{Ka:95}, the space of {\em anti-linear} continuous functionals on $H$.
It is often convenient to identify $H^*$, itself a Hilbert space with the standard induced operator norm, with some other Hilbert space $\mathcal{H}$, termed a {\em realisation} of $H^*$.
If $\langle\cdot, \cdot \rangle$ is a continuous sesquilinear form on $\mathcal{H}\times H$, and the mapping taking $\phi\in \mathcal{H}$ to $\ell_\phi\in H^*$, given by $\ell_\phi(\psi)= \langle\phi,\psi\rangle$, $\psi\in H$, is a unitary isomorphism,
then we say that $\cH$
is a {\em unitary realisation} of $H^*$ with associated {\em duality pairing}
$\langle\cdot, \cdot \rangle$.
If
$W\subset H$ is a closed subspace, then a unitary realisation of $W^*$ is provided by the following simple but important result, e.g.\ \cite[Lem.~2.2]{ChaHewMoi:13}, which is a special case of a more general Banach space result, e.g.\ \cite[Thm.~4.9]{Ru91}.

\begin{lem} \label{lem:hs_orth} Suppose that $H$ and $\cH$ are Hilbert spaces, and $\cH$ is a unitary realisation of $H^*$, with duality pairing $\langle\cdot,\cdot\rangle$,
and $W\subset H$ is a closed linear subspace.
Set ${\mathcal W} := \left(W^{a,\mathcal{H}}\right)^\perp\subset \cH$, where $\perp$ denotes orthogonal complement and
\begin{align*}%
W^{a,\mathcal{H}} := \{\psi\in \cH:\langle \psi,\phi\rangle = 0, \mbox{ for all }\phi\in W\}\subset\cH
\end{align*}
is the {\em annihilator of $W$ in $\cH$}.
Then $\mathcal W$ is a unitary realisation of $W^*$, and its associated duality pairing is just the restriction to $\mathcal{W}\times W$ of the duality pairing on $\mathcal{H}\times H$.
\end{lem}
For a closed subspace $W\subset H$,
$(W^\perp)^{a,\mathcal{H}}= (W^{a,\mathcal{H}})^\perp$ and $(W^{a,\mathcal{H}})^{a,H} = W$.
\subsection{Variational problems} \label{sec:var}
Suppose $H$ is a Hilbert space with norm $\|\cdot\|_H$, and $a(\cdot,\cdot)$ is a sesquilinear form on $H\times H$ that is {\em continuous}, i.e., for some $C>0$ (the {\em continuity constant}), $|a(u,v)|\le C\|u\|_H \|v\|_H$, for all $u,v\in H$. To each such $a$, and each unitary realisation $\cH$ of the dual space $H^*$, with associated duality pairing $\langle\cdot,\cdot\rangle$, there corresponds a unique bounded linear operator $A:H\to \cH$ defined by
\begin{equation} \label{eq:blo}
a(u,v) = \langle Au,v\rangle, \quad u,v\in H.
\end{equation}
Conversely, every bounded linear operator $A:H\to \cH$ defines via \eqref{eq:blo} an associated sesquilinear form $a(\cdot,\cdot)$ on $H\times H$.

We say that $a$ and
$A$ are {\em coercive} if, for some $\alpha>0$ (the {\em coercivity constant}),
\begin{equation*} %
|a(u,u)|\ge \alpha\|u\|^2_H,
\quad \mbox{for all } u\in H.
\end{equation*}
We recall that $A$ is compact, i.e.\ maps bounded sets to relatively compact sets, if and only if it is completely continuous, i.e., for every sequence $(u_j)\subset H$ and $u\in H$,
$u_j\rightharpoonup u$  implies that $Au_j\to Au$. (Here $\to$ denotes norm convergence in $\cH$ and $\rightharpoonup$ weak convergence in $H$.) We say that $a$ is {\em compact} if $A$ is compact: equivalently,
if $a(u_j,v_j)\to a(u,v)$ whenever $(u_j)\subset H$, $(v_j)\subset H$ and $u,v\in H$ satisfy $u_j\rightharpoonup u$ and $v_j\rightharpoonup v$.
We say that $a$ and $A$ are {\em compactly perturbed coercive} if $a=a_0+a_1$ %
with $a_0$ coercive and $a_1$ compact; equivalently, $A=A_0+A_1$ with $A_0$ coercive and $A_1$ compact.
\footnote{Terminology varies: what we call here {\em coercive} is often termed {\em $H$-elliptic} or {\em strongly elliptic}, and what we call here {\em compactly perturbed coercive} is often termed {\em coercive} (e.g.\ \cite{Steinbach}).}

Let $W\subset H$ be a closed subspace of $H$ %
and let $\cW := (W^{a,\cH})^\perp\subset \cH$ be the unitary realisation of $W^*$ provided by Lemma \ref {lem:hs_orth}.  %
We say that $a$ is invertible on $W$ if, for every $f\in \cW$, the variational problem: find $u_W\in W$ such that
\begin{equation} \label{eq:varW}
a(u_W,v) = \langle f,v\rangle, \quad \mbox{for all } v\in W,
\end{equation}
has exactly one solution $u_W\in W$. This holds (e.g.\ \cite[Thm.~2.15]{Ihlenburg}) if and only if
\begin{equation*} %
\beta := \inf_{u\in W, \,\|u\|_H=1}\sup_{v\in W,\, \|v\|=1}|a(u,v)| >0
\quad\text{and} \; \sup_{u\in W}|a(u,v)| >0 \quad\forall v\in W \setminus\{0\}.
\end{equation*}
In terms of the associated operator $A$, \eqref{eq:varW} can be written equivalently as
\begin{equation} \label{eq:proj}
P_{\cW}A u = f,
\end{equation}
 where $P_\cW$ is orthogonal projection in $\cH$ onto $\cW$, so that $a$ is invertible on $W$ if and only if $P_{\cW}A|_W$ is invertible, in which case $\|\left(P_{\cW}A|_W\right)^{-1}\|=\beta^{-1}$.
If $a$ is coercive then $a$ is invertible on $W$ by the Lax--Milgram lemma and $\beta\ge\alpha$.
More generally, if $a$ is compactly perturbed coercive, then $a$ is invertible on $W$ if and only if it is injective, meaning that the problem \eqref{eq:varW} has at most one solution $u_W\in W$ for every $f\in \cW$.

\subsection{Mosco convergence} \label{sec:Mosco}
We now consider the problem of approximating the solution of the variational problem \eqref{eq:varW} by the solutions of variational problems posed on a sequence of closed subspaces $(W_j)_{j=1}^\infty \subset H$.
We say that $a$ is \emph{uniformly invertible} on such a sequence %
$(W_j)_{j=1}^\infty$ if $a$ is invertible on $W_j$ for all $j\in \N$ and the inverses are uniformly bounded, meaning that, for some constant $C>0$ and all $j\in \N$ and $f_j\in \cW_j:= (W_j^{a,\cH})^\perp$,
$$
\|u_{W_j}\|_H \leq C \|f_j\|_{\cW_j},
$$
where $u_{W_j}$ is the unique solution of \eqref{eq:varW} with $W$ replaced by $W_j$ and $f$ by $f_j$. Equivalently, $a$ is uniformly invertible on $(W_j)$ if $P_{\cW_j}A|_{W_j}$ is invertible for $j\in \N$ and
$$
\sup_{j\in \N} \left\|\left(P_{\cW_j}A|_{W_j}\right)^{-1}\right\| <\infty.
$$

Roughly speaking, given a variational problem \eqref{eq:varW}, to ensure that the corresponding solutions on $(W_j)_{j=1}^\infty$ converge to the solution in $W$ we require that $a$ is sufficiently ``well-behaved'' %
 and that $W_j$ approximates $W$ increasingly well as $j\to\infty$ in an appropriate sense.  The precise requirement on $W_j$ is that it converges to  $W$ in the Mosco sense (Lemma \ref{lem:Mosco} below), and Lemma \ref{lem:dec3} makes clear that $a$ being compactly perturbed coercive is sufficiently ``well-behaved''. 
The following definition of Mosco convergence is precisely the notion of set convergence for convex sets introduced in \cite[Definition 1.1]{Mosco69}, except that we specialise here from the general Banach space setting to the Hilbert space case, and our convex sets are specifically closed linear subspaces (as, for example, in \cite{MenegattiRondi13}).

\begin{defn}[Mosco convergence] \label{def:mosco} Let $W$ and $W_j$, for $j\in \N$, be closed subspaces of a Hilbert space $H$. We say that $W_j$ converges in the Mosco sense, or Mosco-converges, to $W$ (written $W_j\xrightarrow M W$) if the following conditions hold:
\begin{itemize}
\item[(i)]  For  every $w\in W$ and $j\in \N$ there exists $w_j\in W_j$ such that $\|w_j- w\|_H\to 0$.
\item[(ii)] If $(W_{j_m})$ is a subsequence of $(W_j)$, $w_m\in W_{j_m}$, for $m\in \N$, and  $w_m\rightharpoonup w$ as $m\to\infty$, then $w\in W$.
\end{itemize}
\end{defn}
Two simple cases in which Mosco convergence holds \cite[Lems.~1.2 and 1.3]{Mosco69} %
are %
\begin{align} \label{eq:increase}
W_1\subset W_2 \subset \cdots, \quad &\mbox{with} \quad W = \overline{\bigcup_{j=1}^\infty W_j}, \quad\text{and}
\\
\label{eq:decrease}
W_1\supset W_2 \supset \cdots \quad &\mbox{with} \quad W = \bigcap_{j=1}^\infty W_j.
\end{align}
The following ``sandwich lemma'' is a trivial consequence of Definition \ref{def:mosco}.
\begin{lem}
\label{lem:Sandwich}
If $W^+_j$, $W^-_j$ and $W_j$ are closed subspaces of $H$ satisfying $W^-_j\subset W_j\subset W^+_j$ for each $j\in\N$, and both $W^+_j$ and $W^-_j$ Mosco-converge to some closed subspace $W$ of $H$, then $W_j$ also Mosco-converges to $W$.
\end{lem}

The following lemma will also be useful. Again the proof is straightforward.
\begin{lem}
\label{lem:MoscoSuff}
Let $W_j$ and $W$ be closed subspaces of $H$. To prove that $W_j \xrightarrow{M} W $ it suffices to show that (i) there exists a dense subspace $\widetilde{W}\subset W$ such that for every $w\in \widetilde{W}$ and $j\in \N$ there exists $w_j\in W_j$ such that $\|w_j- w\|_H\to 0$, and (ii) there exists a sequence of closed subspaces $W^+_j$ of $H$ such that $W_j\subset W^+_j$ for all $j\in\N$ and $W^+_j \xrightarrow{M} W $.
\end{lem}

Mosco convergence was introduced to study convergence of approximate solutions to variational inequalities. The following lemma, which applies to variational {\em equalities}, appears to be new, but its proof, if specialised to the coercive case, has something of the flavour of the original arguments of Mosco \cite[Thm.~A and its Corollary]{Mosco69}. Indeed, in the case that $H$ is a real Hilbert space and $a$ is coercive, this lemma is a corollary of the results in \cite{Mosco69}, since variational inequalities on linear subspaces of real Hilbert spaces are in fact equalities.

\begin{lem} \label{lem:dec3} Let $W\subset H$ and $W_j\subset H$, for $j\in \N$, be closed subspaces such that $W_j$ Mosco-converges to $W$ %
as $j\to\infty$. Let $a$ be compactly perturbed coercive and invertible on $W$. Then there exists $J\in \N$ such that $a$ is uniformly invertible on $W_j$ for $j\geq J$. Further, if, for some $f\in \cH$, $u_W$ denotes the solution to \eqref{eq:varW} and, for $j\geq J$, $u_{W_j}$ denotes the solution to \eqref{eq:varW} with $W$ replaced by $W_j$,
then $\|u_{W_j}-u_W\|_H\to 0$ as $j\to\infty$.
\end{lem}
\begin{proof}
We show first that, for some $J\in \N$, $a$ is uniformly invertible on $W_j$ for $j\geq J$.
Suppose first that $a$ is not invertible on $W_j$ for all sufficiently large $j$, in which case neither is it injective on $W_j$. Then there exists a subsequence of $(W_j)$, which we will denote again by $(W_j)$, and $v_j\in W_j$ with $\|v_j\|_H=1$, such that
\begin{equation} \label{eqA}
a(v_j,v) = 0, \quad v\in W_j.
\end{equation}
If on the other hand $a$ is invertible on $W_j$ for all sufficiently large $j$ but is not uniformly invertible then there exists a subsequence of $(W_j)$, which we will denote again by $(W_j)$, and $v_j\in W_j$ with $\|v_j\|_H=1$ such that
\begin{equation} \label{eqB}
\sup_{v\in W_j,\, \|v\|_H=1}|a(v_j,v)|%
\to 0\mbox{ as } j\to\infty.
\end{equation}
In both of these cases as $\|v_j\|_H=1$ is bounded we can extract a subsequence that is weakly convergent to some $v\in H$. Denoting the subsequence again by $(v_j)$, we have that $v_j\in W_j$ and $v\in W$ by (ii) in Definition \ref{def:mosco}. Further, by (i) in Definition \ref{def:mosco}, for all $w\in W$, there exists a sequence $(w_j)\subset H$ with $w_j\in W_j$ such that $\|w_j-w\|_H\to 0$. Thus, and by \eqref{eqA} or \eqref{eqB},
$$
a(v_j,w) = a(v_j,w_j) + a(v_j,w-w_j) \to 0 \mbox{ as } j\to\infty.
$$
But also $a(v_j,w)\to a(v,w)$. %
Thus $a(v,w)=0$ for all $w\in W$, so that $v=0$ as $a$ is invertible on $W$. So $v_j\rightharpoonup v=0$ and, by \eqref{eqA} or \eqref{eqB}, $a(v_j,v_j)\to 0$. Further,  recalling that $a=a_0+a_1$ with $a_0$ coercive and $a_1$ compact, we have also $a_1(v_j,v_j)\to 0$ as $a_1$ is compact. But this implies that $a_0(v_j,v_j)\to 0$, which contradicts that $a_0$ is coercive. Thus, for some $J\in \N$, $a$ is invertible on $W_j$ for $j\geq J$ and is uniformly bounded.

Thus, for $j\geq J$, $u_{W_j}$ is well-defined and $(u_{W_j})_{j=J}^\infty$ is bounded and so has a weakly convergent subsequence, converging to a limit $u_*$, and $u_*\in W$ by (ii) in Definition \ref{def:mosco}. Further, by (i) in Definition \ref{def:mosco}, for all $w\in W$ there exists $w_j\in W_j$ such that $\|w_j-w\|_H\to 0$, and \eqref{eq:varW} gives
$$
a(u_W,w)=\langle f, w\rangle = \langle f, w-w_j\rangle + \langle f,w_j\rangle = \langle f, w-w_j\rangle + a(u_{W_j},w_j) \to a(u_*,w),
$$
as $j\to\infty$ through that subsequence. Thus $a(u_W,w)=a(u_*,w)$, for all $w\in W$, so that $u_*=u_W$ by the invertibility of $a$ on $W$. By the same argument every subsequence of $(u_{W_j})_{j=J}^\infty$ has a subsequence converging weakly to $u_W$, so that $(u_{W_j})_{j=J}^\infty$ converges weakly to $u_W$. Finally, we see that
\begin{eqnarray*}
a(u_{W_j}-u_W,u_{W_j}-u_W)
 &=& \langle f,u_{W_j}\rangle - a(u_{W_j},u_W)-a(u_W,u_{W_j}-u_W)\to 0
\end{eqnarray*}
as $j\to \infty$, by the weak convergence of $(u_{W_j})_{j=J}^\infty$ and \eqref{eq:varW}. Since $a_1$ is compact, $a_1(u_{W_j}-u_W,u_{W_j}-u_W)\to 0$, so that also $a_0(u_{W_j}-u_W,u_{W_j}-u_W)\to 0$. Since $a_0$ is coercive it follows that $u_{W_j}\to u_W$.
\end{proof}

\begin{rem} \label{rem:coer}
The statement of Lemma \ref{lem:dec3} can be strengthened if additional assumptions are made on $a$, $(W_j)_{j=1}^\infty$ and $W$.

\begin{enumerate}[(i)]
\item If $a$ is coercive then
one can take
$J=1$ (since a coercive sesquilinear form is automatically invertible on every subspace).
In the special case when \eqref{eq:decrease} holds, this was noted already in \cite[Lem.~2.4]{ChaHewMoi:13}.

\item If $W_j\subset W$ for each $j\in\N$ %
then (ii) in Definition \ref{def:mosco} holds automatically and %
{\em quasi-optimality} holds asymptotically \cite[Thms.~8.10-11]{Steinbach}, meaning that, for some $M>0$ and $J\in \N$, $a$ is invertible on $W_j$ for $j\geq J$, and
\begin{equation} \label{eq:quasiopt}
\|u_{W_j}-u_W\|_H \leq M \inf_{w_j\in W_j}\|u_W-w_j\|_H, \quad \mbox{for } j\geq J.
\end{equation}
Furthermore, if $a$ is coercive then, by C\'ea's lemma, \rf{eq:quasiopt} holds with $J=1$ and $M=C/\alpha$, where $C$ and $\alpha$ are the continuity and coercivity constants for $a$.
\end{enumerate}
\end{rem}

The following lemma \cite[Theorem 2]{Ch:20} (and see \cite[Proposition 9.4]{Daners:03}) makes clear that $W_j\xrightarrow{M} W$ is necessary for the convergence in Lemma \ref{lem:dec3} to hold.
\begin{lem}
\label{lem:Mosco} Suppose that $W\subset H$ and $W_j\subset H$, for $j\in \N$, are closed subspaces, and that $a$ is invertible on $W$ and, for some $J\in \N$, on $W_j$ for $j\geq J$.  Suppose also that $\|u_{W_j}-u_W\|_H\to 0$ as $j\to\infty$, for every $f\in \cH$, where $u_W$ and $u_{W_j}$ are as defined in Lemma \ref{lem:dec3}. Then $W_j\xrightarrow{M} W$ as $j\to\infty$.
\end{lem}

\subsection{Sobolev spaces and trace operators} \label{sec:SobolevSpaces}

Our notation follows that of \cite{McLean} and \cite{ChaHewMoi:13}.
Let $m\in\N$.
For a subset $E\subset\R^m$ we denote its complement $E^c:=\R^m\setminus E$, its closure $\overline{E}$ and its interior $E^\circ$.
We denote the Hausdorff (fractal) dimension of $E$ by $\dimH{E}$ (see, e.g.,\ \cite[\S I.2]{Triebel97FracSpec}).
We say that a non-empty closed set $F\subset\R^m$ is a $d$-set for some $0\leq d\leq m$,
if it is ``uniformly $d$-dimensional'',
more precisely if there exist $c_{1},c_{2}>0$
such that
\[
c_{1}r^{d}\leq\mathcal{H}^{d}\big(B_{r}(\bx)\cap F\big)\leq c_{2}r^{d},\quad \bx\in F,\,\,0<r\leq1,
\]
where $B_{r}(\bx)$ is the closed ball of radius $r$ with centre at $\bx$ and $\mathcal{H}^{d}$ denotes the $d$-dimensional Hausdorff measure on $\R^{m}$ \cite[\S I.3]{Triebel97FracSpec}.
We say that a non-empty open set $\Omega\subset \R^m$ is $C^0$ (respectively Lipschitz) if its boundary $\partial\Omega$ can at each point be locally represented as the graph (suitably rotated) of a $C^0$ (respectively Lipschitz) function from $\R^{m-1}$ to $\R$, with $\Omega$ lying only on one side of $\partial\Omega$.
For a more detailed definition see, e.g., \cite[Defn~1.2.1.1]{Gri}.
For $m=1$ these definitions coincide: we interpret them both to mean that $\Omega$ is a countable union of open intervals whose closures are disjoint and whose endpoints have no limit points.

For $s\in \R$, let $H^s(\R^m)$ denote the Hilbert space of tempered distributions
whose Fourier transforms (defined for $\bxi\in\R^m$ as $\hat{u}(\bxi):= \frac{1}{(2\pi)^{m/2}}\int_{\R^m}\re^{-\ri \bxi\cdot \bx}u(\bx)\,\rd \bx$ in the case that $u\in C_0^\infty(\R^m)$) are locally integrable with
\[\|u\|_{H^s(\R^m)}^2:=\int_{\R^m}(1+|\bxi|^2)^{s}\,|\hat{u}(\bxi)|^2\,\rd \bxi < \infty.\]
In particular,
$H^0(\R^m)=L^2(\R^m)$ with equal norms. For the dual space of $H^s(\R^m)$ we have the unitary realisation $(H^s(\R^m))^*\cong H^{-s}(\R^m)$,
with duality pairing
\begin{align}\label{DualDef}
\left\langle u, v \right\rangle_{H^{-s}(\R^m)\times H^{s}(\R^m)}:= \int_{\R^m}\hat{u}(\bxi) \overline{\hat{v}(\bxi)}\,\rd \bxi,
\end{align}
which coincides with the $L^2(\R^m)$ inner product when both $u$ and $v$ are in $L^2(\R^m)$.

Given a closed set $F\subset \R^m$, we define
\begin{equation*} %
H_F^s := \{u\in H^s(\R^m): \supp{u} \subset F\},
\end{equation*}
and given a non-empty open set $\Omega\subset\R^m$, we %
define
\begin{equation*}%
\tH^s(\Omega):=\overline{C^\infty_0(\Omega)}^{H^s(\R^m)},
\end{equation*}
the closure of $C^\infty_0(\Omega)$ in $H^s(\R^m)$.
Clearly $\tH^s(\Omega)\subset H^s_{\overline\Omega}$, and when $\Omega$ is sufficiently regular
it holds that $\tH^s(\Omega)=H^s_{\overline\Omega}$ (see Proposition \ref{prop:TildeSubscript} for results relevant to the current study); however, in general these spaces can be different \cite[\S3.5]{ChaHewMoi:13}.
For non-empty open $\Omega$ we also define %
\begin{align*}
H^s(\Omega)&:=\{u=U|_\Omega \textrm{ for some }U\in H^s(\R^m)\},\\
\|u\|_{H^{s}(\Omega)}&:=\inf_{\substack{U\in H^s(\R^m)\\ U|_{\Omega}=u}}\normt{U}{H^{s}(\R^m)}.
\end{align*}
Although $H^s(\Omega)$ is a space of distributions on $\Omega$, it can be naturally identified with a space of distributions on $\R^m$, namely $(H^s_{\Omega^c})^\perp\subset H^s(\R^m)$, where $^\perp$ denotes orthogonal complement in $H^s(\R^m)$, with the restriction operator $|_\Omega :(H^s_{\Omega^c})^\perp\to H^s(\Omega)$ providing a unitary isomorphism between the two spaces. %

Regarding duality,
for arbitrary $F\subset\R^m$ closed and $\Omega\subset\R^m$ open, we can unitarily realise the dual spaces of $H^s_F$ and $\tH^s(\Omega)$ by certain closed subspaces of $H^{-s}(\R^m)$, with the duality pairing inherited from \eqref{DualDef}.  %
Precisely, by Lemma \ref{lem:hs_orth},
\begin{align} \label{dualrelns1}
(H^s_F)^*&\cong (\tH^{-s}(F^c))^\perp,\\
\label{dualrelns2}
(\tH^s(\Omega))^*&\cong (H^{-s}_{\Omega^c})^\perp,
\end{align}
where $^\perp$ denotes orthogonal complement in $H^{-s}(\R^m)$, since \cite[\S3.2]{ChaHewMoi:13} $\tH^{-s}(F^c)$ and $H^{-s}_{\Omega^c}$ are the annihilators of $H^s_F$ and $\tH^s(\Omega)$, respectively, with respect to the duality pairing \eqref{DualDef}.
An alternative, and more widely-known unitary realisation of $(\tH^s(\Omega))^*$ (also valid for arbitrary open $\Omega\subset\R^m$, see \cite[Thm.~3.3]{ChaHewMoi:13}) is
\begin{align}
\label{isdual}
(\tH^s(\Omega))^*  \cong H^{-s}(\Omega) \quad \mbox{with} \quad
\langle u,v \rangle_{H^{-s}(\Omega)\times \tH^s(\Omega)}:=\langle U,v \rangle_{H^{-s}(\R^m)\times H^s(\R^m)},
\end{align}
where $U\in H^{-s}(\R^m)$ is any extension of $u\in H^{-s}(\Omega)$ with $U|_\Omega=u$. That \rf{dualrelns2} and \rf{isdual} are both unitary realisations of $(\tH^s(\Omega))^*$ is consistent with the fact that $|_\Omega :(H^{-s}_{\Omega^c})^\perp\to H^{-s}(\Omega)$ is a unitary isomorphism, as mentioned above.

In the context of our screen scattering problem, we define Sobolev spaces on the hyperplane $\Gamma_\infty = \R^{n-1}\times\{0\}$ by associating $\Gamma_\infty$ with $\R^{n-1}$ and setting $H^s(\Gamma_\infty) := H^s(\R^{n-1})$, for $s\in\R$, which we shall frequently abbreviate to simply $H^s$.
For $E\subset\Gamma_\infty$ we set $\widetilde{E}:=\{\widetilde \bx\in\R^{n-1}: (\widetilde \bx,0)\in E\}\subset \R^{n-1}$.
Then for a closed subset $F\subset\Gamma_\infty$ we define $H^s_F:=H^s_{\widetilde{F}}\subset H^s$, and for a (relatively) open subset $\Omega\subset \Gamma_\infty$ we set
$\tH^{s}(\Omega):=\tH^{s}(\widetilde{\Omega})\subset H^s$
and $H^{s}(\Omega):=H^{s}(\widetilde{\Omega})$, etc.
We stress that all Sobolev spaces on subsets of $\Gamma_\infty$ such as $H^s_F$, $\tH^s\OO$ and $H^s\OO$ are defined starting from $H^s=H^s(\Gamma_\infty)=H^s(\R^{n-1})$, as opposed to $H^s\Rn$; in other words, in the definitions earlier in this section we have $m=n-1$.

In the exterior domain $D:=\R^n\setminus\overline{\Gamma}$ we work with Sobolev spaces defined via weak derivatives. %
Given a non-empty open $\Omega\subset \R^n$, let
$W^1(\Omega) := \{u\in L^2(\Omega): \nabla u \in L^2(\Omega)\}$
and let $W^{1,\mathrm{loc}}(\Omega)$ denote the ``local'' space in which square integrability of $u$ and $\nabla u$ is required only on bounded subsets of $\Omega$.
We note that, typically, $H^1(D)\subsetneqq W^1(D)$, since the restriction space inherits from $H^1\Rn$ a requirement of (weak) continuity across $\Gamma$.
We define
$U^+:=\{(x_1,\ldots,x_n)\in\R^n:x_n>0\}$ and $U^-:=\R^n\setminus\overline{U^+}$, and adopt the convention that the unit normal vector $\bn$ on $\Gamma_\infty$ points into $U^+$.
From the half spaces $U^\pm$ to the hyperplane $\Gamma_\infty$ we define the standard trace operators $\gamma^\pm:W^1(U^\pm)\to H^{1/2}= H^{1/2}(\Gamma_\infty)$ and $\partial_\bn^\pm:
\{u\in W^1(U^\pm): \Delta u \in L^2(U^\pm)\}
\to H^{-1/2}=H^{-1/2}(\Gamma_\infty)$.
We shall frequently abuse notation and apply $\gamma^\pm$ and $\partial_\bn^\pm $ to elements $u$ of the local space $W^{1,\mathrm{loc}}(D)$, assuming implicitly that $u$ has been pre-multiplied by a cutoff $\phi \in C^\infty_0(\R^n)$ satisfying $\phi=1$ in some neighbourhood of $\overline\Gamma$, and restricted to $U^\pm$ as appropriate; for example, $\gamma^+u$ should be interpreted as $\gamma^+((u\phi)|_{U^+})$.

\section{Boundary value problems and boundary integral equations}
\label{sec:BVPs}

Given a screen $\Gamma\subset\Gamma_\infty$ (a bounded subset of $\Gamma_\infty$), and an incident wave $u^i\in H^{1,\mathrm{loc}}(\R^n)$ (for instance, a plane wave $u^i(\bx):=\re^{\ri k \bd\cdot \bx}$ with $\bd$ a unit direction vector), we seek a scattered acoustic field $u$
satisfying
the Helmholtz equation %
\begin{align}
\label{eqn:HE}
\Delta u + k^2 u = 0, \qquad k>0,
\end{align}
in $D=\R^n\setminus \omS$,
the Sommerfeld radiation condition
\begin{align}
\label{eqn:SRC}
\pdone{u(\bx)}{r} - \ri k u(\bx) = o(r^{(1-n)/2}), \qquad r:=|\bx|\to\infty, \text{ uniformly in } \hat\bx:=\bx/|\bx|,
\end{align}
and the Dirichlet boundary condition
\begin{align}
\label{eqn:DirBC}
u=-u^i \qquad \text{on } \Gamma.
\end{align}
To formulate a well-posed BVP, one needs to be more precise about the sense in which the boundary condition \rf{eqn:DirBC} holds. A detailed investigation into this issue for general bounded subsets $\Gamma\subset \Gamma_\infty$ was carried out in \cite{ScreenPaper}.
Here we apply the results of \cite{ScreenPaper} to describe well-posed BVP formulations for the two types of screen mentioned in the Introduction, namely (i) bounded, relatively open subsets of $\Gamma_\infty$, and (ii) compact subsets of $\Gamma_\infty$, possibly with empty relative interior.
From now on, for brevity we shall omit the words ``relative'' and ``relatively'' when discussing relatively open subsets of $\Gamma_\infty$ and relative complements, boundaries and interiors of subsets of $\Gamma_\infty$.

\subsection{Well-posed BVPs and BIEs for bounded open screens}
\label{sec:Open}
Let $\Gamma$ be a non-empty bounded open subset of $\Gamma_\infty$. Then we can formulate the scattering BVP by imposing the Dirichlet boundary condition \rf{eqn:DirBC} by restriction to $\Gamma$ (denoting this problem as $\sD(\Gamma)^\mathrm{r}$, $\sD$ for Dirichlet, $\mathrm{r}$ for restriction).

\begin{defn}[Problem $\sD(\Gamma)^{\mathrm{r}}$ for bounded open screens]
\label{def:OpenPrime}
Let $\Gamma\subset\Gamma_\infty$ be non-empty, bounded and open. Given $g^{\mathrm{r}}\in H^{1/2}(\Gamma)$ (specifically $g^{\mathrm{r}}:=-(\gamma^\pm u^i)|_{\Gamma}$ for the scattering problem), find $u\in C^2\left(D\right)\cap  W^{1,\mathrm{loc}}(D)$ satisfying \rf{eqn:HE} in $D$, \rf{eqn:SRC}, and %
\begin{align*}%
(\gamma^\pm u)|_{\Gamma}&=g^{\mathrm{r}}.
\end{align*}
\end{defn}

A well-posedness result for this formulation is provided in Theorem \ref{thm:OpenPrime}, which is proved in \cite[Thm.~6.2(a)]{ScreenPaper}.
Before stating it we need some more notation. For $\Gamma\subset\Gamma_\infty$ non-empty, bounded and open let $\cS_{\Gamma}:\tH^{-1/2}(\Gamma)\to C^2(D)\cap W^{1,{\rm loc}}(\R^n)$ denote the standard single-layer potential, defined for $\phi\in C^\infty_0(\Gamma)$ by
\begin{align}\label{SLP}
\cS_{\Gamma}\phi(\bx)=\int_{\Gamma}\Phi(\bx,\by)\phi(\by)\,\rd s(\by), \qquad \bx\in D,
\end{align}
with $\Phi(\bx,\by)=\re^{\ri k |\bx-\by|}/(4\pi |\bx-\by|)$ ($n=3$) or $\Phi(\bx,\by)=(\ri/4)H^{(1)}_0(k|\bx-\by|)$ ($n=2$), %
and $S^{\mathrm r}_{\Gamma}:\tH^{-1/2}(\Gamma)\to H^{1/2}(\Gamma)$ the single layer boundary integral operator (BIO), the bounded linear operator
defined by
\begin{align}\label{SLBIO}
S^{\mathrm r}_{\Gamma}\phi:=(\gamma^\pm\cS_{\Gamma}\phi)|_{\Gamma}.
\end{align}

\begin{thm}[\!\!{\cite[Thm.~6.2(a)]{ScreenPaper}, \cite[Lem.~4.15(ii)]{caetano2018}}]
\label{thm:OpenPrime}
Let $\Gamma\subset\!\Gamma_\infty$ be non-empty, bounded and open, with %
$\tH^{-1/2}(\Gamma)=H^{-1/2}_{\overline\Gamma}$.
Then problem $\sD(\Gamma)^{\mathrm{r}}$
has a unique solution. Moreover, it satisfies the representation formula
\begin{align}
\label{eqn:Rep}
u(\bx )= -\cS_{\Gamma}\phi(\bx), \qquad \bx\in D,
\end{align}
where $\phi = \partial_\bn^+u-\partial_\bn^-u \in \tH^{-1/2}(\Gamma)$ is the unique solution of the BIE
\begin{align}
\label{eqn:SBIEPrime}
S^{\mathrm r}_{\Gamma}\phi = -g^{\mathrm{r}}.
\end{align}
\end{thm}

\begin{rem}
\label{rem:TilSubImpliesNull}
The statement of \cite[Thm.~6.2(a)]{ScreenPaper} includes an extra assumption that $H^{1/2}_{\partial\Gamma}=\{0\}$, where $\partial\Gamma$ denotes the relative boundary of $\Gamma\subset\Gamma_\infty$. But
this extra assumption is superfluous, since by \cite[Lem.~4.15(ii)]{caetano2018} it follows automatically from the assumption that $\tH^{-1/2}(\Gamma)=H^{-1/2}_{\overline\Gamma}$.
\end{rem}

The key condition for the well-posedness in Theorem~\ref{thm:OpenPrime} is $\tH^{-1/2}(\Gamma)=H^{-1/2}_{\overline\Gamma}$; the next proposition gives sufficient conditions on $\Gamma$ for this to hold.

\begin{prop}
\label{prop:TildeSubscript}
Each of the following are sufficient %
for $\tH^{-1/2}(\Gamma)=H^{-1/2}_{\overline\Gamma}$:
\begin{itemize}
\item[(i)] $\Gamma$ is $C^0$ (which holds in particular if $\Gamma$ is Lipschitz) \cite[Thm.~3.29]{McLean};
\item[(ii)] $\Gamma$ is $C^0$ except at a set of countably many points $P\subset\partial\Gamma$ such that $P$ has only finitely many limit points %
\cite[Thm.~3.24]{ChaHewMoi:13};
\item[(iii)] $|\partial\Gamma|=0$, where $|\cdot|$ denotes the Lebesgue measure on $\Gamma_\infty\cong \R^{n-1}$, and $\Gamma$ is ``thick'' in the sense of Triebel (\!\!{\cite[Def.~4.5(iii)]{caetano2018}} or \cite[Def.~3.1(ii)–(iv), Rem.~3.2]{Tri08}).
\end{itemize}
\end{prop}

In \S\ref{sec:examples} we shall combine Theorem \ref{thm:OpenPrime} and Proposition \ref{prop:TildeSubscript} to obtain well-posed\-ness results for three-dimensional scattering by certain generalisations of the classical Koch snowflake, as immediate corollaries of the recently established thickness results in \cite{caetano2018} (which build on earlier results in \cite[Prop.~3.8(iii)]{Tri08}).
We shall also deduce well-posedness results for scattering by the standard prefractal approximations to various well-known fractals including the Koch snowflake (and its generalisations), the Sierpinski triangle, and the Cantor dust. In all these cases the standard prefractals are either $C^0$ or $C^0$ except at a finite set of points.

Recalling from \S\ref{sec:SobolevSpaces} that $|_\Gamma :(H^{1/2}_{\Gamma^c})^\perp\to H^{1/2}(\Gamma)$ is a unitary isomorphism, we note that
the problem $\sD(\Gamma)^{\mathrm{r}}$ can be equivalently stated with the boundary condition \rf{eqn:DirBC} imposed by orthogonal projection. (See Remark \ref{rem:phys} below for an explanation of why this makes sense physically.)
It is instructive to write down this equivalent formulation explicitly, since we will adopt a similar viewpoint when defining BVPs for scattering by compact screens in \S\ref{sec:Closed}. In the following let
$P_\Gamma:H^{1/2}\to(H^{1/2}_{\Gamma^c})^\perp$
denote %
orthogonal projection %
and define  $S_\Gamma:\tH^{-1/2}(\Gamma)\to (H^{1/2}_{\Gamma^c})^\perp$ by $S_\Gamma := P_\Gamma\gamma^\pm \cS_\Gamma$.
\begin{defn}[Problem $\sD(\Gamma)$ for bounded open screens]
\label{def:Open}
Let $\Gamma\subset\Gamma_\infty$ be non-empty, bounded and open.
Given $g\in (H^{1/2}_{\Gamma^c})^\perp$
(specifically $g:=-P_\Gamma \gamma^\pm u^i$  for the scattering problem), find $u\in C^2\left(D\right)\cap  W^{1,\mathrm{loc}}(D)$ satisfying
\rf{eqn:HE} in $D$, \rf{eqn:SRC}, and the boundary condition
\begin{align}\label{BC:Open}
P_\Gamma\gamma^\pm u&=g.
\end{align}
\end{defn}

\begin{rem} \label{rem:phys} To understand why the boundary condition \eqref{BC:Open} makes sense in the scattering problem,
let $u^t:= u+u^i$ be the total field (the sum of the scattered and incident fields), and consider the traces $\gamma^\pm u^t$ of $u^t$ on $\Gamma_\infty \supset \Gamma$. According to formulation $\sD(\Gamma)^{\mathrm{r}}$ these traces vanish on $\Gamma$, so their supports lie in the complement $\Gamma^c$, i.e.\ $\gamma^\pm u^t \in H^{1/2}_{\Gamma^c}$ (more precisely $\gamma^\pm (\chi u^t) \in H^{1/2}_{\Gamma^c}$ for every $\chi\in C_0^\infty(\R^n)$ with $\chi=1$ in a neighbourhood of $\overline{\Gamma}$). But since the kernel of $P_\Gamma$ is $H^{1/2}_{\Gamma^c}$, this is equivalent to $P_\Gamma \gamma^\pm u^t = 0$ (more precisely  $P_\Gamma \gamma^\pm (\chi u^t) = 0$), which is just \eqref{BC:Open} with $g:=-P_\Gamma \gamma^\pm u^i$.
\end{rem}

Since the restriction operator $|_\Gamma: (H^{-1/2}_{\Gamma^c})^\perp\to H^{-1/2}(\Gamma)$ is unitary, the following proposition is a restatement of Theorem \ref{thm:OpenPrime}.
\begin{prop}
\label{prop:equiv}
Problems $\sD(\Gamma)^{\mathrm{r}}$ and $\sD(\Gamma)$ are equivalent, under the identification $g^{\mathrm{r}}= g|_\Gamma$.
Furthermore, when %
$\tH^{-1/2}(\Gamma)=H^{-1/2}_{\overline\Gamma}$ the common unique solution of both problems can be represented as \rf{eqn:Rep} where $\phi = \partial_\bn^+u-\partial_\bn^-u \in \tH^{-1/2}(\Gamma)$ is the common unique solution of the BIE \rf{eqn:SBIEPrime} and the BIE
\begin{align}
\label{eqn:SBIE}
S_{\Gamma}\phi = -g.
\end{align}
\end{prop}

\subsection{Well-posed BVPs and BIEs for compact screens}
\label{sec:Closed}
Now let $\Gamma\subset\Gamma_\infty$ be compact. In particular we have in mind the case where $\Gamma$ has empty interior, in which case it is not possible to impose the boundary condition \rf{eqn:DirBC} by restriction. %
But, inspired by Definition \ref{def:Open} and Proposition \ref{prop:equiv}, we can impose \rf{eqn:DirBC} by an appropriate orthogonal projection; we justify this in Remark \ref{rem:phys2} below. Extending our existing notation, for compact $\Gamma$ let $P_\Gamma$ denote the orthogonal projection $P_\Gamma:H^{1/2}\to(\tH^{1/2}(\Gamma^c))^\perp$.
\begin{defn}[Problem $\sD(\Gamma)$ for compact screens] \label{def:Compact}
Let $\Gamma\subset\Gamma_\infty$ be non-empty and compact. Given $g\in (\tH^{1/2}(\Gamma^c))^\perp$ %
(specifically $g:=-P_\Gamma \gamma^\pm u^i$ for the scattering problem), find $u\in C^2\left(D\right)\cap  W^{1,\mathrm{loc}}(D)$ satisfying
\rf{eqn:HE} in $D$, \rf{eqn:SRC}, and %
\begin{align}\label{a1bc}
P_\Gamma\gamma^\pm u&=g.
\end{align}
\end{defn}

\begin{rem}\label{rem:phys2} We can justify the formulation in Definition \ref{def:Compact}, in particular
\eqref{a1bc}, by relating it to a more familiar formulation of the scattering problem for compact screens \cite[Def.~3.1]{ScreenPaper} that replaces \eqref{a1bc} with the requirement that $u^t:= u+u^i\in W_0^{1,\mathrm{loc}}(D)$, where $W_0^{1,\mathrm{loc}}(D)$ is the ``local'' version of $W_0^{1}(D)$, and $W_0^{1}(D)$ is the closure of $C_0^\infty(D)$ in $W^1(D)$. This formulation is well-posed \cite[Thm~3.1]{ScreenPaper}
when $\Gamma$ is any compact subset of $\R^n$. In the particular case when $\Gamma\subset \Gamma_\infty$ is a screen it is easy to see that $u^t\in W_0^{1,\mathrm{loc}}(D)$ implies that $\gamma^\pm u^t \in \tH^{1/2}(\Gamma^c)$ (more accurately $\gamma^\pm (\chi u^t) \in \tH^{1/2}(\Gamma^c)$ for every $\chi\in C_0^\infty(\R^n)$ with $\chi=1$ in a neighbourhood of $\Gamma$). %
But since the kernel of $P_\Gamma$ is $\tH^{1/2}(\Gamma^c)$, this is equivalent to $P_\Gamma \gamma^\pm u^t = 0$ (more precisely  $P_\Gamma \gamma^\pm (\chi u^t) = 0$), which is just \eqref{a1bc} with $g:=-P_\Gamma \gamma^\pm u^i$.
\end{rem}
Before stating a well-posedness result for this formulation (which we do in Theorem \ref{thm:Closed}), we need some more notation.
Given $\Gamma\subset\Gamma_\infty$ compact and $\epsilon>0$, let $\Gamma(\epsilon):=\{\bx\in\Gamma_\infty:\dist(\bx,\Gamma)<\epsilon\}$ %
and $D_\epsilon:=\R^n\setminus\overline{\Gamma(\epsilon)}$. Define $\cS_{\Gamma(\epsilon)}:\tH^{-1/2}(\Gamma(\epsilon))\to C^2(D_\epsilon)\cap W^{1,{\rm loc}}(\R^n)$
by \eqref{SLP}
with $\Gamma$ replaced by $\Gamma(\epsilon)$.
Define $\cS_{\Gamma}:H^{-1/2}_{\Gamma}\to C^2(D)\cap W^{1,{\rm loc}}(\R^n)$ by $\cS_{\Gamma}\phi(\bx):=\cS_{\Gamma(\epsilon)}\phi(\bx)$ for $\bx\in D$ and $0<\epsilon<\dist(\bx,\Gamma)$, which is well-defined and independent of $\epsilon>0$ since $H^{-1/2}_{\Gamma}\subset \tH^{-1/2}(\Gamma(\epsilon))$ for every $\epsilon>0$. %
Define $S_{\Gamma}:H^{-1/2}_{\Gamma}\to (\tH^{1/2}(\Gamma^c))^\perp$ by $S_{\Gamma}:=  P_\Gamma\gamma^\pm\cS_{\Gamma}$.
\begin{thm}[\!\!{\cite[Thm.~3.29 and Thm.~6.4]{ScreenPaper}}]
\label{thm:Closed}
Let $\Gamma\subset\Gamma_\infty$ be non-empty and compact.
Then problem $\sD(\Gamma)$
has a unique solution %
satisfying
the representation formula
\rf{eqn:Rep}, where $\phi = \partial_\bn^+u-\partial_\bn^-u \in H^{-1/2}_{\Gamma}$ is the unique solution of the BIE
\begin{equation*}%
S_\Gamma\phi = -g.
\end{equation*}
\end{thm}

\begin{rem}
\label{rem:equiv}
Suppose that $\Gamma\subset\Gamma_\infty$ is non-empty, bounded and open, in which case $\overline{\Gamma}$ is compact, and suppose also that %
$\widetilde{H}^{-1/2}(\Gamma) = H^{-1/2}_{\overline{\Gamma}}$. Then we have a choice of well-posed formulations for the scattering problem, potentially with different solutions: problem $\sD(\Gamma)$ (see Definition \ref{def:Open}) with $g:= -P_\Gamma \gamma^\pm u^i$ (equivalently, $\sD(\Gamma)^{\mathrm{r}}$ with $g^{\mathrm{r}}:= -(\gamma^\pm u^i)|_\Gamma$) and problem $\sD(\overline\Gamma)$ (see Definition \ref{def:Compact}) with $g:= -P_{\overline \Gamma} \gamma^\pm u^i$.
But the assumption that $\widetilde{H}^{-1/2}(\Gamma)= H^{-1/2}_{\overline{\Gamma}}$ %
implies (in fact, is equivalent to) $\tH^{1/2}(\overline\Gamma^c) = H_{\Gamma^c}^{1/2}$ (see \cite[Lem.~3.26]{ChaHewMoi:13} and the proof of \cite[Lem.~4.15]{caetano2018}),
and from this it follows that $\cS_{\overline{\Gamma}}=\cS_{\Gamma}$, $P_{\overline\Gamma}=P_\Gamma$, and $S_{\overline{\Gamma}}=S_{\Gamma} = |_{\Gamma}^{-1}\circ S_{\Gamma}^{\mathrm{r}}$. %
 So the two problems $\sD(\Gamma)$ and $\sD(\overline\Gamma)$ are equivalent, sharing the same unique solution.
\end{rem}

It is natural to ask whether, in the case where $\Gamma\subset\Gamma_\infty$ is compact with empty interior, the screen scatters waves at all. This question was answered in \cite{ScreenPaper}.
\begin{prop}[\!\!{\cite[Thm.~4.6]{ScreenPaper}}]
\label{prop:Scatters}\quad
Let $\Gamma\subset\Gamma_\infty$ be non-empty and compact.
If $H^{-1/2}_\Gamma= \{0\}$ then the solution of $\sD(\Gamma)$ satisfies $u=0$ for all $g\in (\tH^{1/2}(\Gamma^c))^\perp$ (and so for all incident waves $u^i$). If $H^{-1/2}_\Gamma\neq \{0\}$ and $0\neq g\in (\tH^{1/2}(\Gamma^c))^\perp$ (in particular, if $g=-P_\Gamma \gamma^\pm u^i$ and $u^i$ is $C^\infty$ in a neighbourhood of $\Gamma$ with $u^i(\bx)\neq 0$ for all $\bx\in\Gamma$) then $u\neq 0$.
\end{prop}

The question of whether $H^s_K=\{0\}$ for given $s\in\R$ and compact $K\subset\R^m$ was investigated in detail in \cite{HewMoi:15}. For sets of Lebesgue measure zero, the Hausdorff dimension $\dimH{K}$ provides a partial characterisation. Specifically, if $s>(\dimH{K}-m)/2$ then $H^s_K=\{0\}$, and if $s<(\dimH{K}-m)/2$ then $H^s_K\neq\{0\}$. For $s=(\dimH{K}-m)/2$ both behaviours are possible. But if $K$ is a $d$-set for some $0\leq d<m$
(see \S\ref{sec:SobolevSpaces} for definition)
then $H^{(d-m)/2}_K=\{0\}$. For such $d$-sets, the question of whether $H^{t}_K$ is dense in $H^{s}_K$ for $s<t<(d-m)/2$ was investigated in \cite{caetano2018}. The following lemma collects the results from \cite{HewMoi:15,caetano2018} relevant to our current purposes, translated to our spaces $H^s_\Gamma$, where $\Gamma\subset\Gamma_\infty=\R^{n-1}\times\{0\}$.

\begin{lem}[\!\!{\cite[Thms.~2.12, 2.17, Cor.~2.16]{HewMoi:15},
\cite[Prop.~3.7(i), Thm.~6.14]{caetano2018}}] \label{lem:zero}
\label{lem:Density}
{\color{white} .}\\ %
Let $\Gamma\subset\Gamma_\infty=\R^{n-1}\times\{0\}$ ($n=2,3$) be compact. If $d:=\dimH(\Gamma)>n-2$ then $H^{t}_\Gamma\neq\{0\}$ for $-1/2\leq t< (d-n+1)/2$. If $\Gamma$ is countable, or if $n=3$ and $\dimH(\Gamma)<n-2$ or $\Gamma$ is an $(n-2)$-set, then $H^{-1/2}_\Gamma=\{0\}$.
Furthermore, if $\Gamma$ is a compact $d$-set for some $n-2<d<n-1$ then $H^t_\Gamma\neq \{0\}$ is dense in $H^{-1/2}_\Gamma\neq \{0\}$ for  $-1/2\leq t< (d-n+1)/2$. %
\end{lem}

\subsection{Variational formulations of the BIEs} \label{sec:varf}
To state and analyse Galerkin methods for the BIE formulations, in particular to use the Mosco convergence theory of \S\ref{sec:Mosco}, we need
variational formulations.

When $\Gamma$ is a bounded open set
and $\widetilde H^{-1/2}(\Gamma) = H^{-1/2}_{\overline \Gamma}$, or when $\Gamma$ is compact,
we have written down, in Proposition \ref{prop:equiv} and Theorem \ref{thm:Closed}, BIEs that are equivalent to the corresponding BVP formulation $\sD(\Gamma)$.
In each case these take the form \eqref{eqn:SBIE}, i.e.\ $S_\Gamma\phi=-g$.
In this equation $\phi \in V(\Gamma)$, where
\begin{equation}\label{eq:VG}
V(\Gamma)=\begin{cases}
\tH^{-1/2}(\Gamma) & \text{if }\Gamma \text{ is bounded and open,}\\
H^{-1/2}_\Gamma & \text{if }\Gamma \text{ is compact,}
\end{cases}
\end{equation}
and $g\in V(\Gamma)^*$, where $V(\Gamma)^*$ is a unitary realisation of the dual space of $V(\Gamma)$, specifically $V(\Gamma)^* = (H^{1/2}_{\Gamma^c})^\perp$ or $V(\Gamma)^* = (\tH^{1/2}(\Gamma^c))^\perp$, in the respective cases. In each case $S_\Gamma:V(\Gamma)\to V(\Gamma)^*$ is a bounded linear operator, a version of the single-layer potential operator.

As noted in \S\ref{sec:var}, $S_\Gamma$ has an associated sesquilinear form $a_\Gamma(\cdot,\cdot)$ defined by
\begin{equation*} %
a_\Gamma(\phi,\psi) := \langle S_\Gamma \phi,\psi\rangle_{V(\Gamma)^*\times V(\Gamma)}, \quad \phi,\psi\in V(\Gamma).
\end{equation*}
Further (see \S\ref{sec:SobolevSpaces}), the duality pairing on $V(\Gamma)^*\times V(\Gamma)$ is simply the restriction to  $V(\Gamma)^*\times V(\Gamma)$ of the duality pairing \eqref{DualDef} on $H^{1/2}\times H^{-1/2}$.
Thus, for $\phi,\psi\in V(\Gamma)$,
\begin{equation} \label{eq:agendef2}
a_\Gamma(\phi,\psi) =  \langle S_\Gamma\phi,\psi\rangle_{H^{1/2}\times H^{-1/2}}=\langle \gamma^\pm \cS_\Gamma\phi,\psi\rangle_{H^{1/2}\times H^{-1/2}}.
\end{equation}
The second equality in \eqref{eq:agendef2} holds since $S_\Gamma = P_\Gamma \gamma^\pm \cS_\Gamma$, where $P_\Gamma$ is orthogonal projection onto $V(\Gamma)^*$, and since $(V(\Gamma)^*)^\perp$ is the annihilator of $V(\Gamma)$, as noted below \eqref{dualrelns2}.

A consequence of \eqref{eq:agendef2} and the definition of $\cS_\Gamma$ is that, if $\Gamma_\dag\subset \Gamma_\infty$ is any~bound\-ed open set containing $\Gamma$, in which case $V(\Gamma)$ is a closed subspace of $\tH^{-1/2}(\Gamma_\dag)$, then
\begin{equation} \label{eq:agendef3}
a_\Gamma(\phi,\psi) =  \langle \gamma^\pm \cS_{\Gamma_\dag}\phi,\psi\rangle_{H^{1/2}\times H^{-1/2}}= a_{\Gamma_\dag}(\phi,\psi), \quad \phi,\psi\in V(\Gamma),
\end{equation}
i.e. $a_\Gamma(\cdot,\cdot)$ is the restriction to $V(\Gamma)$ of $a_{\Gamma_\dag}(\cdot,\cdot)$. Since we can choose $\Gamma_\dag$ to be as smooth as we wish, e.g. $C^\infty$, or indeed just an open disk, we see that, even when $\Gamma$ is fractal or has fractal boundary, the sesquilinear forms we have to deal with are no more complicated than in the case when $\Gamma$ is a disk.

A further consequence of \eqref{eq:agendef2} is that, in the case when $\Gamma$ is bounded and open,
\begin{equation} \label{eq:agendef4}
a_\Gamma(\phi,\psi) =  \langle (\gamma^\pm \cS_{\Gamma}\phi)|_\Gamma,\psi\rangle_{H^{1/2}(\Gamma)\times \tH^{-1/2}(\Gamma)}, \quad \phi,\psi\in V(\Gamma)=\tH^{-1/2}(\Gamma),
\end{equation}
so that, where $S_\Gamma^{\mathrm{r}}$ is the more familiar screen single-layer potential operator defined by \eqref{SLBIO}, $a_\Gamma(\cdot,\cdot)$ is also the sesquilinear form associated to $S_\Gamma^{\mathrm{r}}$.

Whether $\Gamma$ is bounded and open or is compact, the sesquilinear form $a_\Gamma(\cdot,\cdot)$ is continuous. Further, as a consequence of our assumption that the screen is flat (i.e.~$\Gamma\subset\Gamma_\infty=\R^{n-1}\times\{0\}$), it is coercive.
For the bounded open case this is hinted at in \cite[Rem.~6]{Ha-Du:90} and proved rigorously in \cite{CoercScreen2}, the latter reference also detailing the wavenumber-dependence of the continuity and coercivity constants.
That coercivity of $a_\Gamma(\cdot,\cdot)$ holds also for every compact $\Gamma$ is a simple consequence of \eqref{eq:agendef3}, since coercivity implies coercivity on every closed subspace.

As noted in the general Hilbert space setting in \S\ref{sec:var} (see \eqref{eq:varW} and \eqref{eq:proj}),  the variational problem: given $g\in V(\Gamma)^*$
\begin{align}
\label{eqn:variational}
\text{find }\phi\in V(\Gamma) \text{ s.t. }a_\Gamma(\phi,\psi)=-\langle g,\psi\rangle_{V(\Gamma)^*\times V(\Gamma)},\; \;\;\forall \psi\in V(\Gamma),
\end{align}
is equivalent to the BIE $S_\Gamma\phi=-g$. The duality pairing on the right hand side can be written equivalently as
$$
\langle g,\psi\rangle_{V(\Gamma)^*\times V(\Gamma)}=\langle g,\psi\rangle_{H^{1/2}\times H^{-1/2}}.
$$
Since $a_\Gamma(\cdot,\cdot)$ is coercive, \eqref{eqn:variational} is well-posed by the Lax--Milgram lemma.

\section{Prefractal to fractal convergence}
\label{sec:Conv}

Now suppose we want to study a sequence of problems on a sequence of screens $(\Gamma_j)_{j\in\N_0}$ (each non-empty and either bounded and open or compact) approximating a limiting screen $\Gamma$ (again, non-empty and either bounded and open or compact). Assuming that the $\Gamma_j$ are uniformly bounded, let $\Gamma_\dag\subset\Gamma_\infty$ be a bounded
open set (e.g.\ a disk) such that $\Gamma\subset \Gamma_\dag$ and $\Gamma_j\subset \Gamma_\dag$ for each $j\in\N_0$.
Let
$g_\dag\in (H^{1/2}_{\Gamma_\dag^c})^\perp\subset H^{1/2}$ be fixed.
Then by the continuity and coercivity of the sesquilinear form $a_{\Gamma_\dag}(\cdot,\cdot)$ (discussed in \S\ref{sec:varf}), for any closed subspaces $V_j$ and $V$ of $V(\Gamma_\dag)=\tH^{-1/2}(\Gamma_\dag)$ the variational problems
\begin{align}
\label{eqn:variationalSTAR}
\text{find }\phi\in V \text{ s.t. }a_{\Gamma_\dag}(\phi,\psi)=-\langle g_\dag,\psi \rangle_{H^{1/2}\times H^{-1/2}},\; \;\;\forall \psi\in V,\\
\label{eqn:variationalSTARj}
\text{find }\phi_j\in V_j \text{ s.t. }a_{\Gamma_\dag}(\phi_j,\psi)=-\langle g_\dag,\psi \rangle_{H^{1/2}\times H^{-1/2} },\; \;\;\forall \psi\in V_j,
\end{align}
are well-posed.
The following theorem follows from Lemma \ref{lem:dec3}, combined with the continuity of $\cS_{\Gamma_\dag}:\tH^{-1/2}(\Gamma_\dag)\to W^{1,{\rm loc}}(\R^n)$ (defined by \eqref{SLP}).

\begin{thm}\label{thm:MoscoConv}%
Let $V_j$ and $V$ be closed subspaces of $V(\Gamma_\dag)=\tH^{-1/2}(\Gamma_\dag)$ and let $\phi\in V$ and $\phi_j\in V_j$  denote the solutions of \rf{eqn:variationalSTAR} and \rf{eqn:variationalSTARj} respectively. If $V_j \xrightarrow{M}V$ as $j\to\infty$ then $\phi_j\to \phi$ as $j\to\infty$ in $\tH^{-1/2}(\Gamma_\dag)$, and hence $\cS_{\Gamma_j}\phi_j=\cS_{\Gamma_\dag}\phi_j \to \cS_{\Gamma_\dag}\phi=\cS_{\Gamma}\phi$ as $j\to\infty$ in $W^{1,{\rm loc}}(\R^n)$.%
\end{thm}

\begin{defn}[BVP convergence]
Suppose that $\Gamma$ is non-empty and either bounded and open or compact, and that the sequence $(\Gamma_j)_{j\in \N_0}$ is uniformly bounded, and that each $\Gamma_j$ is non-empty and either bounded and open or compact. Then we say
that ``BVP convergence holds'' if $V_j \xrightarrow{M}V$ as $j\to\infty$, where  $V:=V(\Gamma)$ and $V_j:=V(\Gamma_j)$, with $V(\cdot)$ defined as in \eqref{eq:VG}.%

The rationale behind this terminology is that, with these conditions on $\Gamma$ and $\Gamma_j$, if $V_j \xrightarrow{M}V$, then
it follows by Theorem~\ref{thm:MoscoConv} that $\phi_j\to \phi$ in $H^{-1/2}$, where $\phi$ and $\phi_j$ are the solutions to \eqref{eqn:variationalSTAR} and \eqref{eqn:variationalSTARj}, respectively, and that $\cS_{\Gamma_j}\phi_j\to\cS_\Gamma \phi$ in $W^{1,\mathrm{loc}}(\R^n)$, where (noting \eqref{eq:agendef3} and the equivalence of \eqref{eq:agendef4} with the BIE)
$\cS_{\Gamma}\phi$ and $\cS_{\Gamma_j}\phi_j$ are the unique solutions to the BVPs $\sD(\Gamma)$ and $\sD(\Gamma_j)$, with data $g:= P_\Gamma g_\dag$ and $g_j:= P_{\Gamma_j}g_\dag$. In particular, $\cS_{\Gamma}\phi$ and $\cS_{\Gamma_j}\phi_j$ are the scattered fields for scattering of the incident field $u^i$ by $\Gamma$ and $\Gamma_j$, respectively, provided $g_\dag:=-P_{\Gamma_\dag} \gamma^\pm u^i$.
\end{defn}

Sufficient conditions guaranteeing BVP convergence
are given in the following proposition, which follows trivially from \rf{eq:increase}, \rf{eq:decrease}, \cite[Props.~3.33 and 3.34]{ChaHewMoi:13} and the ``sandwich lemma'' of Mosco convergence, Lemma \ref{lem:Sandwich}\footnote{We need only the case $s=-1/2$, but this proposition in fact holds, by the identical argument, with $-1/2$ replaced throughout by any $s\in \R$.}.
\begin{prop}
\label{prop:Mosco}
$V_j \xrightarrow{M}V$ in any of the three following situations:
\begin{enumerate}[(i)]
\item \label{Moscoi} {\bf(increasing open)} %
$V:=V(\Gamma)=\tH^{-1/2}(\Gamma)$ and $V_j:=V(\Gamma_j)=\tH^{-1/2}(\Gamma_j)$, where $\Gamma$ and $\Gamma_j$ are non-empty, bounded and open with $\Gamma_j\subset \Gamma_{j+1}$, $j\in\N_0$, and $\Gamma=\bigcup_{j\in\N_0} \Gamma_j$;
\item \label{Moscoii} {\bf (decreasing compact)}
$V:=V(\Gamma)=H^{-1/2}_\Gamma$ and $V_j:=V(\Gamma_j)=H^{-1/2}_{\Gamma_j}$, where $\Gamma$ and $\Gamma_j$ are non-empty and compact with $\Gamma_j\supset \Gamma_{j+1}$, $j\in\N_0$, and $\Gamma=\bigcap_{j\in\N_0} \Gamma_j$;
\item \label{Moscoiii} {\bf (sandwiched)}
$V:=V(\Gamma)=\tH^{-1/2}(\Gamma)$ and  $V_j=V(\Gamma_j)=\tH^{-1/2}(\Gamma_j)$, where $\Gamma$ and $\Gamma_j$ are non-empty, bounded and open with  $\widetilde{H}^{-1/2}(\Gamma) = H^{-1/2}_{\overline{\Gamma}}$, and there exist $\Gamma_j^-$ non-empty, bounded and open and $\Gamma_j^+$ non-empty and compact such that: $\Gamma_j^-\subset \Gamma_{j+1}^-$, $j\in\N_0$; $\Gamma_j^+\supset \Gamma_{j+1}^+$, $j\in\N_0$; $\Gamma_j^-\subset\Gamma_j\subset\Gamma_j^+$, $j\in\N_0$;  $\Gamma=\bigcup_{j\in\N_0} \Gamma_j^-$; $\overline\Gamma=\bigcap_{j\in\N_0} \Gamma_j^+$.
\end{enumerate}
\end{prop}

\begin{rem} By combining Theorem \ref{thm:MoscoConv} with Proposition \ref{prop:Mosco} \rf{Moscoi}-\rf{Moscoii} we reproduce the convergence results of \cite[Thm.~7.1]{ScreenPaper} (these phrased without reference to Mosco convergence). The convergence result obtained by combining Theorem \ref{thm:MoscoConv} with Proposition \ref{prop:Mosco} \rf{Moscoiii}, that applies in more subtle cases where neither $\Gamma_j\subset \Gamma$ nor $\overline{\Gamma}\subset \overline{\Gamma_j}$, is new. We present an example of this type (the ``square snowflake'') in \S\ref{sec:Square} below.
\end{rem}

\begin{rem} The sequences $\Gamma_j^\pm$ required in Proposition \ref{prop:Mosco}\rf{Moscoiii} exist if and only if
\begin{equation} \label{eq:construct}
\bigcap_{j\in \N_0}\Lambda_j^+ \!= \overline{\Gamma} \mbox{ and} \bigcup_{j\in\N_0} \Lambda_j^- = \Gamma,
\mbox{ where }\Lambda_j^+ := \overline{\Gamma \cup \bigcup_{i=j}^\infty \Gamma_i} \mbox{ and } \Lambda_j^- := \Gamma \cap \Bigg(\bigcap_{i=j}^\infty \Gamma_i\Bigg)^\circ;
\end{equation}
indeed if \eqref{eq:construct} holds then we can take $\Gamma_j^\pm := \Lambda_j^\pm$.
\end{rem}
While parts \rf{Moscoi} and \rf{Moscoiii} of Proposition \ref{prop:Mosco}, and their BEM versions in Proposition \ref{prop:DiscreteOpen} below, apply to more general bounded open screens $\Gamma$, we are particularly interested in cases where $\Gamma$ has fractal boundary.
The main challenge in applying part \rf{Moscoiii} in such cases is to show that $\widetilde{H}^{-1/2}(\Gamma) = H^{-1/2}_{\overline{\Gamma}}$. But, as noted in Proposition \ref{prop:TildeSubscript}, this holds if $\Gamma$ is ``thick'' in the sense of Triebel, as recently shown in \cite{caetano2018}. Moreover, it follows from \cite[Prop.~5.1]{caetano2018} that a large class of domains with fractal boundaries are thick. We provide some examples in \S\ref{sec:examples} and \S\ref{sec:Numerics}.

\subsection{Fractals that are attractors of iterated function systems} \label{sec:IFSpart1}
While Proposition \ref{prop:Mosco} \rf{Moscoii}, and its BEM version in Proposition \ref{prop:DiscreteCompact} below, apply to more general compact $\Gamma$, our main motivation for these results is the case when $\Gamma$ is fractal. %
An important fractal class (e.g.\ \cite[Chap.~9]{Fal}) is the set of fractals obtained as attractors of an iterated function system (IFS) $\{s_1,s_2,\ldots,s_\nu\}$. Here $\nu\geq 2$ and each $s_m:\R^{n-1}\to\R^{n-1}$ is a {\em contraction}, meaning that
 \begin{equation*} %
|s_m(x)-s_m(y)| \leq c|x-y|, \quad x,y\in \R^{n-1},
\end{equation*}
for some $c\in (0,1)$.
The attractor of the IFS is the unique non-empty compact set $\Gamma$ satisfying
\begin{equation} \label{eq:fixedfirst}
\Gamma = s(\Gamma), \quad \mbox{where} \quad s(U) := \bigcup_{m=1}^\nu s_m(U), \quad \mbox{for } U\subset \R^{n-1}.
\end{equation}
That \eqref{eq:fixedfirst} has a unique fixed point follows from the contraction mapping theorem since $s$ is a contraction on the set of compact subsets of $\R^{n-1}$, a complete metric space equipped with the standard Hausdorff metric, e.g.\ \cite[Thm.~9.1 and its proof]{Fal}.
If $\Gamma_0$ is any non-empty compact set then the sequence $\Gamma_j$ defined by
\begin{equation} \label{eq:prefract0}
\Gamma_{j+1} := s(\Gamma_j), \quad j=0,1,\ldots
\end{equation}
converges in the Hausdorff metric to $\Gamma$. In particular, if
$\Gamma_0$ is such that $s(\Gamma_0)\subset \Gamma_0$ then \cite[Thm.~9.1]{Fal}
\begin{equation} \label{eq:bigcap}
\Gamma_{j+1} \subset \Gamma_j, \quad j\in \N_0, \quad \mbox{and} \quad \Gamma = \bigcap_{j\in \N_0} \Gamma_j.
\end{equation}
In the case that $\Gamma$ is a fractal or where $\Gamma$ is not fractal but has a fractal boundary it is common to refer to $\Gamma_j$ as a sequence of {\em prefractals}.

The following is an obvious corollary of the above observations, Theorem \ref{thm:MoscoConv}, and Proposition \ref{prop:Mosco}\rf{Moscoii}.

\begin{cor} \label{cor:ifs} Suppose that $\nu\geq 2$, $s_1,\ldots,s_\nu$ are contractions, and that the non-empty compact set $\Gamma\subset \Gamma_\infty \cong \R^{n-1}$ is the unique attractor of the IFS $\{s_1,\ldots,s_\nu\}$, satisfying \eqref{eq:fixedfirst}. Suppose that $\Gamma_0$ is non-empty and compact with $s(\Gamma_0)\subset \Gamma_0$, and define the sequence of compact sets $\Gamma_j$ by \eqref{eq:prefract0}.
Then BVP convergence holds.

\end{cor}

\begin{rem}
With $\Gamma_j$ defined by \eqref{eq:prefract0}, it holds for any non-empty compact $\Gamma_0$ that $\Gamma_j\to \Gamma$ in the Hausdorff metric. However, if %
$s(\Gamma_0)\not\subset \Gamma_0$ it may or may not hold that $\phi_j\to \phi$ as $j\to\infty$. In particular, if: i) $\Gamma_0$ is a countable set; or ii) $n=3$ and $\dimH\Gamma_0 < 1$;
then $\Gamma_j$ defined by \eqref{eq:prefract0} is also countable or has $\dimH \Gamma_j<1$, respectively. (The latter case is a consequence of \cite[Prop.~2.3]{Fal}.) In such cases it follows from Lemma \ref{lem:Density} that $H^{-1/2}_{\Gamma_j}=\{0\}$ so that $\phi_j=0$, for $j\in \N_0$. Thus $\phi_j\not\to \phi$ unless $\phi=0$.
\end{rem}
\section{Boundary element methods and their convergence}%
\label{sec:Convergence}
In this section we propose Galerkin boundary element methods, based on piecewise constant approximations, for solving the screen scattering problems of \S\ref{sec:BVPs} and prove their convergence. These methods are based on discretisation of a sequence of more regular screens $(\Gamma_j)_{j\in \N_0}$; as in the previous section these converge in an appropriate sense to a limiting screen $\Gamma\subset \Gamma_\infty$  for which we wish to compute the solution to the scattering  problem $D(\Gamma)$ of Definition \ref{def:Open} or \ref{def:Compact}.

In particular we have in mind cases where $\Gamma$ is a compact set with fractal dimension $<n-1$, and cases where $\Gamma$ is a bounded open set with fractal boundary. We prove convergence results that apply in both of these cases, and our examples in \S\ref{sec:examples} and our numerical results in \S\ref{sec:Numerics} are of these types. In each of our convergence results, $\Gamma_j$ is a sequence of open sets, divided into a mesh of elements. Appropriately, given that we are approximating on very rough domains, the constraints on the elements are mild compared to conventional BEM results. In particular our elements need not be convex or even connected.

In  more detail, we assume each
$\Gamma_j$ is a non-empty bounded open set. On each $\Gamma_j$ we construct a  {\em pre-convex mesh} $M_j=\{T_{j,1},T_{j,2},\ldots,T_{j,N_j}\}$ in the sense of Appendix \ref{sec:fem}, meaning that $T_{j,l}\subset \Gamma_j$ is non-empty and open for $l=1,\ldots,N_j$, the convex hulls of $T_{j,l}$ and $T_{j,l'}$ are disjoint for $l'\neq l$, $\partial T_{j,l}$ has zero $(n-1)$-dimensional Lebesgue measure, and $\Gamma_j$ is the interior of the union of the closures of the $T_{j,l}$, i.e.
\[
\Gamma_j = \left(\,\bigcup\nolimits_{l=1}^{N_j} \overline{T_{j,l}^{\vphantom{1}}}\,\right)^\circ.
\]
We call %
$h_j:= \max_{l\in\{1,\ldots,N_j\}} \diam(T_{j,l})$ the {\em mesh size} %
and $T_{j,1},T_{j,2},\ldots,T_{j,N_j}$ the {\em elements} of the mesh.

Our Galerkin boundary element method (BEM) is to solve the variational problem \eqref{eqn:variationalSTARj} with $V_j$ chosen to be the $N_j$-dimensional space of piecewise constant functions on the mesh $M_j$, which we denote by $\VjBEM$. It follows from
\eqref{eq:agendef3}, \eqref{eq:agendef4}, and the comment following \eqref{DualDef}, that the BEM solution $\phi_j^h$ is defined explicitly by
\begin{equation} \label{eq:BEMdef}
(S_{\Gamma_j}^r\phi_j^h,\psi)_{L^2} = -(g_\dag,\psi)_{L^2}, \quad \forall \psi\in \VjBEM,
\end{equation}
where $(\cdot,\cdot)_{L^2}$ is the inner product on $L^2=L^2(\Gamma_\infty)= L^2(\R^{n-1})$. Moreover, when we are solving the scattering problem, $g_\dag = -P_{\Gamma_\dag} \gamma^\pm u^i$, and \eqref{eq:BEMdef} can be written as
\begin{equation*} %
(S_{\Gamma_j}^r\phi_j^h,\psi)_{L^2} = (\gamma^\pm u^i,\psi)_{L^2}, \quad \forall \psi\in \VjBEM.
\end{equation*}

\begin{defn}[BEM convergence]
Let the discrete approximation space $\VjBEM$ be defined as above, and let the Sobolev space $V(\Gamma)$ be defined as in \eqref{eq:VG}. If $\VjBEM$ Mosco-converges to $V(\Gamma)$
then we say that ``BEM convergence holds''. In this case, it follows by Theorem~\ref{thm:MoscoConv} that $\phi_j^h\to \phi$ in $H^{-1/2}$, where $\phi_j^h$ and $\phi$ are the solutions to \eqref{eq:BEMdef} and \eqref{eqn:variationalSTAR}, respectively, and $\cS_{\Gamma_j}\phi_j^h\to\cS_\Gamma \phi$ in $W^{1,\mathrm{loc}}(\R^n)$, with $\cS_\Gamma \phi$ the solution to $\sD(\Gamma)$ with $g=P_\Gamma g_\dag$.
\end{defn}

\subsection{Bounded open screens}
The following theorem provides a BEM convergence result in the case where the limiting screen $\Gamma$ is non-empty, bounded and open.
The result is stated for case (iii) in Proposition \ref{prop:Mosco}, but it also covers case (i), since if $\widetilde{H}^{-1/2}(\Gamma) = H^{-1/2}_{\overline{\Gamma}}$ (which we require in any case for well-posedness of $\sD(\Gamma)$) then case (i) is a special case of case (iii) with $\Gamma_j^-=\Gamma_j$ and $\Gamma_j^+=\overline\Gamma$, $j\in\N_0$.
\begin{thm}[$\Gamma$ bounded and open]
\label{prop:DiscreteOpen}
Let $\Gamma$, $\Gamma_j$ and $\Gamma_j^\pm$ satisfy the conditions of Proposition \ref{prop:Mosco}\rf{Moscoiii}.
For each $j$, let $M_j$ be a pre-convex mesh on $\Gamma_j$ with mesh size $h_j$.
Then BEM convergence holds provided that $h_j\to 0$ as $j\to\infty$.
\end{thm}
\begin{proof}
To show that $\VjBEM \xrightarrow{M}V=\tH^{-1/2}(\Gamma)$ we proceed by verifying that the conditions (i) and (ii) in Lemma \ref{lem:MoscoSuff} hold (with $W_j = V_j^h$, $W=V$, and $H=\tH^{-1/2}(\Gamma_0^+)$).
Regarding condition (ii), note that $\VjBEM\subset \tH^{-1/2}(\Gamma_j)$, and $\tH^{-1/2}(\Gamma_j)$ $\xrightarrow{M}\tH^{-1/2}(\Gamma)= V$ by Proposition \ref{prop:Mosco}\rf{Moscoiii}.
Regarding condition (i), let $v\in C^\infty_0(\Gamma)$ (which is dense in $V=\tH^{-1/2}(\Gamma)$ by definition of the latter) be given.
Since $\Gamma=\bigcup_{j\in\N_0} \Gamma_j^-$ and $\Gamma_j^-\subset \Gamma_{j+1}^-$ for $j\in \N_0$, it follows (e.g.\ \cite[Lem.~4.15]{ChLi16}) that
there exists $j_*\in \N$ such that
$\supp{v}\subset \Gamma_{j_*}^-\subset\Gamma_{j_*}$, and hence that $v\in C^\infty_0(\Gamma_j)\subset \tH^{-1/2}(\Gamma_j)$, for all $j\geq j_*$. In particular, $v|_{\Gamma_j}\in L^2(\Gamma_j)$, so that by Lemma \ref{lem:PWc}
\[\|v - v_{j}\|_{\tH^{-1/2}(\Gamma)} =  \|v - v_{j}\|_{\tH^{-1/2}(\Gamma_j)} \leq (h_j/\pi)^{1/2}\|v|_{\Gamma_j}\|_{L^2(\Gamma_j)}, \quad j\geq j_*, \]
where $v_{j}\in \VjBEM$ is the $L^2$ projection of $v$ onto the discrete space $\VjBEM$. But since for $j\geq j_*$ the norm $\|v|_{\Gamma_j}\|_{L^2(\Gamma_j)}=\|v|_{\Gamma_{j_*}}\|_{L^2(\Gamma_{j_*})}$ does not depend on $j$, it follows that $v_{j}\to v$ provided $h_j\to 0$.
\end{proof}

\subsection{Compact screens}

When the limiting screen $\Gamma$ is %
compact and $\tH^{-1/2}(\Gamma^\circ) \neq H^{-1/2}_\Gamma$,
neither Theorem \ref{prop:DiscreteOpen} nor its method of analysis can be applied. In particular, if $H^{-1/2}_\Gamma\neq \{0\}$ and $\Gamma$ has empty interior ($\Gamma^\circ=\emptyset$) then it is clearly impossible to approximate a limiting non-trivial integral equation solution $v\in H^{-1/2}_\Gamma$ by a sequence of elements of $C^\infty_0(\R^{n-1})$ supported inside $\Gamma$, since no such non-trivial functions exist. In the following theorem we address this case. The proof relies on
mollification arguments (see Appendix \ref{sec:Mollification})
to obtain smooth approximations to $v$ to which we can apply the BEM approximation theory (Lemma \ref{lem:PWc}). This produces approximating smooth functions whose support is strictly larger than that of $v$. This introduces a constraint on the %
sequence $\Gamma_j$ to which the analysis applies. In particular, each $\Gamma_j$ must contain $\Gamma(\epsilon) := \{\bx \in \Gamma_\infty:\dist(\bx,\Gamma)< \epsilon\}$, the $\epsilon$-neighbourhood of $\Gamma$, for some carefully chosen $j$-dependent $\epsilon>0$. As is the case throughout this section, our approximation space on $\Gamma_j$ remains the space $V_j^h$ of piecewise constants on a mesh $M_j$.

\begin{thm}[$\Gamma$ compact]
\label{prop:DiscreteCompact}
Let $\Gamma\subset\Gamma_\infty$ be non-empty and compact.
Let
$\Gamma_j$ be a sequence of bounded open subsets of $\Gamma_\infty$ such that $\Gamma\subset \Gamma(\epsilon_j)\subset\Gamma_j\subset \Gamma(\eta_j)$, for some $0<\epsilon_j<\eta_j$, with $\eta_j\to 0$ as $j\to\infty$.
Let $M_j$ be a pre-convex mesh on $\Gamma_j$ with mesh size $h_j$.
If $H^t_\Gamma$ is dense in $H^{-1/2}_\Gamma$ for some $t\in [-1/2,0]$
(always true for $t=-1/2$)
then BEM convergence holds if $h_j=o((\epsilon_j)^{-2t})$ as $j\to \infty$.
\end{thm}
\begin{proof}
Assuming that $\eta_j\to 0$ as $j\to\infty$,
to show that $\VjBEM \xrightarrow{M}V=H^{-1/2}_\Gamma$ we proceed by verifying the two conditions (i) and (ii) in Lemma \ref{lem:MoscoSuff} (with $W_j=\VjBEM$, $W=V$, and $H= \tH^{-1/2}(\Gamma_\dag)$ for some bounded
open set $\Gamma_\dag\subset \Gamma^\infty$ that contains $\Gamma(\eta_j)$ for $j\in \N_0$).
Regarding condition (ii) we note that $\VjBEM\subset H^{-1/2}_{\overline{\Gamma(\eta_j)}}$, and $H^{-1/2}_{\overline{\Gamma(\eta_j)}}\xrightarrow{M} H^{-1/2}_\Gamma= V$ by Proposition \ref{prop:Mosco}\rf{Moscoii}, since $\eta_j\to 0$.
To establish condition (i), suppose that $v\in H^t_\Gamma$, where $t\in[-1/2,0]$ is such that $H^t_\Gamma$ is dense in $H^{-1/2}_\Gamma$. (Note that we are including $t=-1/2$ as a possibility here, in which case density holds trivially.)
For each $j\in\N$ define $\tilde v_j:=\psi_{\epsilon_j/2}*v$ to be the mollification defined in Appendix \ref{sec:Mollification}.
Then $\tilde v_j\in C^\infty_0(\Gamma_j)$ (since $\tilde v_j$ is smooth and $\supp{\tilde v_j}\subset \overline{\Gamma(\epsilon_j/2)}\subset \Gamma(\epsilon_j)\subset\Gamma_j$) and $\|\tilde v_j- v\|_{H^{-1/2}(\Gamma_\infty)}\to 0$ as $j\to\infty$, since $\epsilon_j\to 0$. It remains to show that there exists $v_j\in \VjBEM$ such that $\|v_j-\tilde v_j\|_{H^{-1/2}(\Gamma_\infty)}\to 0$ as $j\to\infty$.  For this we define $v_j$ to be the orthogonal projection in $L^2(\Gamma_j)$ of $\tilde v_j\in C^\infty_0(\Gamma_j)$ onto $\VjBEM\subset L^2(\Gamma_j)$. From \eqref{eq:mollbound} we have that
$$
\|\tilde v_j\|_{L^2(\Gamma_\infty)} \leq c_{n-1} %
c'_{-t} (\epsilon_j/2)^{t} \|v\|_{H^t_\Gamma},
$$
so that, by Lemma \ref{lem:PWc},
\begin{align*}
\|v_j-\tilde v_j\|_{H^{-1/2}(\Gamma_\infty)} = \|v_j-\tilde v_j\|_{\tH^{-1/2}(\Gamma_j)} &\leq (h_j/\pi)^{1/2}\|\tilde v_j\|_{L^2(\Gamma_j)}\\
&=(h_j/\pi)^{1/2}\|\tilde v_j\|_{L^2(\Gamma_\infty)}\\
&=c_{n-1} %
c'_{-t} (\epsilon_j/2)^t(h_j/\pi)^{1/2}\|v\|_{H^t_\Gamma}.
\end{align*}
Hence $\|v_j-\tilde v_j\|_{H^{-1/2}(\Gamma_\infty)}\to 0$ as $j\to\infty$ provided that $h_j^{1/2}\epsilon_j^{t}\to 0$ as $j\to\infty$, which is equivalent to saying that $h_j=o(\epsilon_j^{-2t})$ as $j\to\infty$.
\end{proof}

\begin{rem} \label{rem:zero} Theorem \ref{prop:DiscreteCompact} applies in the case when $V= H^{-1/2}_\Gamma=\{0\}$ when (trivially) $H^0_\Gamma=\{0\}$ is dense in $V$, to give that $\VjBEM \xrightarrow{M}V$ as $j\to\infty$, provided that $\eta_j\to 0$ and $h_j\to 0$ as $j\to \infty$. But, in this case, if $\eta_j\to 0$ then, as argued in the above proof, $\VjBEM \subset H^{-1/2}_{\overline{\Gamma(\eta_j)}}\xrightarrow{M}V$ as $j\to\infty$, so that, by Lemma \ref{lem:MoscoSuff}, $\VjBEM\xrightarrow{M}V=\{0\}$ as $j\to\infty$, with no constraint on the mesh size $h_j$.  More generally, if $M_j$ is any mesh on any open set $\Gamma_j$ such that $V_j=\tH^{-1/2}(\Gamma_j)\xrightarrow{M}\{0\}$, then $\VjBEM\xrightarrow{M}\{0\}$.
\end{rem}

\begin{rem} \label{rem:every}
For many compact $\Gamma$ of interest it is straightforward to see how to construct sequences $\Gamma_j$ and $M_j$ satisfying the conditions of Theorem \ref{prop:DiscreteCompact} (see, for instance the examples in \S\ref{sec:examples}). But here is a construction that works in every case.
Let $\cM_1$, $\cM_2$, \ldots be a sequence of uniform meshes of convex elements on $\Gamma_\infty=\R^{n-1}\times \{0\}$, i.e. $\cM_j = \{ S_{j,n}:n\in \N\}$ is a family of open, bounded, convex, pairwise disjoint, congruent subsets of $\Gamma_\infty$ that tile $\Gamma_\infty$ in the sense that $\Gamma_\infty$ is the closure of $\bigcup_{n=1}^\infty S_{j,n}$. Let $h_j$ be the (common) diameter of $S_{j,n}$, for $n\in \N$, and assume that $h_j\to 0$ as $j\to\infty$. (For example, we might take (for $n=2$) $\cM_j = \{((n-1)h_j,nh_j):n\in \Z\}\times \{0\}$.)
For $j\in \N$ choose $\epsilon_j > 0$ with $\epsilon_j\to 0$ as $j\to\infty$, let $M_j$
denote the set of those elements of $\cM_j$ that have a non-empty intersection with $\Gamma(\epsilon_j)$, and let $\Gamma_j$ denote the interior of
$\overline{\bigcup_{T\in M_j} T}$, so that $M_j$ is a convex mesh on $\Gamma_j$ in the sense of \S\ref{sec:fem} (and, in particular, is pre-convex).
Then $\Gamma(\epsilon_j) \subset \Gamma_j \subset \overline{\Gamma(\epsilon_j + h_j)}\subset \Gamma(\eta_j)$ provided $\eta_j>\epsilon_j + h_j$, so that $\Gamma_j$ and $M_j$ satisfy the conditions of Theorem \ref{prop:DiscreteCompact} provided that $h_j=o(\epsilon_j^{-2t})$ for some $t\in [-1/2,0]$ such that $H^t_\Gamma$ is dense in $H^{-1/2}_\Gamma$.
\end{rem}

Combining Theorem \ref{prop:DiscreteCompact}  and Remark \ref{rem:zero} with the density results in Lemma \ref{lem:Density} we obtain the following corollary.

\begin{cor} \label{cor:dset}
Suppose that $\Gamma\subset\Gamma_\infty=\R^{n-1}\times\{0\}$, $V$, $\Gamma_j$, $\VjBEM$ and $M_j$ satisfy the assumptions of Theorem \ref{prop:DiscreteCompact}.
Then BEM convergence holds if either
\begin{enumerate}[(i)]
\item $\dimH{\Gamma}<n-2$ or $\Gamma$ is a $(n-2)$-set (so that $V=\{0\}$); %
\item $\Gamma$ is a $d$-set for some $n-2<d<n-1$
(so that $V\neq \{0\}$)
and $h_j=o(\epsilon_j^{\mu})$ as $j\to \infty$, for some $\mu>n-1-d$.
\end{enumerate}

\end{cor}
We emphasize that in Corollary \ref{cor:dset}(ii) it is possible (since $0<n-1-d<1$) to take $\mu< 1$, giving convergence when $h_j\sim \eps_j$ or even $h_j\gg \epsilon_j$. See the discussion around \eqref{eq:Mj2} and after Corollary \ref{cor:dsetconv}(b) below, where this result is applied.

\subsection{The BEM on fractals and prefractals arising from iterated function systems} %
An important IFS subclass is where each $s_m$ is a {\em contracting similarity}, %
i.e.
\begin{equation} \label{eq:similarfirst}
|s_m(x)-s_m(y)| = r_m|x-y|, \quad x,y\in \R^{n-1},
\end{equation}
for some $r_m\in (0,1)$.
In the case that the IFS additionally satisfies the standard {\em open set condition} \cite[(9.11)]{Fal}, that there exists a non-empty bounded open set $O\subset \R^{n-1}$ such that
\begin{align} \label{oscfirst}
s(O) \subset O \quad \mbox{and} \quad s_m(O)\cap s_{m'}(O)=\emptyset, \quad m\neq m',
\end{align}
it is well known \cite[Thm.~4.7]{Triebel97FracSpec} that $\Gamma$ is a $d$-set, in particular that $\dimH{\Gamma}=d$, where $d\in (0,n-1]$ is the unique solution of the equation
\begin{align} \label{eq:dfirst}
\sum_{m=1}^\nu r_m^d = 1.
\end{align}
Thus Corollary \ref{cor:dset} applies to this important class of fractal examples.

Suppose that $O$ satisfies the open set condition, and consider the sequence of compact sets $\Gamma_j$ defined by \eqref{eq:prefract0} with $\Gamma_0 := \overline{O}$, i.e.%
\begin{equation} \label{eq:prefract1}
\Gamma_0 := \overline{O}, \quad \Gamma_{j+1} := s(\Gamma_j), \quad j=0,1,\ldots
\end{equation}
so that $\Gamma_j = s^j(\overline{O})$, where $s^j$ is the mapping $s$ iterated $j$ times. Then the open set condition implies that $s(\overline{O}) \subset \overline{O}$, so that \eqref{eq:bigcap} holds, in particular $\Gamma\subset \overline{O}$, and it follows from Corollary \ref{cor:ifs} that the solution to the BVP $\sD(\Gamma_j)$ converges to that of $\sD(\Gamma)$ as $j\to\infty$.
  Note also that $s(O)$ has Lebesgue measure $|s(O)| = \sum_{m=1}^\nu r_m^{n-1}|O|$, so it follows from \eqref{eq:dfirst} that $d<n-1$ unless $|s(O)|=|O|$, in which case $s(\overline{O})=\overline{O}$ so that $\Gamma=\overline{O}$ (as $\Gamma$ is the unique fixed point).

Furthermore, for any $O$ satisfying the open set condition,
the sequence $\Gamma_j$ (cf.\ \eqref{eq:prefract1}) given by
\begin{equation} \label{eq:prefract}
\Gamma_0 := O, \quad \Gamma_{j+1} := s(\Gamma_j), \quad j=0,1,\ldots
\end{equation}
so that $\Gamma_j = s^j(O)$, is a natural sequence of \emph{open} sets converging to the fractal $\Gamma$ that %
can be discretised by the BEM.
For simplicity we assume for the rest of this subsection that
$r_m=r\in (0,1)$ for $m=1,\ldots,\nu$ in \eqref{eq:similarfirst},
in which case \eqref{eq:dfirst} becomes
\begin{equation} \label{eq:d2}
d = \log(1/\nu)/\log(r).
\end{equation}
Assume that $O$ is connected. Then the open set condition implies that $\Gamma_j$ given by \eqref{eq:prefract} has $\nu^j$ components, each component similar to $O$ but reduced in diameter by a factor $r^j$.  If $O$ is convex and $M_0=\{T_{0,1},\ldots,T_{0,N_0}\}$ is a convex mesh on $\Gamma_0$, a natural construction of a convex mesh on $\Gamma_j$
is to take
\begin{equation} \label{eq:Mj}
M_j := \left\{s_{m_1}\circ \cdots \circ s_{m_j}\left(T_{0,l}\right):1\leq m_{j'}\leq \nu \mbox{ for } j'=1,\ldots,j \mbox{ and } 1\leq l\leq N_0\right\}.
\end{equation}
The mesh $M_j$ has $N_0$ elements on each component of $\Gamma_j$, so $N_j = \nu^jN_0$ elements in total. If $h_0$ is the mesh size for $M_0$, then $M_j$ has mesh size $h_j = r^j h_0$.

One can also consider meshes for which
there is less than one degree of freedom (DOF) per component of $\Gamma_j$.
Precisely, for each $j$ choose $i=i(j)\in \{0,\ldots,j\}$, let $\tau_{i,1}$, \ldots, $\tau_{i,\nu^i}$ be the components of $\Gamma_i$, and consider the mesh $M_j$ on $\Gamma_j$ defined by
\begin{equation} \label{eq:Mj2}
M_j := \left\{\Gamma_j\cap\tau_{i,1}, \ldots, \Gamma_j\cap\tau_{i,\nu^i}\right\}.
\end{equation}
(Figure \ref{fig:PreConvexMeshes} shows the meshes given by \eqref{eq:Mj2} for $j=0,1,2$ and $0\leq i\leq j$ for a Cantor dust example from \S\ref{sec:CantorDust} below.)
If $i=j$ the mesh \eqref{eq:Mj2} is convex (the elements are convex sets), indeed $M_j$ coincides with the mesh given by \eqref{eq:Mj} with $N_0=1$. But if $i<j$ then the mesh $M_j$ given by \eqref{eq:Mj2} has only $N_j =\nu^i$ elements and each element is comprised of $\nu^{j-i}$ separate components. This mesh $M_j$ is clearly not convex, if $i<j$, but it is pre-convex (in the sense of Appendix \ref{sec:fem}) under the above assumptions on $O$, as captured in the following straightforward lemma.

\begin{lem} \label{lem:qc}%
Suppose that the bounded open set $O\subset \R^{n-1}$ is convex and satisfies the open set condition \eqref{oscfirst}. Then $M_j$ given by \eqref{eq:Mj2} is a pre-convex mesh on $\Gamma_j$.
\end{lem}

\begin{figure}
\includegraphics[width=.32\textwidth, clip, trim=114 68 100 20]{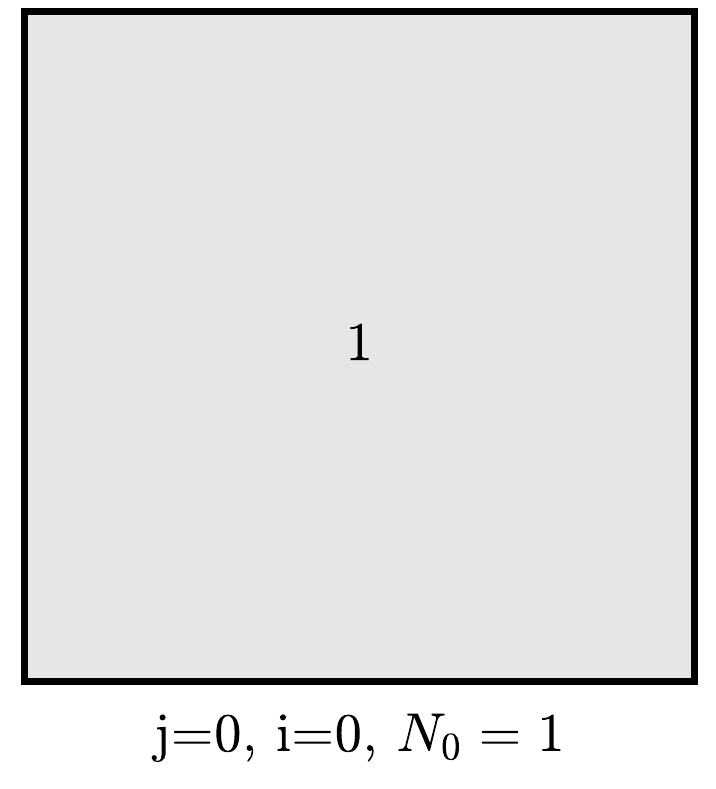}
\includegraphics[width=.32\textwidth, clip, trim=114 68 100 20]{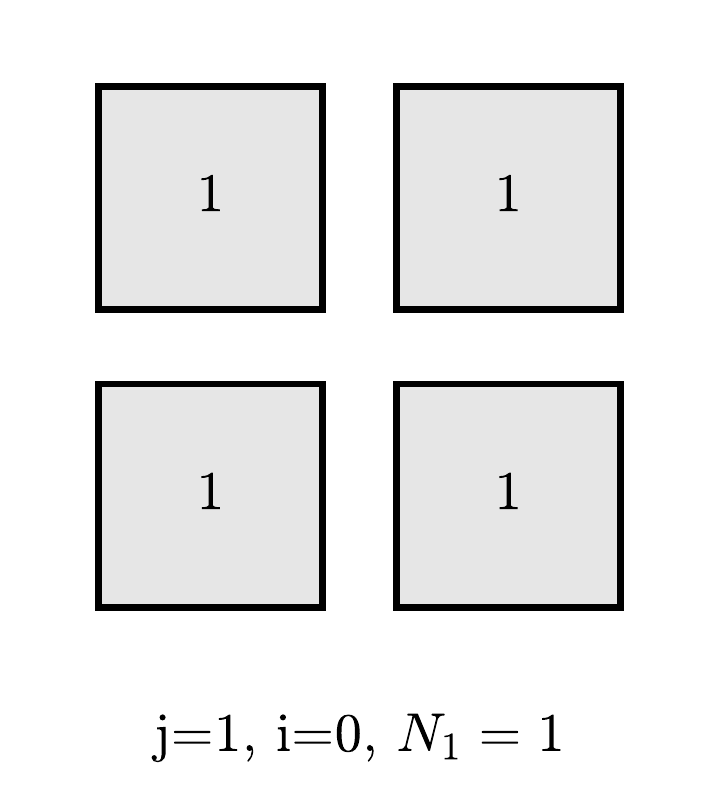}
\includegraphics[width=.32\textwidth, clip, trim=114 68 100 20]{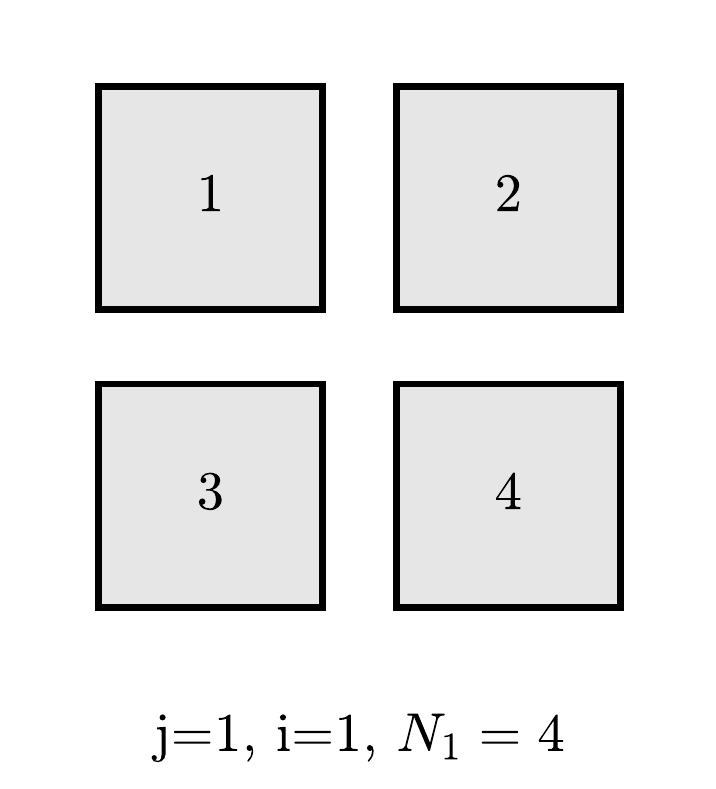}\\
\includegraphics[width=.32\textwidth, clip, trim=114 68 100 20]{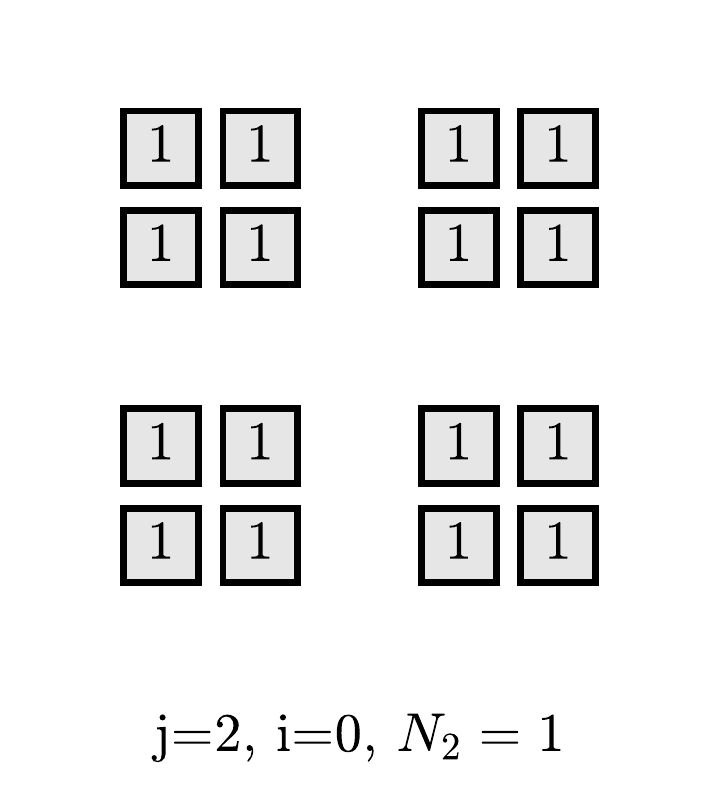}
\includegraphics[width=.32\textwidth, clip, trim=114 68 100 20]{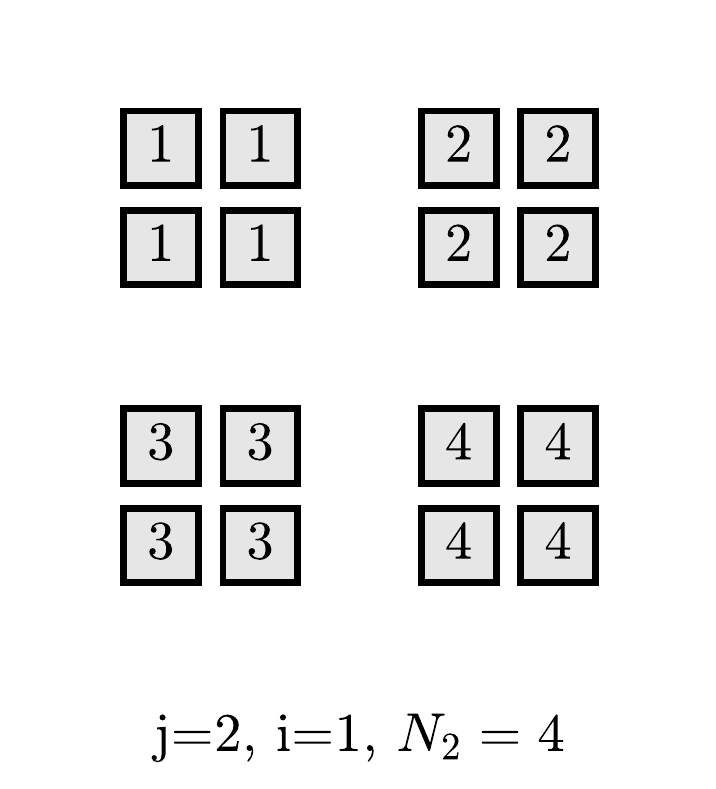}
\includegraphics[width=.32\textwidth, clip, trim=114 68 100 20]{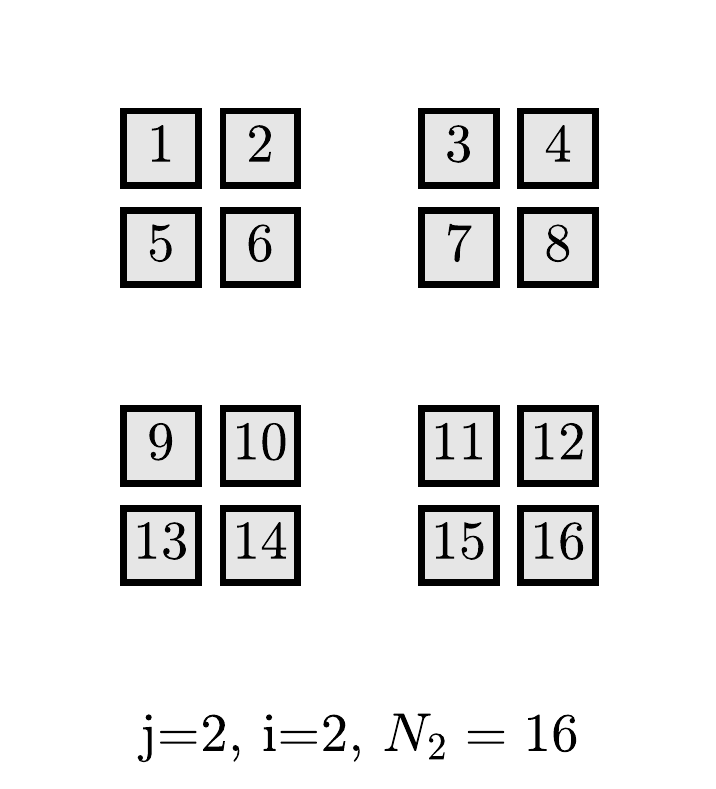}
\caption{Six pre-convex meshes, given by \eqref{eq:Mj2} with $j=0,1,2$ and $i= 0,\ldots,j$, for the Cantor dust example of \S\ref{sec:CantorDust} with parameter $\alpha=1/3$. The prefractals $\Gamma_j$ are defined by \eqref{eq:prefract} with $\Gamma_0:=O:=(-\delta,1+\delta)^2$ and $\delta=1/4$, and with $\nu=4$ and the similarities $s_1,\ldots,s_\nu$ given by \eqref{eq:similar}. Each mesh corresponds to a different pair of values $(j,i)$.
In each mesh the number inside each component of $\Gamma_j$ is the index $\ell\in\{1,\ldots,N_j\}$ of the mesh element $\Gamma_j\cap\tau_{i,\ell}$ to which that component belongs, and $N_j=\nu^i$ is the total number of elements in the mesh on~$\Gamma_j$.}
\label{fig:PreConvexMeshes}
\end{figure}

The following corollary, which follows from Theorems \ref{thm:MoscoConv} and \ref{prop:DiscreteCompact}, Corollary \ref{cor:dset} and Lemma \ref{lem:qc}, justifies convergence of the BEM when $\Gamma_j$ is the sequence of prefractals \eqref{eq:prefract}
and $M_j$ is defined by either
\eqref{eq:Mj} or \eqref{eq:Mj2},
under the additional requirement that $\Gamma\subset O$ (rather than $\Gamma \subset \overline{O}$).
We will see that this condition holds for obvious choices of $O$ in the Cantor set and Cantor dust examples that we treat in the next sections.

\begin{cor} \label{cor:dsetconv} Suppose that $\nu\geq 2$ and $s_1$, \ldots, $s_\nu$ are contracting similarities, satisfying \eqref{eq:similarfirst} with $r_m=r$, for $m=1,\ldots,\nu$ and some $r\in (0,1)$, and that the compact set $\Gamma$ is the unique fixed point of the IFS $\{s_1,\ldots,s_\nu\}$,  satisfying \eqref{eq:fixedfirst}.
Suppose that the bounded open set $O\subset \R^{n-1}$ is convex and satisfies the open set condition \eqref{oscfirst} and that $\Gamma \subset O$. Then $\Gamma$ is a $d$-set with $d\in (0,n-1)$ given by \eqref{eq:d2}.

Define $\Gamma_j$ by \eqref{eq:prefract}.
Then BEM convergence holds
if either
\begin{enumerate}[(a)]
\item  $M_j$ is the convex mesh defined by \eqref{eq:Mj}; or
\item $M_j$ is the pre-convex mesh defined by \eqref{eq:Mj2} and either
\begin{enumerate}[(i)]
\item $d\leq n-2$ (so that $V=\{0\}$);
\item $n-2<d<n-1$ and $i(j)> \mu j$, $j\in\N_0$, for some $\mu>n-1-d$.
\end{enumerate}
\end{enumerate}
\end{cor}
\begin{proof} That $\Gamma$ is a $d$-set with $d$ given by \eqref{eq:d2} follows as discussed above; it holds, as discussed above \eqref{eq:prefract1}, that $d<n-1$ since $\Gamma\subset O$ so that $\Gamma\neq \overline{O}$. Since $O$ is bounded and open and $\Gamma\subset O= \Gamma_0$, it follows that $\Gamma\subset \Gamma(\epsilon_0)\subset \Gamma_0 \subset \Gamma(\eta_0)$, for some $0<\epsilon_0<\eta_0$, so that, since $\Gamma=s^j(\Gamma)$ and $\Gamma_j=s^j(\Gamma_0)$, $\Gamma\subset \Gamma(\epsilon_j)\subset \Gamma_j \subset \Gamma(\eta_j)$, for $j\in \N_0$, with $\epsilon^j= r^j\epsilon_0$ and $\eta_j=r^j\eta_0$.

Suppose next that the mesh $M_j$ on $\Gamma_j$ is given by \eqref{eq:Mj}. Then the mesh size for $M_j$ is $h_j = r^jh_0$, and that $\VjBEM \xrightarrow{M}V$ as $j\to\infty$ follows from Corollary \ref{cor:dset} applied with $\mu=1$.

Suppose instead that  $M_j$ is given by \eqref{eq:Mj2} with $i=i(j)\in \{0,1,\ldots,j\}$ and let $L$ be the diameter of $\Gamma_0$. Then $h_j = r^iL$, and that $\VjBEM \xrightarrow{M}V$ as $j\to\infty$ follows from Lemma \ref{lem:qc} and Corollary \ref{cor:dset}, since, in the case $n-2<d<n-1$, $i> \mu j$ for some $\mu>n-1-d$ implies that $h_j=o(\epsilon_j^{\mu})$ as $j\to \infty$.
\end{proof}

Corollary \ref{cor:dsetconv} proves BEM convergence for the case $\Gamma \subset O$ under rather mild mesh refinement. When $d\leq n-2$ (zero limiting solution) there is no restriction on the mesh size, in accordance with Remark \ref{rem:zero}. When $n-2<d<n-1$ (non-zero limiting solution) it is possible to take $\mu< 1$ in Corollary \ref{cor:dsetconv}(ii)(b).\footnote{There is a curious ``discontinuity in convergence'' in Corollaries \ref{cor:dset} and \ref{cor:dsetconv}(ii): the infimum of the permitted $\mu$ values increases from $0$ to $1$ as $d$ decreases from $n-1$ to $n-2$, then jumps back to $0$ for $d<n-2$.
}
This means that BEM convergence holds with just one, or even less than one DOF per component of $\Gamma_j$.

\begin{rem}
Corollary \ref{cor:dsetconv} does not apply to what is the standard choice of prefractal sequence in the Cantor set, Cantor dust, and Sierpinski triangle examples that we treat in \S\ref{sec:examples}, namely to define the sequence $\Gamma_j$ by \eqref{eq:prefract} with $O:= (\mathrm{Conv}(\Gamma))^\circ$, the interior of the convex hull of $\Gamma$, which is an open interval, an open square, an open triangle, in the Cantor set, Cantor dust, and Sierpinski triangle cases, respectively. In each case this choice of $O$ satisfies the open set condition, but it does not hold that $\Gamma \subset O$, only that $\Gamma \subset \overline{O}$.

In such cases, assuming that $\tH^{-1/2}(\Gamma_j)=H^{-1/2}_{\overline{\Gamma_j}}$ for each $j$, given an arbitrary element $v\in V = H^{-1/2}_\Gamma$ one can prove the existence of a sequence of mesh sizes $h_j$ for which $\inf_{\substack{v_j^h\in V_j^h}}\|v-v_j^h\|_{\tH^{-1/2}}\to 0$ as $j\to\infty$ by a simple diagonal argument: since $V\subset V_j=\tH^{-1/2}(\Gamma_j)=H^{-1/2}_{\overline{\Gamma_j}}$, there exists $v_j\in C^\infty_0(\Gamma_j)$ such that $\|v-v_j\|_{H^{-1/2}}\leq (1+j)^{-1}$, say, and then by Lemma \ref{lem:PWc} there exists $h_j$ such that, if the mesh size on $\Gamma_j$ is less than $h_j$, there exists  $v_j^h\in V_j^h$ such that $\|v_j-v_j^h\|_{H^{-1/2}}\leq (1+j)^{-1}$, and the claimed convergence follows by the triangle inequality. However, the required choice of $h_j$ depends on $v_j$, which itself depends on $v$. So such an argument does not prove the existence of a single mesh refinement strategy for which $V_j^h \xrightarrow M V$.%

The development of a satisfactory convergence analysis for BEM on these standard prefractal sequences remains an open problem.%
\end{rem}

\section{Examples}\label{sec:examples}
We now apply the theory developed above to some specific examples of fractal screens.
In our first example, the Cantor set, the scattering problem is posed in $\R^2$ (so $n=2$), the screen being a subset of the one-dimensional hyperplane $\Gamma_\infty=\R\times\{0\}$.
In all other examples, the scattering problem is posed in $\R^3$ (so $n=3$), the screen being a subset of the two-dimensional hyperplane $\Gamma_\infty=\R^2\times\{0\}$. In the first three examples $\Gamma$ is a compact fractal $d$-set (for some $d<n-1$) that is the attractor
\eqref{eq:fixedfirst} of some IFS of contracting similarities $\{s_1,\ldots,s_\nu\}$. In the remaining examples $\Gamma$ is a (relatively) open subset of $\Gamma_\infty$ with fractal boundary.

\subsection{Cantor sets}\label{sec:CantorSet}
We consider first the Cantor set $\Gamma$, a compact subset
of $\R$
with empty interior, depending on a parameter $0<\alpha<1/2$,
defined by
\eqref{eq:fixedfirst} with $\nu=2$ and
\[ s_1(x_1)=\alpha x_1, \quad s_2(x_1)=\alpha x_1 + 1-\alpha.\]
Since the open set condition \eqref{oscfirst} holds with $O=(0,1)$,  $\Gamma$ is a $d$-set with Hausdorff dimension $d=\log{2}/\log(1/\alpha)$.
The standard prefractals $\Gamma_j$ are defined by \eqref{eq:prefract0} with $\Gamma_0:=\overline{O}=[0,1]$, so that $\Gamma_j$ is the union of $2^j$ closed intervals of length $\alpha^j$, and the $\Gamma_j$ are a decreasing nested sequence of compact sets satisfying \eqref{eq:bigcap}.
To construct what we will term ``thickened''
open prefractals %
satisfying the conditions of Theorem \ref{prop:DiscreteCompact}, choose a parameter $0<\delta<\frac1{2\alpha}-1$, and define $\Gamma_j$ by \eqref{eq:prefract} with $\Gamma_0:=O:=(-\delta,1+\delta)$, so that
$\Gamma_j$ is the disjoint union of the $2^j$ intervals of length $\alpha^j(1+2\delta)$ centred at the centres of the $2^j$ components of the standard prefractals.
Then $\Gamma\subset\Gamma(\epsilon_j)\subset \Gamma_j\subset \Gamma(\eta_j)$, for $\epsilon_j=\alpha^j\delta$ and for any
$\eta_j>\sup_{x\in\Gamma_j}\dist(x,\Gamma)=\alpha^j\max\{\delta,1/2-\alpha\}$. Clearly we can choose $\eta_j$ so that $\eta_j=O(\alpha^j)\to 0$ as $j\to\infty$.

The following result follows from Proposition \ref{prop:Scatters}, Lemma \ref{lem:Density}, %
Remark \ref{rem:zero}, Corollary \ref{cor:dset}, and Corollary \ref{cor:dsetconv}.

\begin{prop}[Cantor set]\label{prop:CantorSet}
For $0<\alpha<1/2$, let $\Gamma$ be the Cantor set defined above and let $\Gamma_j$ be either the standard compact prefractals or the thickened open prefractals.
Then
the BVPs $\sD(\Gamma)$ and $\sD(\Gamma_j)$ are well-posed,
BVP convergence holds,
and the solution of $\sD(\Gamma)$ is non-zero if and only if $g\ne0$.

For the thickened prefractals,
BEM convergence holds if $h_j = o(\alpha^{\mu j})$ for some $\mu> \mu_0:=1-\frac{\log2}{\log(1/\alpha)}$. In particular, BEM convergence holds for the convex mesh \eqref{eq:Mj}, and for the pre-convex mesh \eqref{eq:Mj2} with $i(j)>\mu j$, $j\in\N_0$, provided $\mu>\mu_0$.
\end{prop}

Interpreted in terms of DOFs, the final statement in Proposition \ref{prop:CantorSet} says that BEM convergence holds for the pre-convex mesh \eqref{eq:Mj2} on the thickened prefractal $\Gamma_j$ using $(2^{1-\frac{\log2}{\log1/\alpha}+\epsilon})^j$ DOFs,
for arbitrary $\epsilon>0$. For example, for the middle third Cantor set ($\alpha=1/3$) it suffices to take $1.3^j$ DOFs on $\Gamma_j$ (note that $\Gamma_j$ has $2^j$ components).

\subsection{Cantor dusts}\label{sec:CantorDust}

We now consider the Cantor dust $\Gamma$, a compact subset of $\R^2$ with empty interior, defined for $0<\alpha<1/2$ to be the Cartesian product of two identical Cantor sets from \S\ref{sec:CantorSet}.
Equivalently, the set $\Gamma$ is defined by
\eqref{eq:fixedfirst} with $\nu=4$ and
\begin{equation}\label{eq:similar}
\begin{aligned}
s_1(x_1,x_2)&=\alpha (x_1,x_2), &s_2(x_1,x_2)= \alpha (x_1,x_2) + (1-\alpha)(1,0),\\
s_3(x_1,x_2)&=\alpha (x_1,x_2) + (1-\alpha)(0,1), &s_4(x_1,x_2)= \alpha (x_1,x_2) + (1-\alpha)(1,1).
\end{aligned}
\end{equation}
Since the open set condition \eqref{oscfirst} holds with $O=(0,1)^2$, $\Gamma$ is a $d$-set with Hausdorff dimension $d=\log{4}/\log(1/\alpha)$.
The standard prefractals $\Gamma_j$ are defined by \eqref{eq:prefract0} with $\Gamma_0:=\overline{O}=[0,1]^2$, so that $\Gamma_j$ is the union of $4^j$ closed squares of side length $\alpha^j$, and the $\Gamma_j$ are a decreasing nested sequence of compact sets satisfying \eqref{eq:bigcap}.
See Figure \ref{Fig:CantorDust} for an illustration.

\begin{figure}[htb]
\centering
\includegraphics[width=.19\textwidth, clip, trim=70 20 55 10]{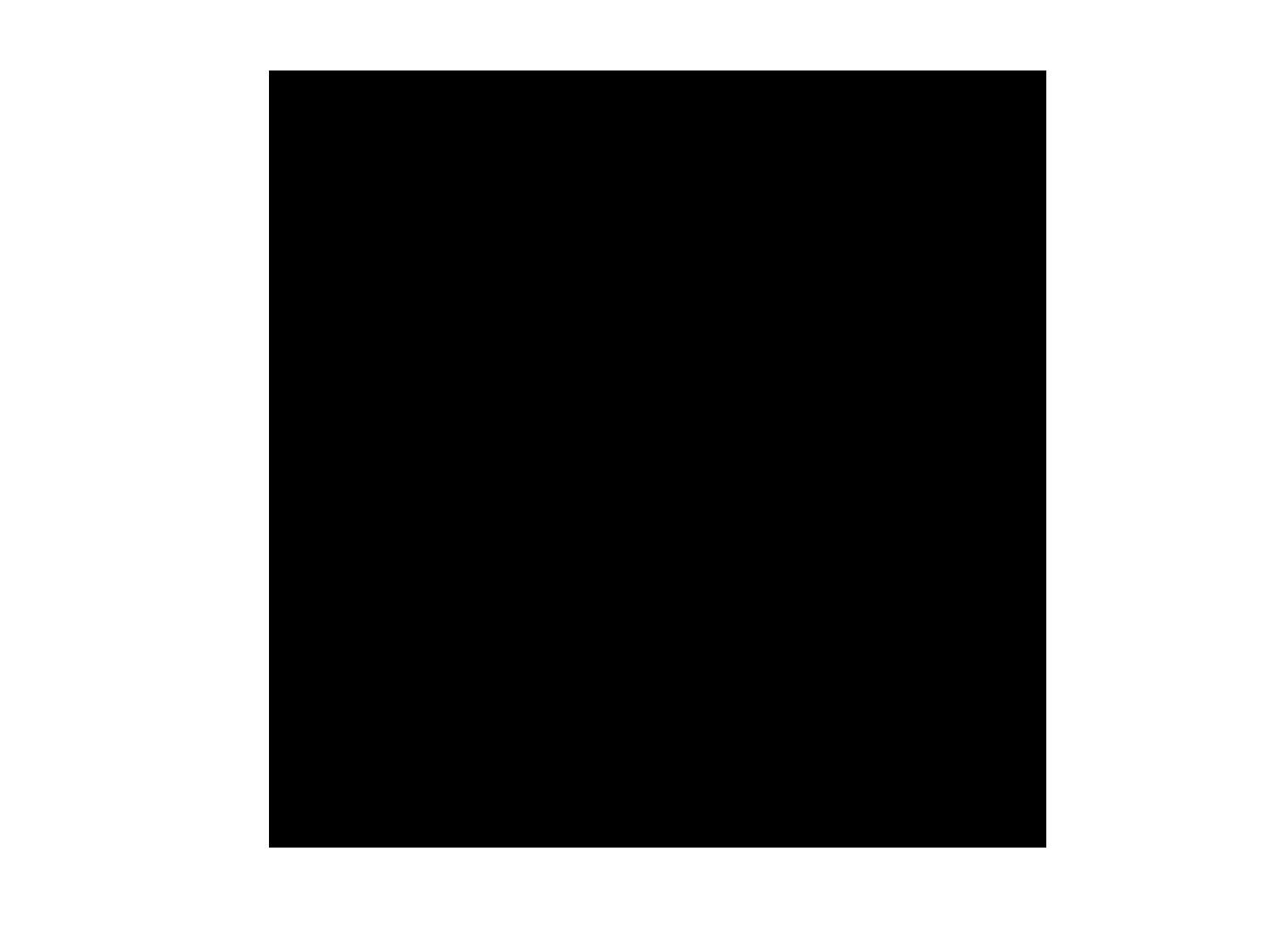}
\includegraphics[width=.19\textwidth, clip, trim=70 20 55 10]{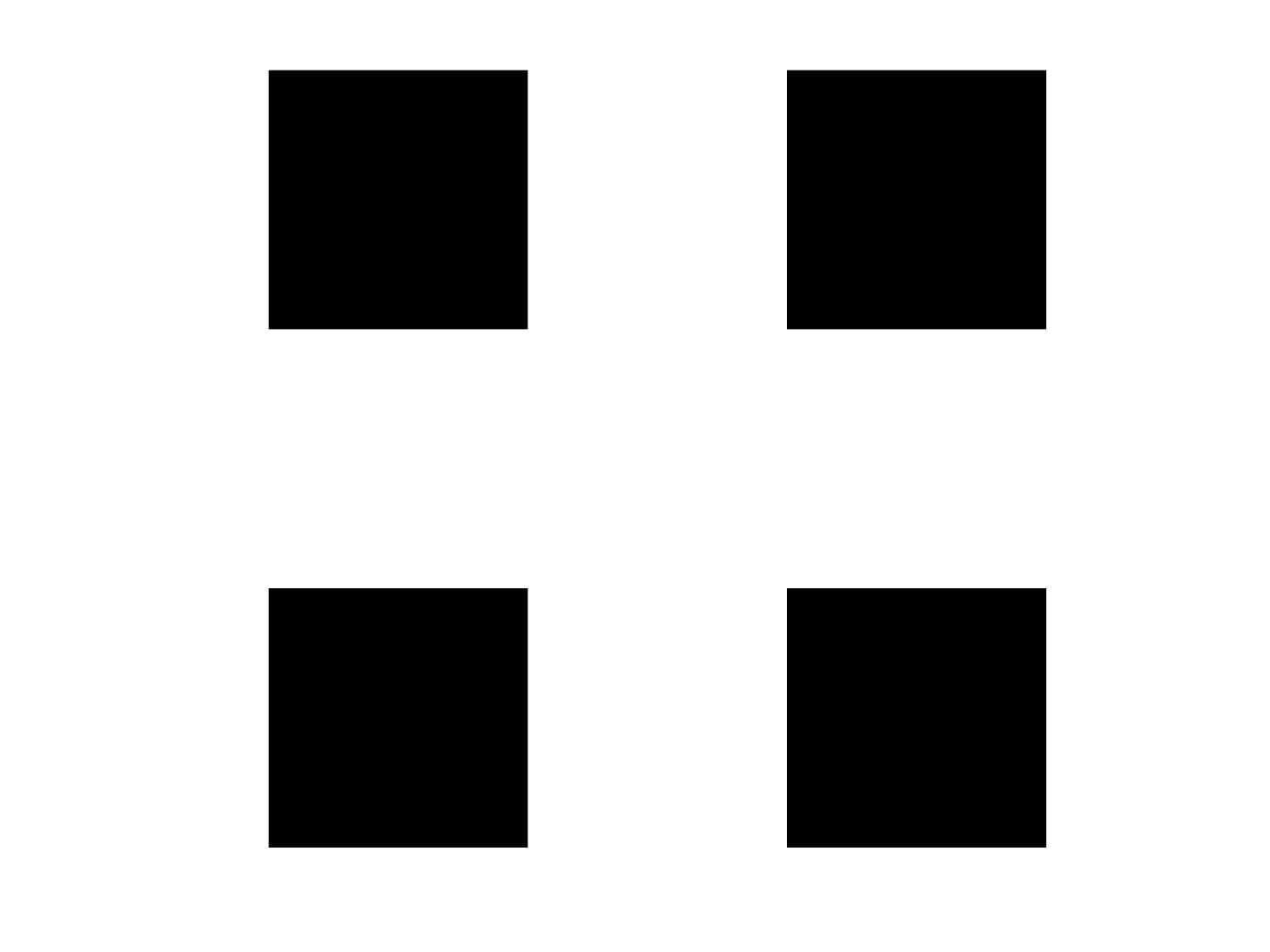}
\includegraphics[width=.19\textwidth, clip, trim=70 20 55 10]{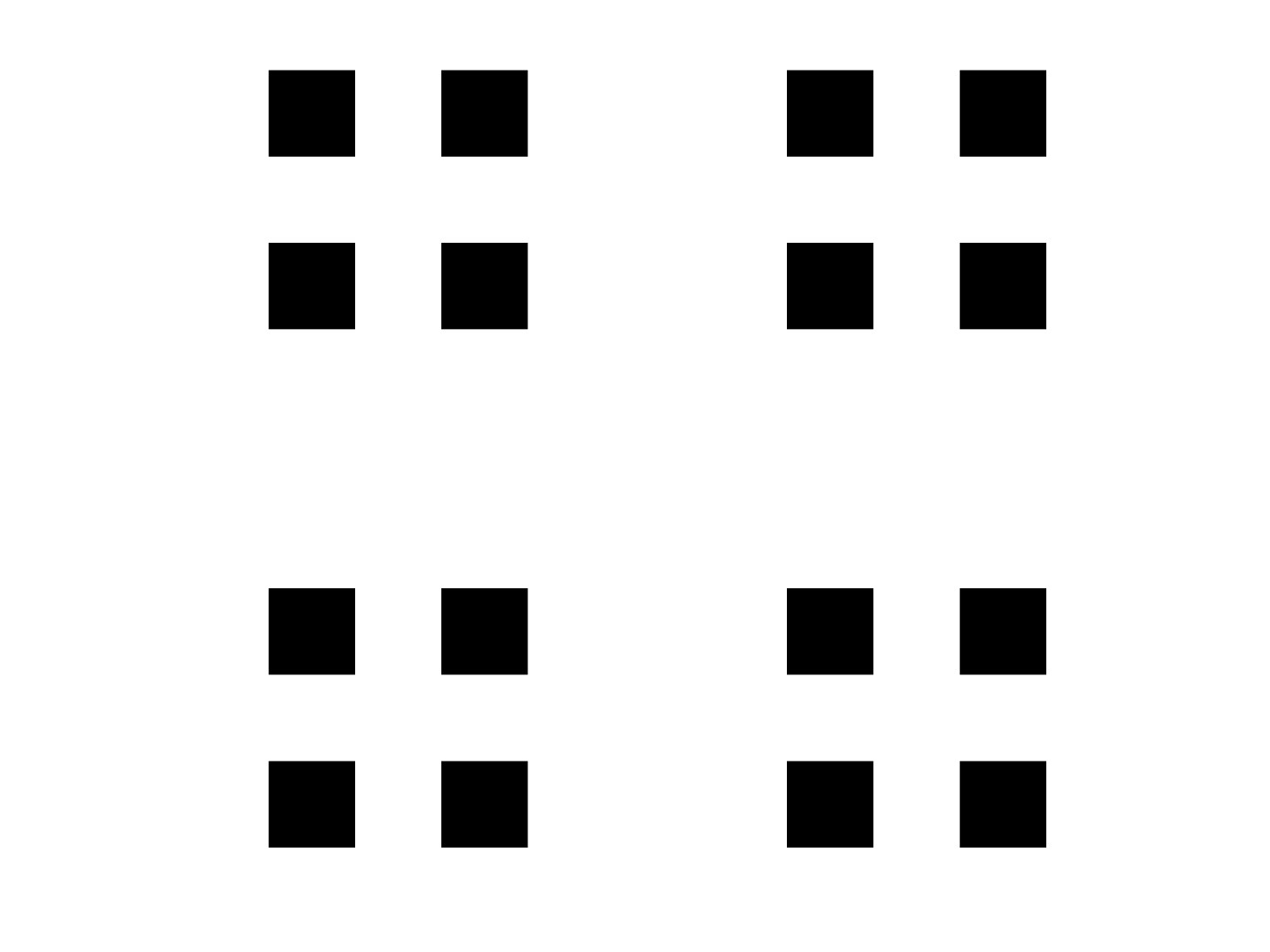}
\includegraphics[width=.19\textwidth, clip, trim=70 20 55 10]{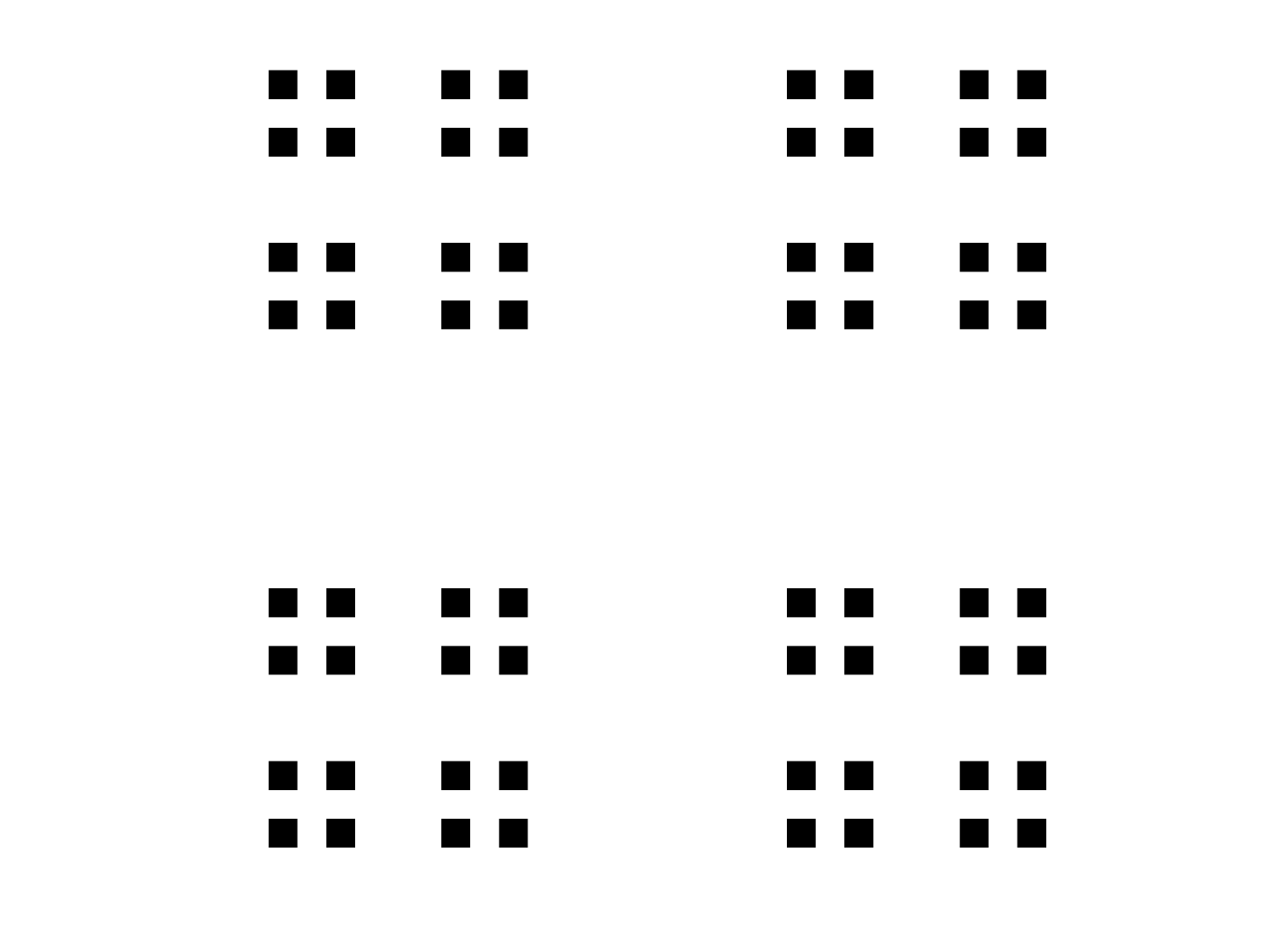}
\includegraphics[width=.19\textwidth, clip, trim=70 20 55 10]{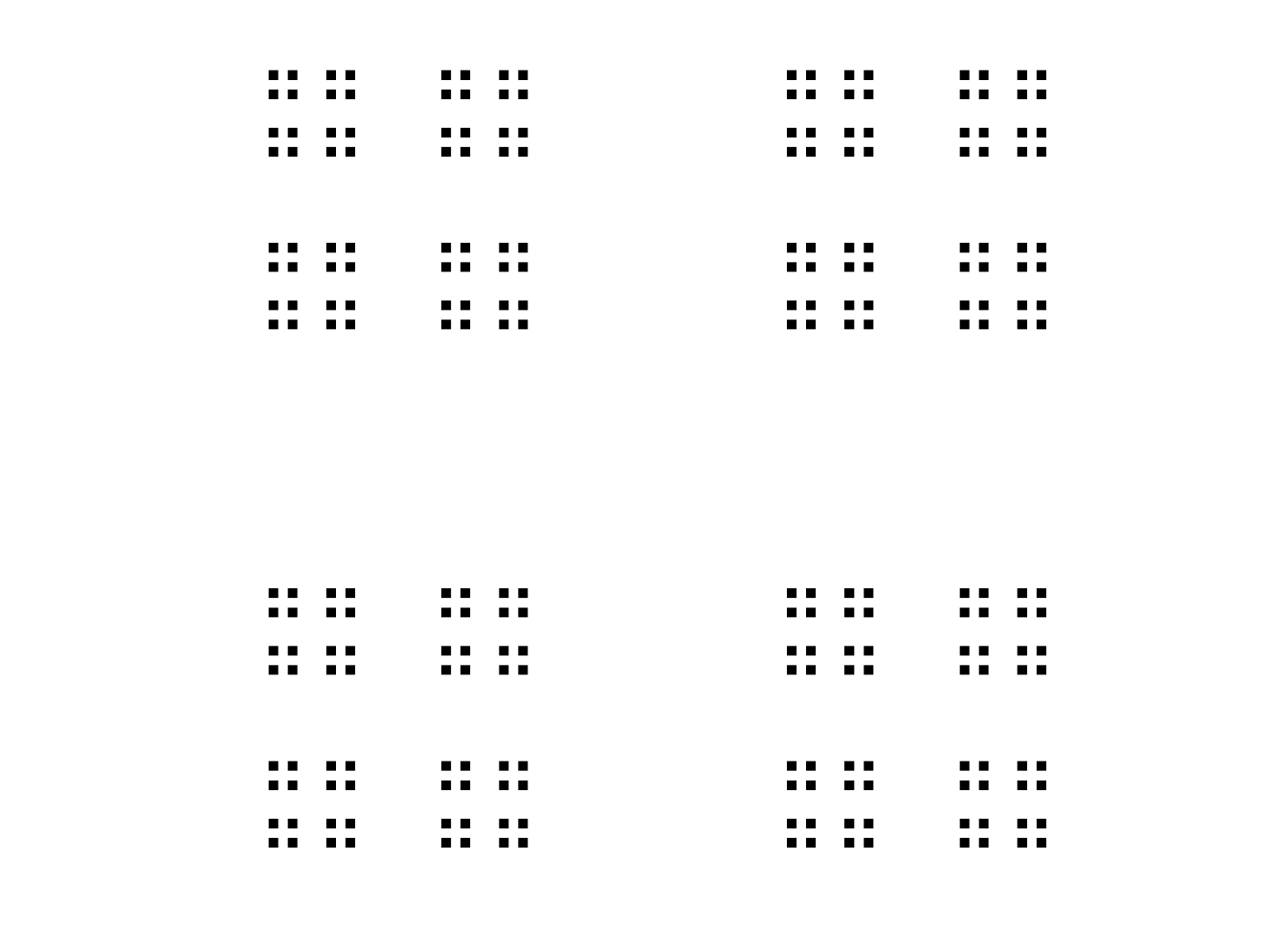}
\caption{The first five standard prefractals $\Gamma_0$, \ldots, $\Gamma_4$ of the middle third Cantor dust ($\alpha=1/3$).\label{Fig:CantorDust}}
\end{figure}

Thickened
open prefractals %
satisfying the conditions of Theorem \ref{prop:DiscreteCompact} can be constructed by taking Cartesian products of the thickened prefractals for the Cantor set. Explicitly, given $0<\delta<\frac1{2\alpha}-1$ we define $\Gamma_j$ by \eqref{eq:prefract} with $\Gamma_0:=O:=(-\delta,1+\delta)^2$, so that
$\Gamma_j$ is the disjoint union of the $4^j$ squares of side length $\alpha^j(1+2\delta)$ centred at the centres of the $4^j$ components of the standard prefractals. (Figure \ref{fig:PreConvexMeshes} shows (meshes on) $\Gamma_0$, $\Gamma_1$,  and $\Gamma_2$ when $\alpha=1/3$ and $\delta = 1/4$.)
Then $\Gamma\subset\Gamma(\epsilon_j)\subset \Gamma_j\subset \Gamma(\eta_j)$, for $\epsilon_j=\alpha^j\delta$ and for any
$\eta_j>\sup_{x\in\Gamma_j}\dist(x,\Gamma)=\sqrt{2}\alpha^j\max\{\delta ,1/2-\alpha\}\to 0$ as $j\to\infty$.
The following result (cf.~Proposition~\ref{prop:CantorSet}) follows from Proposition~\ref{prop:Scatters}, Lemma \ref{lem:Density}, %
Remark~\ref{rem:zero}, Corollary~\ref{cor:dset}, and Corollary~\ref{cor:dsetconv}.
\begin{prop}[Cantor dust]\label{prop:CantorDust}
For $0<\alpha<1/2$, let $\Gamma$ be the Cantor dust defined above and let $\Gamma_j$ be either the standard compact prefractals or the thickened open prefractals.
Then the BVPs $\sD(\Gamma)$ and $\sD(\Gamma_j)$ are well-posed,
and BVP convergence holds.

Now let $\Gamma_j$ be either the interior of
the standard compact prefractals, or the thickened open prefractals. Then
\begin{enumerate}[(i)]
\item for $0<\alpha\leq 1/4$, the solution of $\sD(\Gamma)$ is zero, and BEM convergence holds for any pre-convex mesh on $\Gamma_j$;
\item for $1/4<\alpha<1/2$, the solution of $\sD(\Gamma)$ is non-zero if and only if $g\neq 0$, and BEM convergence holds for the thickened prefractals if $h_j = o(\alpha^{\mu j})$ for some $\mu>\mu_0:=2-\frac{\log4}{\log(1/\alpha)}$. In particular, BEM convergence holds for the convex mesh defined by \eqref{eq:Mj}, and the pre-convex mesh defined by \eqref{eq:Mj2} with $i(j)>\mu j$, $j\in\N_0$, and $\mu>\mu_0$.
\end{enumerate}
\end{prop}

Recalling Proposition \ref{prop:Scatters}, and comparing Propositions \ref{prop:CantorSet} and \ref{prop:CantorDust}, we see that (provided the incident field doesn't vanish on the fractal screen),
Cantor sets ($n=2$) give rise to non-zero scattered fields for any value of $\alpha$,
while
Cantor dusts ($n=3$) give rise to non-zero scattered fields only for $\alpha>1/4$.

\subsection{Sierpinski triangle}\label{sec:Sierpinski}

The Sierpinski triangle (or gasket) $\Gamma$ is a compact subset of $\R^2$ with empty interior, defined by
\eqref{eq:fixedfirst} with $\nu=3$ and
\begin{align*}
s_1(x_1,x_2)&= \tfrac12(x_1,x_2), &  s_2(x_1,x_2)= \tfrac12 (x_1,x_2) + (\tfrac12,0), \\
s_3(x_1,x_2)&=\tfrac12 (x_1,x_2) + (\tfrac14,\tfrac{\sqrt{3}}{4}). &
\end{align*}
Since the open set condition \eqref{oscfirst} holds with $O$ the open unit equilateral triangle with vertices $(0,0)$, $(1,0)$, $(1/2,\sqrt{3}/2)$, $\Gamma$ is a $d$-set with Hausdorff dimension $d=\log{3}/\log2$.
The standard prefractals $\Gamma_j$ are defined by \eqref{eq:prefract0} with $\Gamma_0:=\overline{O}$ (the closed unit equilaterateral triangle), so that $\Gamma_j$ is the (non-disjoint) union of $3^j$ closed equilateral triangles of side length $2^{-j}$, and the $\Gamma_j$ are a decreasing nested sequence of compact sets satisfying \eqref{eq:bigcap}.
The first five prefractals are shown in Figure \ref{Fig:Sierpinski}.

\begin{figure}[htb]
\centering
\includegraphics[width=20mm]{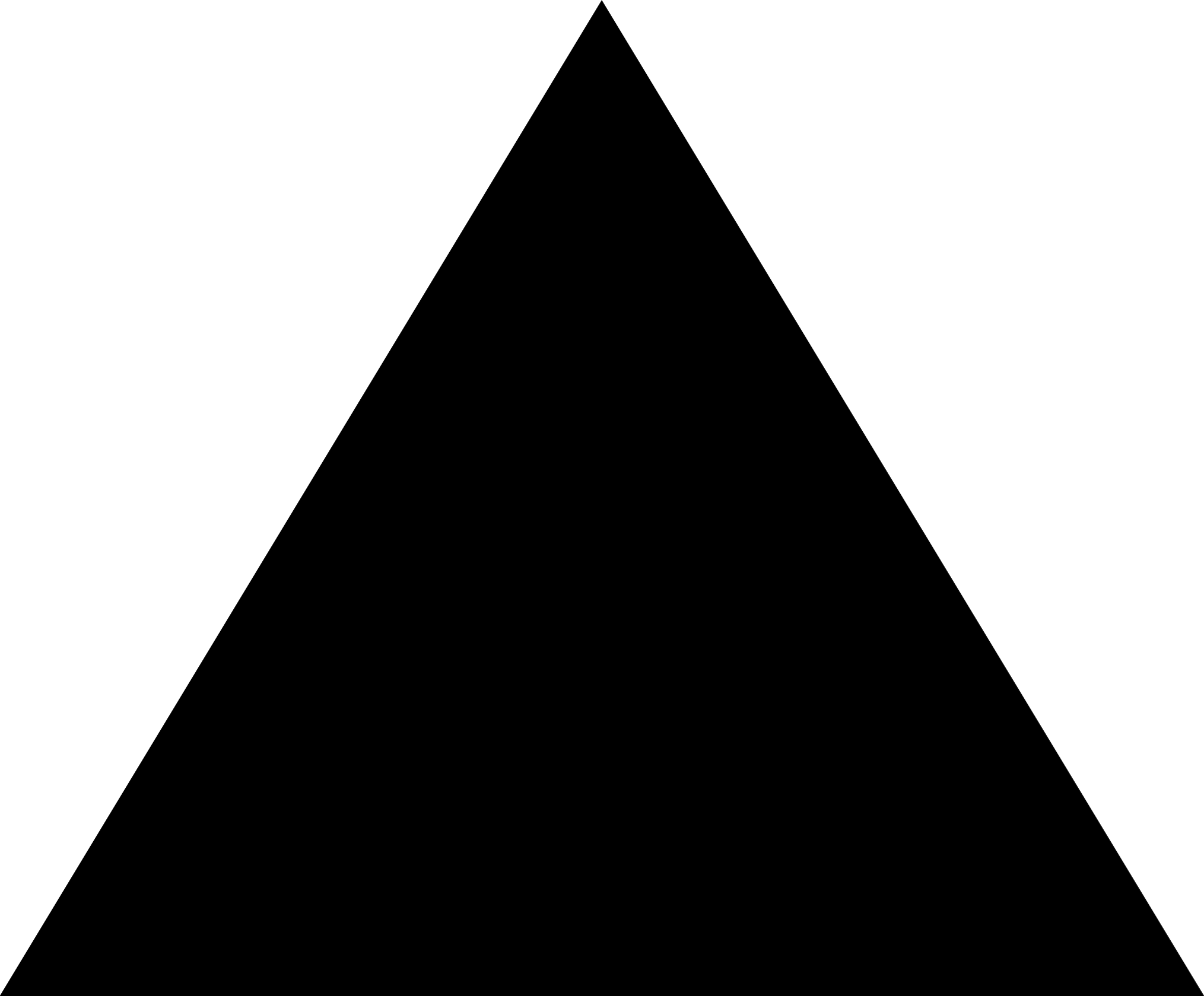}
\hs{2}
\includegraphics[width=20mm]{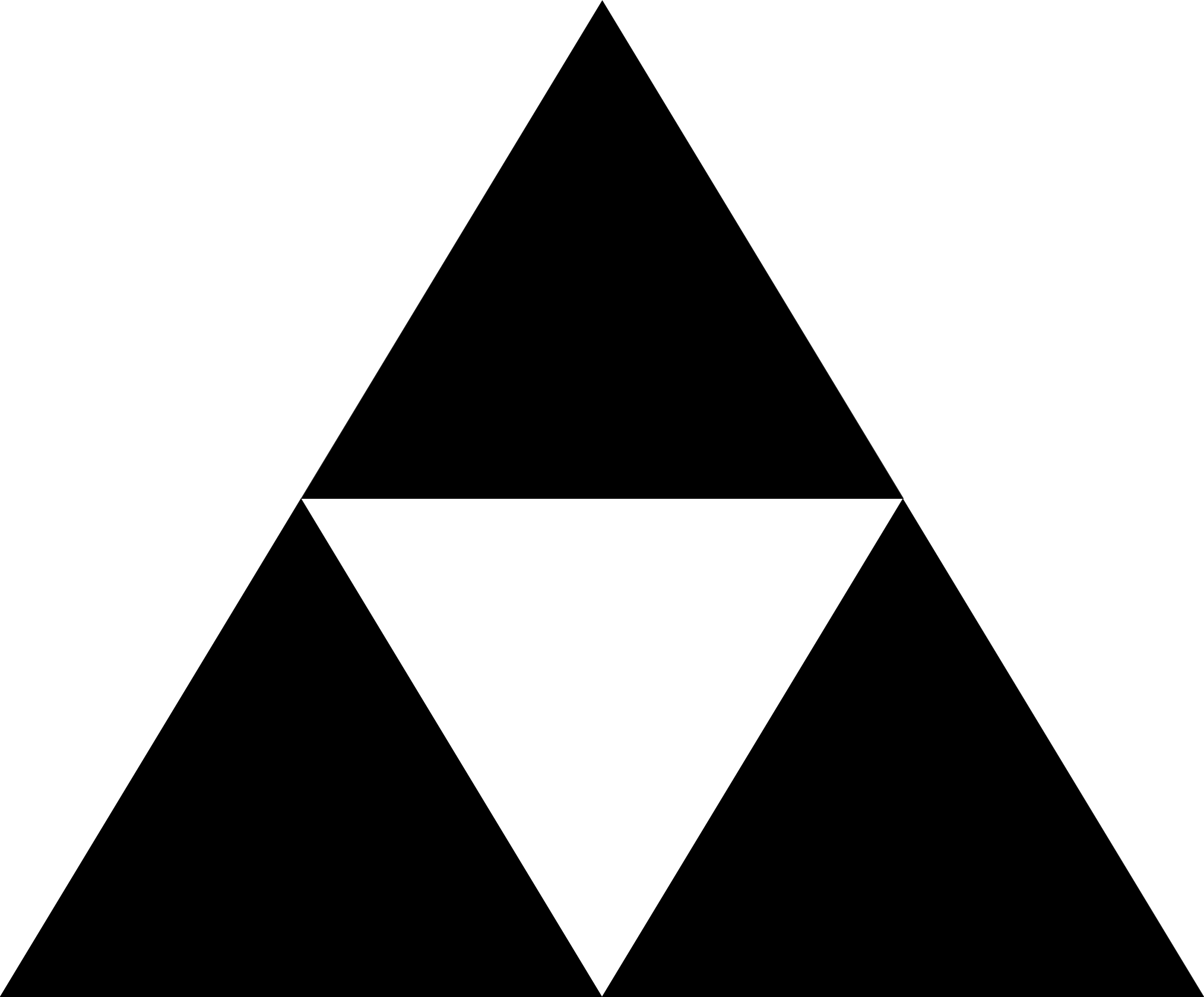}
\hs{2}
\includegraphics[width=20mm]{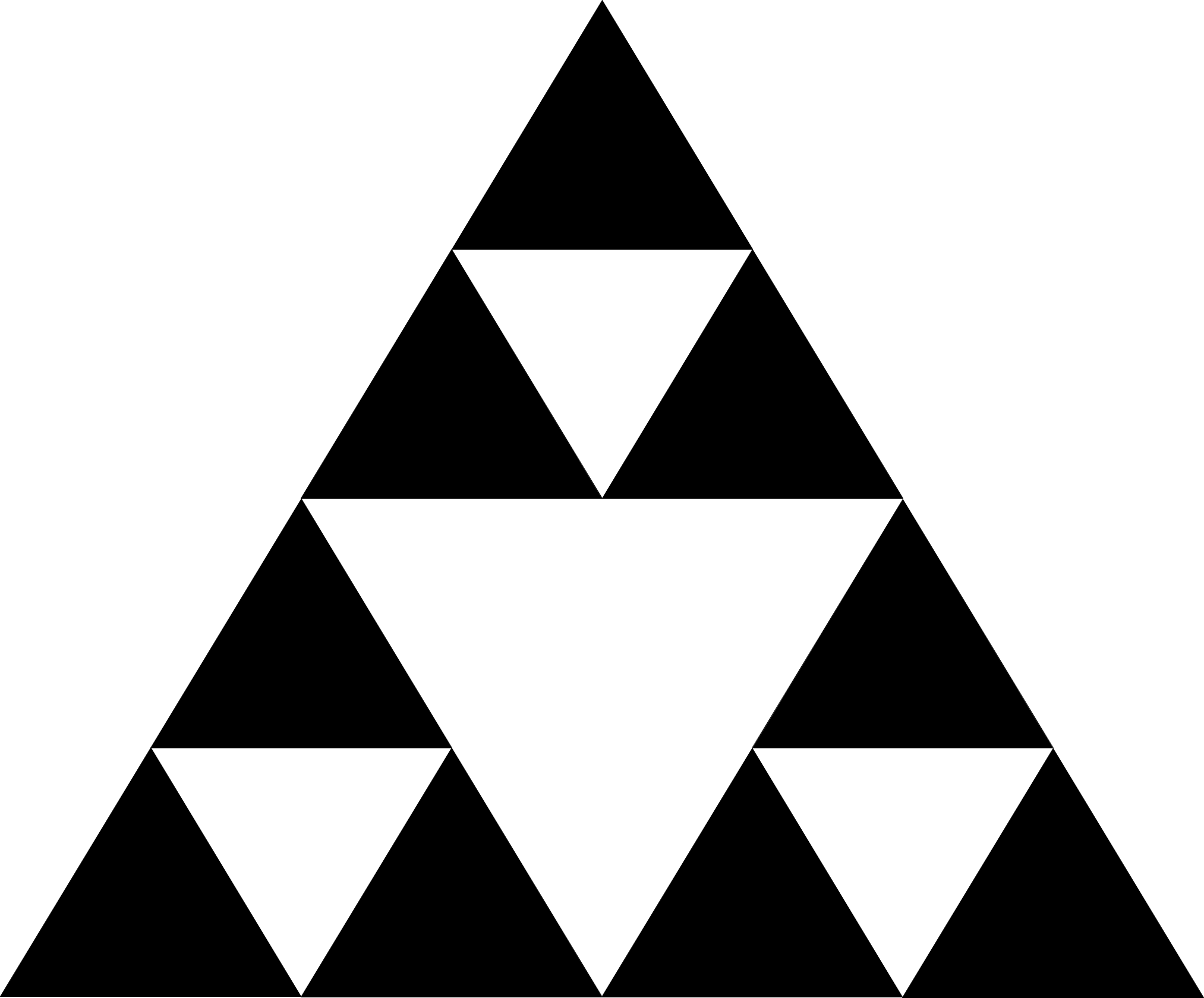}
\hs{2}
\includegraphics[width=20mm]{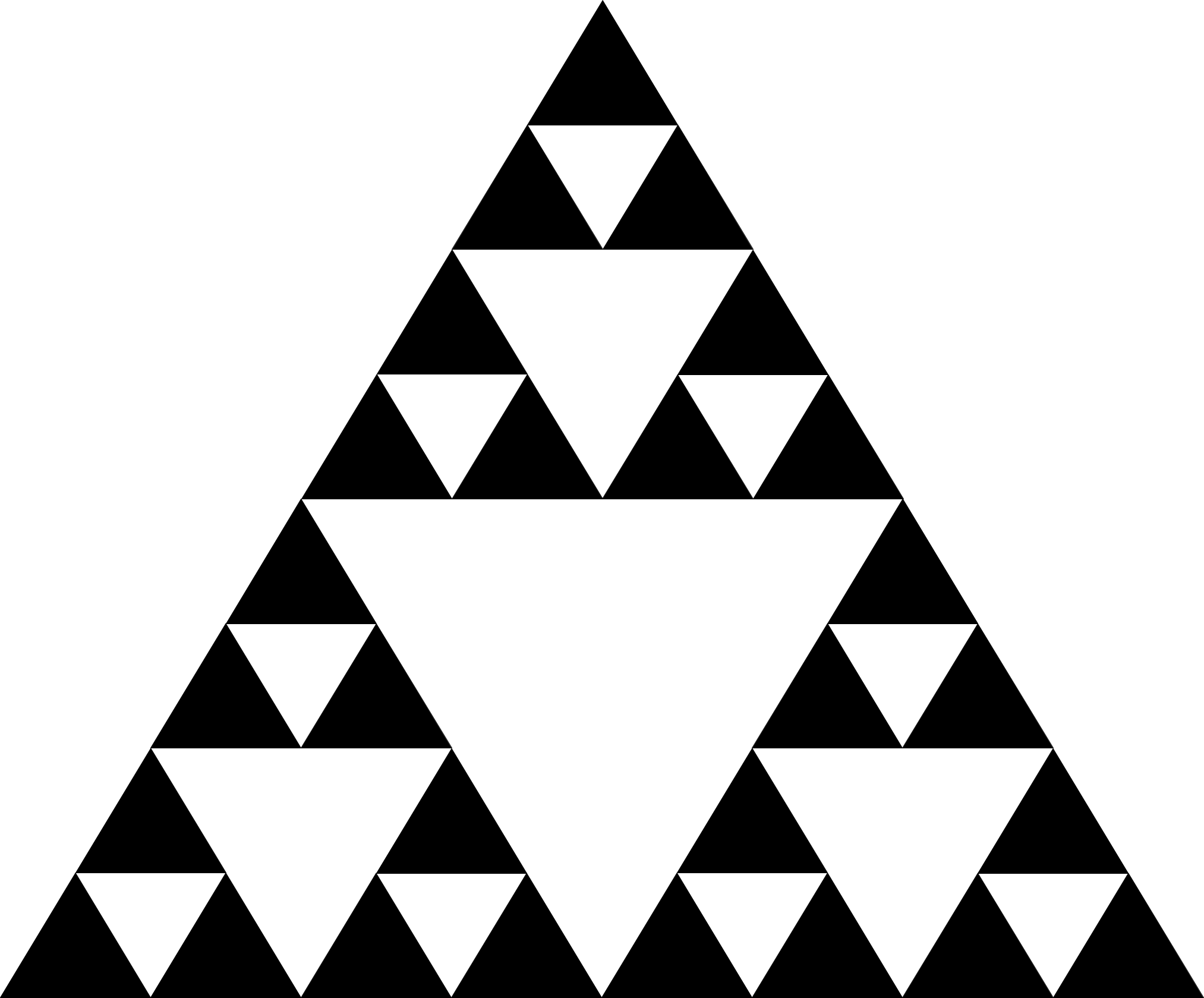}
\hs{2}
\includegraphics[width=20mm]{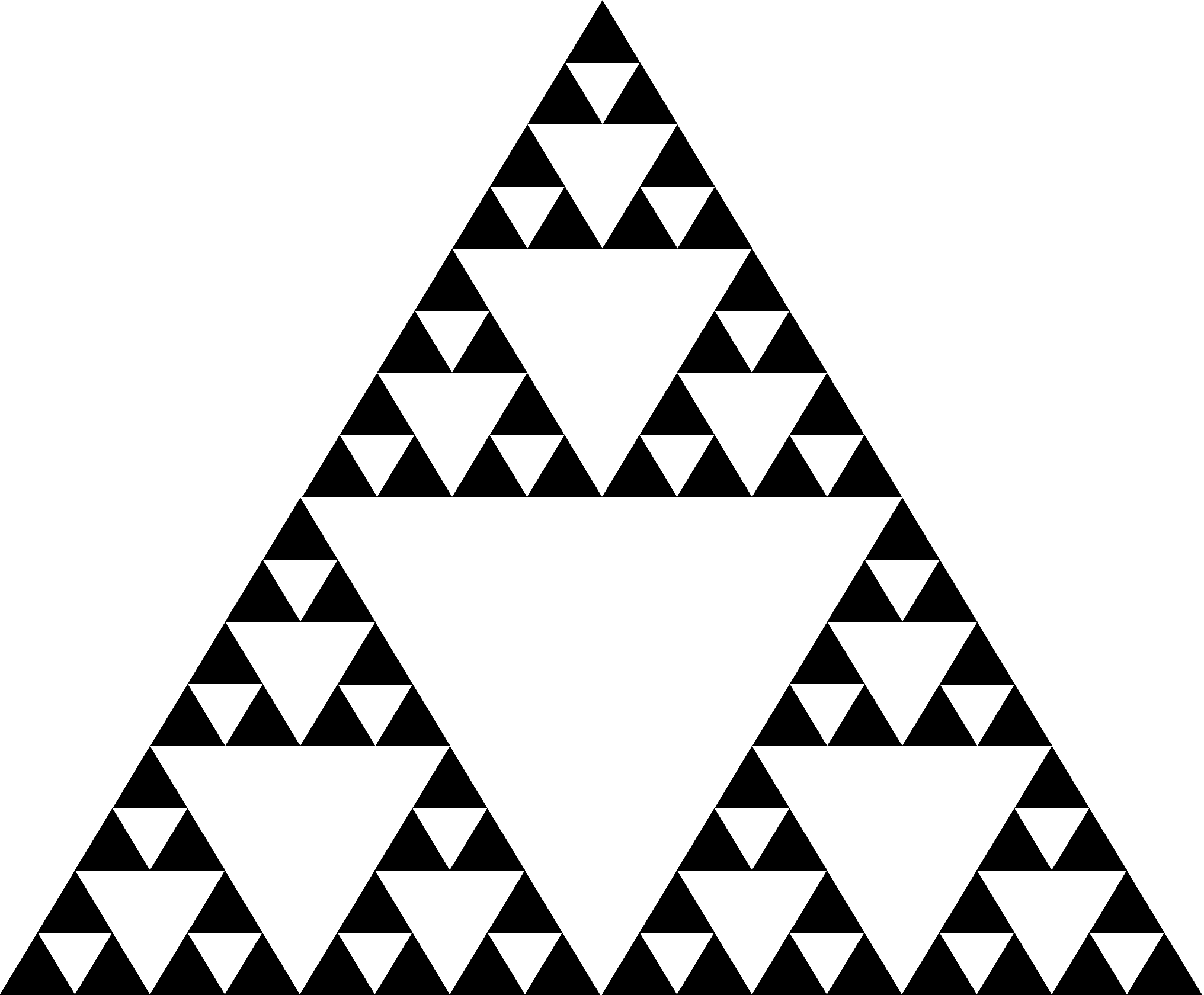}
\caption{The first five standard prefractals $\Gamma_0$, \ldots, $\Gamma_4$ of the Sierpinski triangle.\label{Fig:Sierpinski}}
\end{figure}

The interior $\Gamma_j^\circ$ of each prefractal is $C^0$ except at a finite set of points (the intersection points between neighbouring triangles), and hence by the results of \S\ref{sec:BVPs} (see in particular Remark \ref{rem:equiv}), the problems $\sD(\Gamma_j)$ and $\sD(\Gamma_j^\circ)$ share the same unique solution.
To define thickened
open prefractals %
satisfying the conditions of Theorem \ref{prop:DiscreteCompact}, given $\delta>0$ we define $O$ to be the open triangle of side length $1+2\delta$ with the same centre and side alignment as the unit equilateral triangle considered above, and define $\Gamma_j$ by \eqref{eq:prefract} with $\Gamma_0:=O$, so that $\Gamma_j$ is the (non-disjoint) union of $3^j$ equilateral triangles of side length $2^{-j}(1+2\delta)$. Then $\Gamma\subset\Gamma(\epsilon_j)\subset \Gamma_j\subset \Gamma(\eta_j)$, for $\epsilon_j=2^{-j}\delta/\sqrt{3}$ and for any
$\eta_j>\sup_{x\in\Gamma_j}\dist(x,\Gamma)=(2^{-j+1}/\sqrt{3})\max\{\delta ,1/4\}\to 0$ as $j\to\infty$.

For this choice of $O$, BEM convergence again follows from Theorem \ref{prop:DiscreteCompact}. But Corollary \ref{cor:dsetconv} does not apply here, because $O$ does not satisfy the open set condition (disjointness fails). In fact, it is easy to see that for the Sierpinksi triangle there does not exist an open set $O$ satisfying both the open set condition and the additional requirement that $\Gamma\subset O$.

\begin{prop}[Sierpinski triangle]
\label{prop:Sierpinski}
Let $\Gamma$ be the Sierpinski triangle defined above and let $\Gamma_j$ be either the standard compact prefractals or the thickened open prefractals.
Then the BVPs $\sD(\Gamma)$ and $\sD(\Gamma_j)$ are well-posed,
BVP convergence holds,
and the solution of $\sD(\Gamma)$ is non-zero if and only if $g\ne0$.

For the thickened prefractals,
BEM convergence holds if $h_j = o(2^{-\mu j})$ for some $\mu> 2-\frac{\log3}{\log2}$, in particular if $h_j = O(2^{-j})$.
\end{prop}

\subsection{Classical snowflakes}\label{sec:Snowflake}

We now consider the family of ``classical snowflakes'' studied in \cite{CapitanelliVivaldi15}, which generalise the standard Koch snowflake.
Each snowflake $\Gamma$ is a bounded open subset of $\R^2$ with fractal boundary,
depending on a parameter $0<\beta<\frac\pi2$, or equivalently on $\frac14<\xi:=\frac1{2(1+\sin\beta)}<\frac12$.
The standard Koch snowflake corresponds to the choice $\beta=\pi/6$ ($\xi=1/3$). We note that $\xi$ is denoted $\alpha^{-1}$ in \cite[\S2]{CapitanelliVivaldi15}; %
our notation follows that in \cite{caetano2018}.

To define and approximate $\Gamma$ we introduce a sequence of increasing nested open ``inner prefractals'' $(\Gamma_j^-)_{j\in \N_0}$, defining $\Gamma$ by $\Gamma:=\bigcup_{j\in\N_0}\Gamma_j^-$, and a sequence of decreasing nested closed ``outer prefractals'' $(\Gamma_j^+)_{j\in \N_0}$, such that $\Gamma\subset \Gamma_j^+$, $j\in \N_0$. The inner and outer prefractals for three examples (including the standard Koch snowflake) are shown in Figure \ref{fig:SnowflakesShapesInnerOuter}.

\begin{figure}[htb]
\includegraphics[width=\textwidth, clip, trim=130 130 100 115 ]{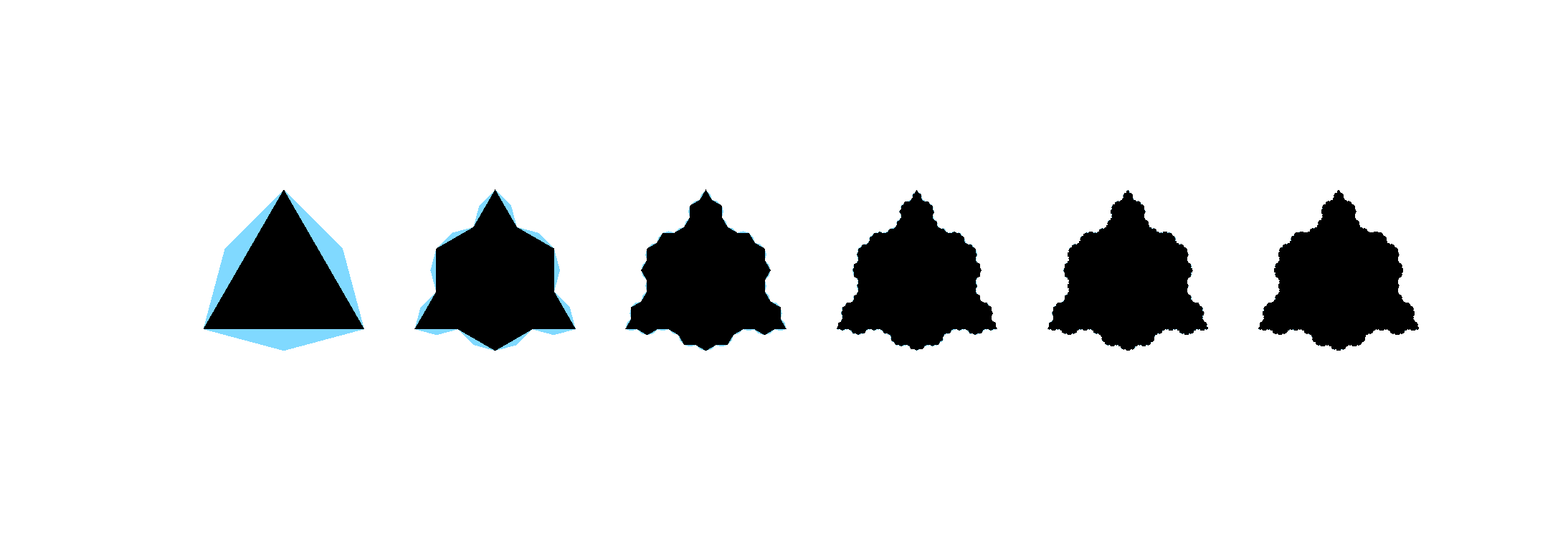}
\includegraphics[width=\textwidth, clip, trim=130 130 100 115 ]{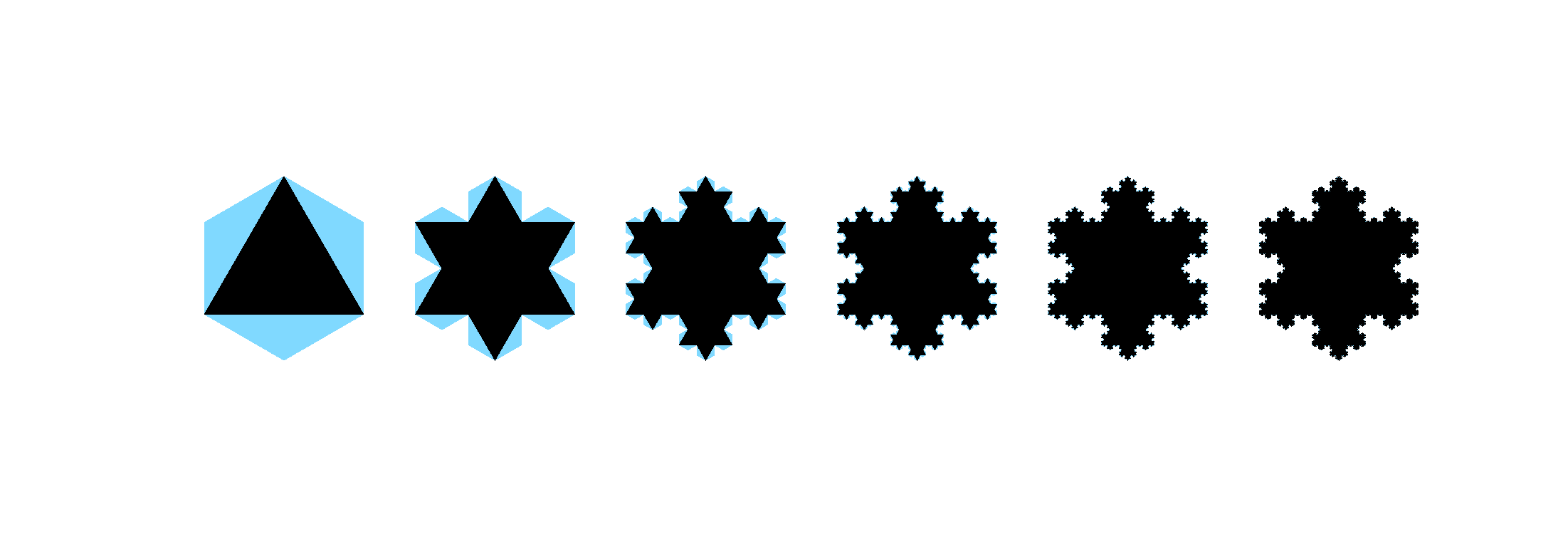}
\includegraphics[width=\textwidth, clip, trim=130 130 100 115 ]{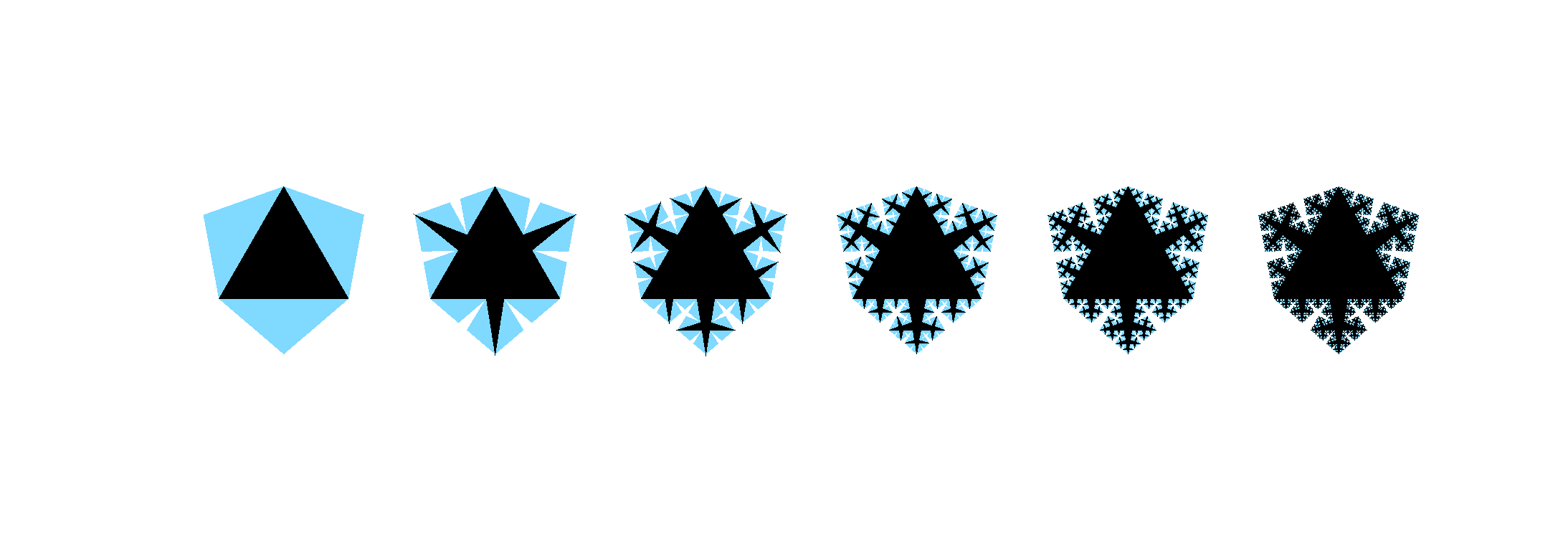}
\caption{The first 6 inner and outer prefractals $\Gamma_0^\pm,\ldots,\Gamma_5^\pm$ of the classical snowflakes for $\beta=\frac\pi3$ (top), $\beta=\frac\pi6$ (centre), $\beta=\frac\pi{20}$ (bottom).
The inner prefractals $\Gamma_j^-$ are the black shapes and the outer ones $\Gamma_j^+$ are the union of the blue and the black shapes. The parameter $0<\beta<\frac\pi2$ represents half the width of each convex angle of the inner prefractals (except possibly the three angles of the first inner prefractal), and the parameter $1/4<\xi=\frac1{2(1+\sin\beta)}<1/2$ represents the ratio of the side lengths of two successive prefractals.}
\label{fig:SnowflakesShapesInnerOuter}
\end{figure}

Each $\Gamma_j^-$ is an open polygon with $M_j^-:=3\cdot4^j$ edges of length $\xi^j$.
$\Gamma_0^-$ is the equilateral triangle with vertices $(0,0)$, $(1,0)$, $(\frac12,\frac12\sqrt3)$.
For $j\in \N$, $\Gamma_j^-$ is the union of $\Gamma_{j-1}^-$ and $M_{j-1}^-$ identical disjoint isosceles triangles (together with their bases) with base length $\xi^{j-1}(1-2\xi)$, side length $\xi^j$, height $\xi^{j-1}\sqrt{\xi-\frac14}$, apex angle $2\beta$, placed in such a way that  the midpoint of the base of the $k$th such triangle coincides with the midpoint of the $k$th side of $\Gamma_{j-1}^-$, for $k=1,\ldots,M_j^-$.

Our sequence of ``outer prefractals'' generalises those considered in \cite{banjai2017poincar} for the standard Koch snowflake.
Each $\Gamma_j^+$ is a closed polygon with $M_j^+:=6\cdot4^j$ edges of length $\xi^{j+\frac12}$.
$\Gamma_0^+$ is the convex hexagon obtained as union of $\Gamma_0^-$ and the three isosceles closed triangles with base the three sides of $\Gamma_0^-$, respectively, and height $\sqrt{\xi-\frac14}$ ($\Gamma_0^+$ is a regular hexagon only if $\beta=\frac\pi6$).
For $j\in \N$, $\Gamma_j^+$ is the difference of $\Gamma_{j-1}^+$ and $M_{j-1}^+$ identical disjoint isosceles triangles (together with their bases) with base length $\xi^{j-\frac12}(1-2\xi)$, side length $\xi^{j+\frac12}$, height $\xi^{j-\frac12}\sqrt{\xi-\frac14}$, apex angle $2\beta$, placed in such a way that the midpoint of the base of the $k$th such triangle coincides with the midpoint of the $k$th side of $\Gamma_{j-1}^+$, for $k=1,\ldots,M_j^+$.

Note that (cf.\ Figure \ref{fig:SnowflakesShapesInnerOuter})
$\Gamma_j^-\subset\Gamma_j^+$, $\Gamma_j^-\subset\Gamma_{j+1}^-$ and $\Gamma_{j+1}^+\subset\Gamma_{j}^+$ for each $j\in\N_0$.
In \cite{caetano2018} we proved that
\begin{itemize}
\item[\textbullet] $\overline{\Gamma}=\bigcap_{j\in\N_0}\Gamma_j^+$, $\Gamma=(\overline{\Gamma})^\circ$ and $|\partial\Gamma|=0$;  %
\item[\textbullet] $\partial\Gamma$ is a $d$-set with Hausdorff dimension $d=\log4/\log(1/\xi)$ (with the standard Koch snowflake having dimension $\log 4/\log 3$);
\item[\textbullet] $\Gamma$ is a ``thick'' domain (in the sense of Triebel), so $\widetilde{H}^{s}(\Gamma)=H_{\overline{\Gamma}}^{s}$ for all $s\in\R$.
\end{itemize}
Combining these facts with Theorem~\ref{thm:OpenPrime}, Proposition~\ref{prop:equiv}, Proposition~\ref{prop:Mosco}\rf{Moscoiii} and Theorem~\ref{prop:DiscreteOpen}, gives the following result.

\begin{prop}[Classical snowflakes]\label{prop:Snowflake}
Let $0<\beta<\frac\pi2$, and define the classical snowflake $\Gamma$ and its inner and outer prefractals $\Gamma_j^\pm$ as above.
Let $\Gamma_j$ be any sequence of bounded open sets satisfying $\Gamma_j^-\subset\Gamma_j\subset\Gamma_j^+$, with %
$\tH^{-1/2}(\Gamma_j)=H^{-1/2}_{\overline{\Gamma_j}}$ (in particular this applies if $\Gamma_j=\Gamma_j^-$ or $\Gamma_j=(\Gamma_j^+)^\circ$, since then $\Gamma_j$ is $C^0$). %
Then the BVPs $\sD(\Gamma)$ and $\sD(\Gamma_j)$ are well-posed, and BVP convergence holds. Furthermore, BEM convergence holds if $h_j\to 0$ as $j\to\infty$.
\end{prop}

\subsection{Square snowflake}\label{sec:Square}
As our final example we consider the ``square snowflake'' studied in \cite{sapoval1991vibrations} (see also \cite[\S7.6]{grebenkov2013geometrical} and the references therein).
This is an open subset of $\R^2$ with fractal boundary, constructed as the limit of a sequence of \textit{non-nested} polygonal prefractals $\Gamma_j$, $j\in\N_0$, the first five of which are shown in Figure \ref{fig:SquareSnowflakeShapesBlack}.
\begin{figure}[htb]
\includegraphics[width=\textwidth, clip, trim = 25 90 0 90]{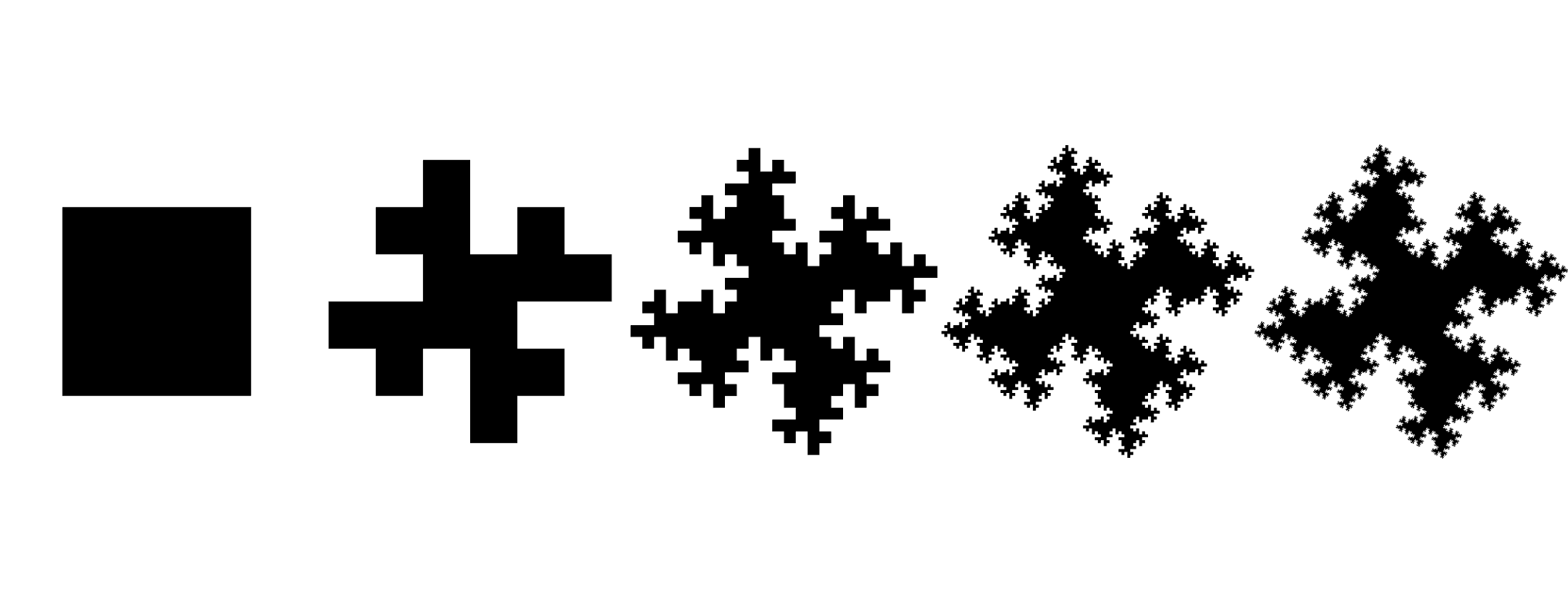}
\caption{The first five standard prefractals $\Gamma_0,\ldots,\Gamma_4$ of the square snowflake.}
\label{fig:SquareSnowflakeShapesBlack}
\end{figure}

Each prefractal $\Gamma_j$ is an open polygon whose boundary is the union of $N_j:=4\cdot 8^j$ segments of length $\ell_j:=4^{-j}$ aligned to the Cartesian axes.
Let $\Gamma_0=(0,1)^2$ be the open unit square.
For $j\in\N$, $\partial\Gamma_j$ is constructed by replacing each horizontal edge and each vertical edge of $\partial\Gamma_{j-1}$ respectively by the following polygonal lines composed of 8 edges each:
\[
\begin{tikzpicture}[scale=.3]
\draw[very thick](-6,0)--(-2,0);
\draw[very thick](0,0)--(1,0)--(1,1)--(2,1)--(2,-1)--(3,-1)--(3,0)--(4,0);
\draw(-1,0)node{$\rightsquigarrow$};
\begin{scope}[xshift=15cm]
\draw[very thick](-3,-2)--(-3,2);
\draw[very thick](1,-2)--(1,-1)--(0,-1)--(0,0)--(2,0)--(2,1)--(1,1)--(1,2);
\draw(-1.5,0)node{$\rightsquigarrow$};
\end{scope}
\end{tikzpicture}
\]
(Note that the fourth and the fifth segments obtained are aligned; in the following however we count them as two different edges of $\Gamma_j$.)
Each polygonal path $\partial\Gamma_j$ constructed with this procedure is the boundary of a simply connected polygon $\Gamma_j$ of unit area, composed of $16^j$ squares of side length $\ell_j$. (See \cite[\S5.2]{caetano2018} for more detail of this construction.)
The resulting sequence of prefractals $\{\Gamma_j\}_{j\in\N_0}$ is not nested: for each $j\in\N$ neither $\Gamma_j\subset\Gamma_{j-1}$ nor $\Gamma_j\supset\Gamma_{j-1}$.
Indeed, the two set differences $\Gamma_j\setminus\Gamma_{j-1}$ and $\Gamma_{j-1}\setminus\Gamma_j$ are composed of $4\cdot 8^{j-1}=2^{3j-1}$ disjoint squares of side length $\ell_j$.
Thus the limit set of the sequence cannot be defined simply as a union or intersection of the prefractals.

In \cite{caetano2018} we showed how to construct inner and outer nested prefractal sequences $\Gamma_j^\pm$ such that $\Gamma_j$ and $\Gamma_j^\pm$ satisfy the assumptions of Proposition \ref{prop:Mosco}\rf{Moscoiii}, with the limit set $\Gamma$ defined as $\Gamma:=\bigcup_{j\in\N_0} \Gamma_j^- = \big(\bigcap_{j\in\N_0} \Gamma_j^+\big)^\circ$.
In \cite{caetano2018} we proved further that $\partial\Gamma$ is a $d$-set with Hausdorff dimension $d=3/2$, and that $\Gamma$ is a thick domain, so that $\widetilde{H}^{s}(\Gamma)=H_{\overline{\Gamma}}^{s}$ for all $s\in\R$.
Combining these facts with Theorem~\ref{thm:OpenPrime}, Proposition~\ref{prop:equiv}, Proposition~\ref{prop:Mosco}\rf{Moscoiii} and Theorem~\ref{prop:DiscreteOpen}, gives the following result.

\begin{prop}[Square snowflake]
\label{prop:SquareSn}
Define the square snowflake $\Gamma$ and its standard prefractals $\Gamma_j$ as above.
Then the BVPs $\sD(\Gamma)$ and $\sD(\Gamma_j)$ are well-posed, and BVP convergence holds. Furthermore, BEM convergence holds if $h_j\to 0$.
\end{prop}

\section{Numerical results}
\label{sec:Numerics}
In this section we present numerical results validating our theory, and demonstrate the feasibility of using BEM to calculate scattering by fractal screens.

While our theoretical convergence analysis in \S\S\ref{sec:Convergence}-\ref{sec:examples} is for Galerkin discretisations, the numerical results in this section were obtained using a collocation method, to make implementation as simple and flexible as possible.
Our Matlab collocation code was validated against our own 2D Galerkin code for the case of the Cantor set (see Figure \ref{fig:CantorGalerkinVsColl-FatVsThin} below) and the open-source 3D Galerkin software \mbox{Bempp} \cite{SBAPS15} for the case of the Sierpinski triangle.
In both cases, for fixed prefractal level the collocation code was found to give similar accuracy to the Galerkin codes (using the same meshes), but with a slightly lower computational cost, allowing us to reach slightly higher prefractal levels in 3D than was possible with the Galerkin code
(using default Bempp settings and dense linear algebra).

All our experiments are on prefractals that are finite unions of disjoint segments (when $n=2$) or finite unions of Lipschitz polygons (when $n=3$).
For simplicity we use uniform meshes throughout.
In fact, in each experiment the elements are either congruent segments (when $n=2$), or congruent squares,
or congruent equilateral triangles.%

\subsection{Collocation method} \label{sec:collocation}

Given a prefractal $\Gamma_j$ partitioned by a uniform mesh $M_j=\{T_{j,1},\ldots,T_{j,N_j}\}$ with mesh size $h_j$ (the diameter of each element of the uniform mesh),
our collocation discretisation of the BIE \eqref{eqn:SBIE} computes $\phijBEM\in \VjBEM$ (the $N_j$-dimensional space of piecewise constants on $M_j$) by solving the equations
$$ (S_{\Gamma_j}\phijBEM)(\bx_l)=-g_\dag(\bx_l),\qquad l=1,\ldots,N_j,$$
where $\bx_l$ is the centre of the element $T_{j,l}$.
This is equivalent to approximating the testing integrals in the Galerkin equations \eqref{eq:BEMdef} with a 1-point-per-element quadrature formula.

The vector $\bf u$ containing the values of $\phijBEM$ on each mesh element satisfies a square linear system $A\bf u=b$ where the matrix $A$ and right-hand side vector $\bf b$ have entries
$A_{l,m}=\int_{T_{j,m}}\Phi(\bx_l,\bx)\rd s(\bx)$ and $b_l=-g_\dag(\bx_l)$, respectively.
The integrals in the off-diagonal matrix entries are approximated with Gauss--Legendre quadrature on line segments and the tensorized version of the same rule on square elements; in both cases the number of
quadrature points per element is chosen to be at least $\max\{20 h_j/\lambda,3\}^{n-1}$, $\lambda=2\pi/k$ being the wavelength.
Numerical tests show that higher-order quadratures do not noticeably improve the solution accuracy for the range of parameters considered.
On triangular elements we use a classical 7-point symmetric formula (as in \cite[p.~415]{QSS07}, with degree of exactness 3). %
The integrands of the diagonal entries $A_{l,l}$ have a weak singularity at the element centre $\bx_l$.
For line segments, $A_{l,l}$ is computed by dividing the segment in half and applying a high-order Gauss--Legendre quadrature on each side of the singularity.
For square and equilateral triangle elements we split $T_{j,l}$ into 4 or 3 identical isosceles triangles respectively (with a common corner at the singularity), apply symmetry, and transform to polar coordinates, to evaluate $A_{l,l}$ as
$$
A_{l,l}=\int_{T_{j,l}}\frac{\re^{\ri k|\bx-\bx_l|}}{4\pi|\bx-\bx_l|}\rd \bx =
\begin{cases}
\displaystyle
\frac1{\pi\ri k}\int_{-\pi/4}^{\pi/4} (\re^{\frac{\ri k L}{2\cos\theta}}-1)\rd\theta
&\text{if $T_{j,l}$ is a square,}
\\
\displaystyle
\frac3{4\pi\ri k}\int_{-\pi/3}^{\pi/3} (\re^{\frac{\ri kL}{2\sqrt3\cos\theta}}-1)\rd\theta
&\text{if $T_{j,l}$ is a triangle,}
\end{cases}
$$
where $L$ is the element side length.
The integrals over $\theta$ are computed using Matlab's \verb|integral| function.
Since the mesh is uniform, all diagonal terms coincide and only one such computation is needed for a given value of $kL$.
The (dense, complex, non-Hermitian) linear systems in our numerical experiments are relatively small (with fewer than 11000 DOFs) and are solved with a direct solver (Matlab's backslash).

\subsection{Experiments performed}
We use our BEM code to compare the numerical solutions on a sequence of prefractal screens $\Gamma_j$ approximating a limiting fractal screen $\Gamma$, for the examples in \S\ref{sec:examples}.
In addition to showing domain plots of the scattered fields for different $\Gamma_j$, we study the $j$-dependence of the norm of the numerical solution using the three norms defined in Table \ref{t:PlotNotation}. The table also shows the marker type these norms will be represented by in all the plots.
To validate our theoretical convergence results we also compute near- and far-field errors for the solution on $\Gamma_j$, relative to the solution on the finest prefractal $\Gamma_{j_{\rm max}}$, using these same norms.  %
 In all tests we simulate scattering problems, and the incident field is a plane wave, so that $g_\dag(\bx)=-\re^{\ri k \bd\cdot\bx}$, $\bx\in \Gamma_j$, for some $\bd\in\mathbb S^{n-1}$, the incident wave direction.
In the convergence plots for compact screens, red continuous lines correspond to results for standard prefractals and blue dashed lines to results for thickened prefractals.

\begin{table}[htb]\centering
\normalsize
\begin{tabular}{|c|l|}
\hline
Marker & Norm\\
\hline
$\bigcirc$ &
\begin{minipage}{100mm}
\vspace{1mm}
The $\tH^{\mhalf}(\Gamma_j)$ norm on $\Gamma_j$,
computed
via an accurate numerical integration of the representation
 $\|\phijBEM\|_{\tH^{\mhalf}(\Gamma_j)}^2$ $=2\int_{\Gamma_j}(S_{\Gamma_j}^{\ri}\phijBEM)\overline\phijBEM\rd s$,
where $S_{\Gamma_j}^{\ri}$ is the single-layer operator for the %
equation $\Delta u-u=0$ (compare the definition of the norm in \S\ref{sec:SobolevSpaces} and \cite[eq.\ (3.5), (3.28)]{Ha-Du:90}).
\vspace{1mm}
\end{minipage}
\\
\hline

$\square$ &
\begin{minipage}{100mm}
\vspace{1mm}
The $L^2(Box)$ norm in a near-field region, computed over the ``box'' used for the domain plots. %
For the Cantor set this ``box'' is the square $(-1,2)\times(-1.5,1.5)$.
For all other examples it comprises
three perpendicular faces of the cuboid $(-1,2)\times (-1,2)\times (-1,1)$.
See e.g.\ Figures \ref{fig:CantorFatThinFields} and \ref{fig:CantorDustFields}.
\vspace{1mm}
\end{minipage} \\
\hline
$*$&
\begin{minipage}{100mm}
\vspace{1mm}
The $L^2(\mathbb S^{n-1})$ norm on
$\mathbb S^{n-1}:=\{\hat\bx\in\R^n, |\hat\bx|=1\}$, the unit sphere. In particular we compute this quantity for the far-field pattern $u_{j,\infty}$ of the field scattered from $\Gamma_j$, defined by
$u_{j,\infty}(\hat\bx):=\frac{-\ri k^{(n-3)/2}}{2(2\pi\ri)^{(n-1)/2}}\int_{\Gamma_j}\re^{-\ri k\hat\bx\cdot\by}\phijBEM(\by)\rd s(\by)$, for $\hat\bx\in  \mathbb S^{n-1}$ \cite[(2.13), (3.64)]{CoKr:98}. %
This norm is proportional to the square root of the total acoustic power flux in the scattered field.
\vspace{1mm}
\end{minipage}\\
\hline
\end{tabular}
\caption{The graphical conventions used in the plots.}
\label{t:PlotNotation}
\end{table}

\subsection{Cantor set}
\label{sec:NumCantorSet}
We first fix $\Gamma$ to be the standard Cantor set, as defined in \S\ref{sec:CantorSet}, with $\alpha=1/3$,
set $k=30$ (so that the acoustic wavelength $\lambda:=2\pi/k\approx0.209$ and there are roughly 5 wavelengths across $\Gamma$),
and choose the direction of the incident plane wave as $\bd=(1/2,-\sqrt3/2)$.%

We make BEM computations on both the standard (open) prefractals $\Gamma_j$, defined by \eqref{eq:prefract} with $O=(0,1)$,
which have $2^j$ components of length $3^{-j}$ (cf.~\S\ref{sec:CantorSet}),
and the thickened prefractals as defined in \S\ref{sec:CantorSet} with $\delta=\frac1{4\alpha}-\frac12=\frac14$, which we denote by $\GammajFat$, and which have $2^j$ components of length $\frac32 3^{-j}$.
We present results for the simplest case where the BEM meshes have exactly one element per component of each prefractal, so that we are using the convex mesh \eqref{eq:Mj} with $N_0=1$. Thus there are
$N_j=2^j$ elements and DOFs on the $j$th prefractal. For these simple meshes (Galerkin) BEM convergence is guaranteed for the thickened prefractals by Proposition \ref{prop:CantorSet}.

Figure \ref{fig:CantorFatThinFields} shows the real part and magnitude of the scattered and total fields on the box $(-1,2)\times(-1.5,1.5)$, computed for prefractal level $j=13$, discretising $\Gamma^\delta_{13}$ with $N_{13}=2^{13}=8192$ elements and DOFs.

\begin{figure}[t!]\centering
\includegraphics[width=.24\textwidth,clip, trim=70 10 70 10]{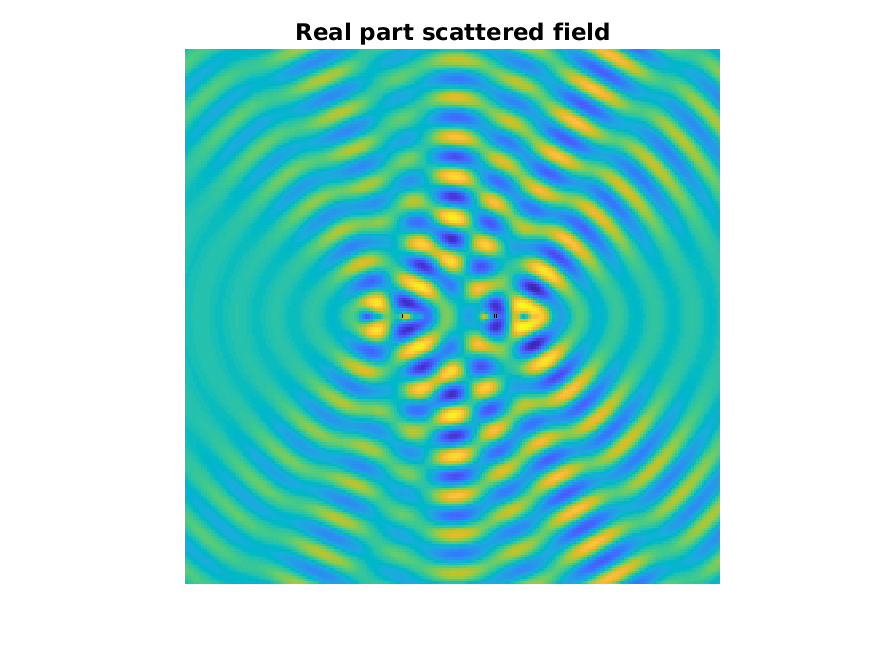}
\includegraphics[width=.24\textwidth,clip, trim=70 10 70 10]{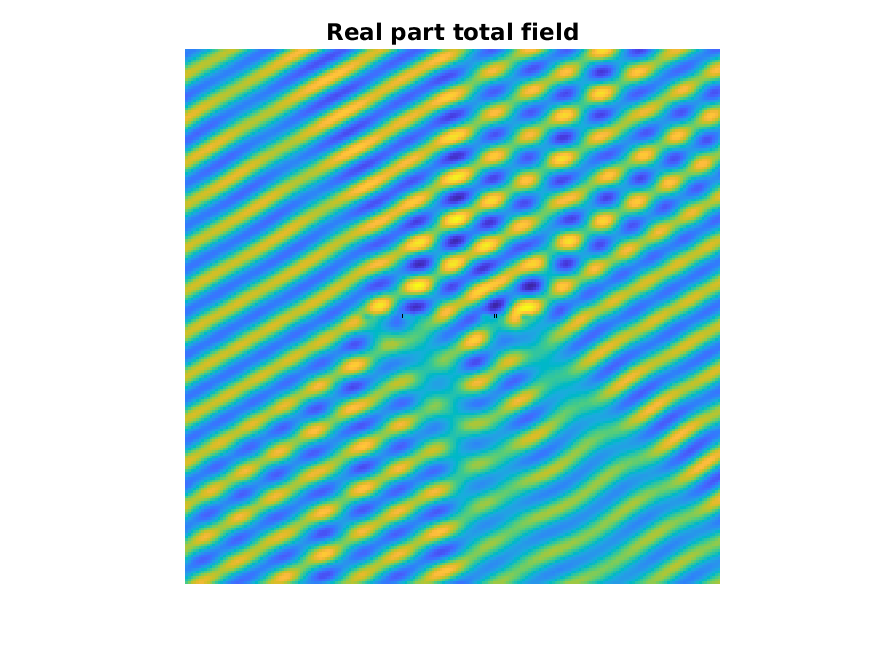}
\includegraphics[width=.24\textwidth,clip, trim=70 10 70 10]{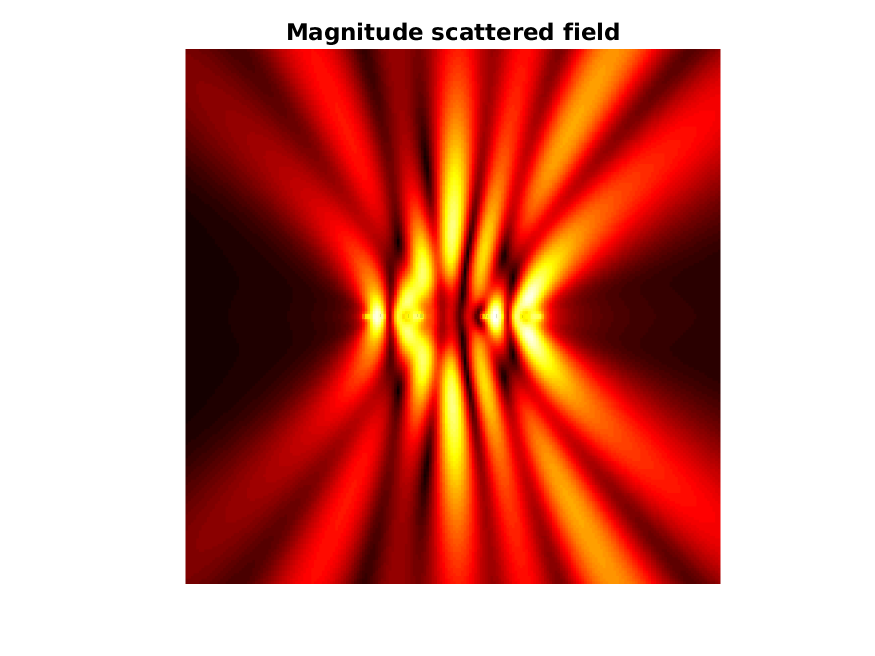}
\includegraphics[width=.24\textwidth,clip, trim=70 10 70 10]{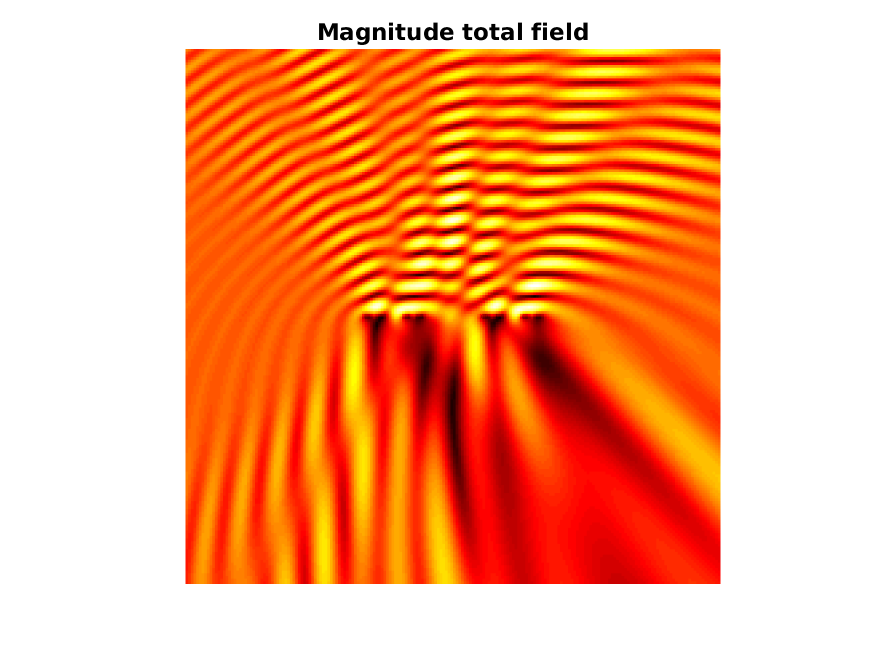}
\caption{The real part and magnitude of the scattered and total fields on the box $(-1,2)\times(-1.5,1.5)$ for the Cantor set problem in \S\ref{sec:NumCantorSet}, %
approximating $\Gamma$ by the level $13$ thickened prefractal $\Gamma^\delta_{13}$ and using $N_{13}=8192$ DOFs. %
}
\label{fig:CantorFatThinFields}
\end{figure}
\begin{figure}[t!]\centering
\includegraphics[width=.49\textwidth,clip,trim=30 0 30 10]{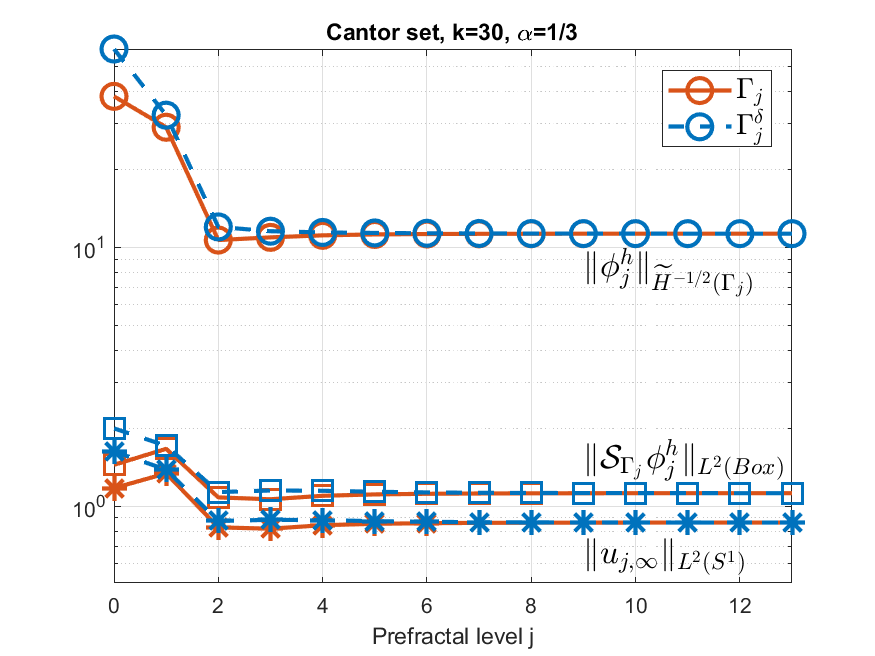}
\includegraphics[width=.49\textwidth,clip,trim=30 0 30 10]{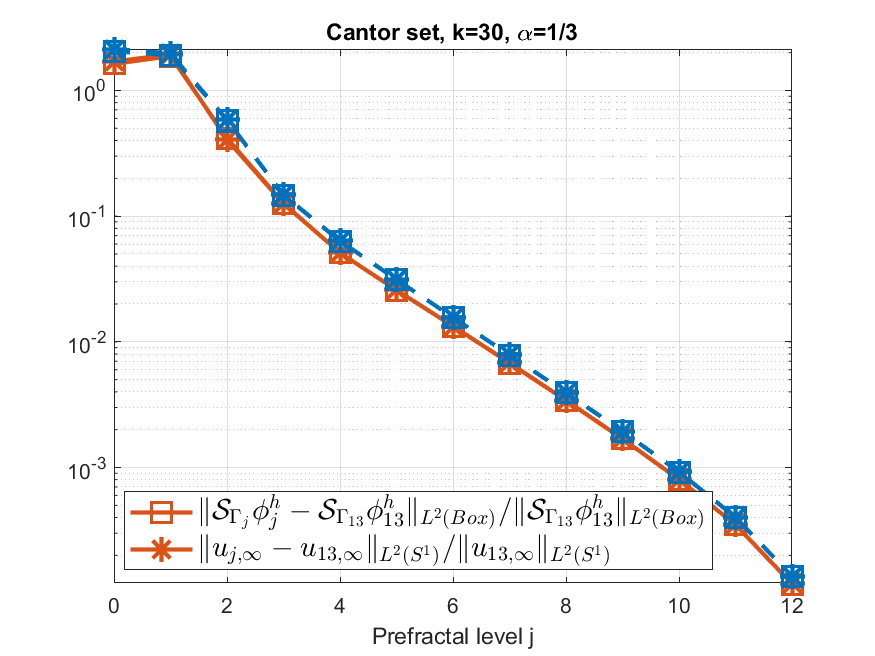}
\caption{Numerical results for the Cantor set problem in \S\ref{sec:NumCantorSet} for prefractal levels $0$ to $13$.
Left: the convergence of the boundary, near- and far-field norms of the discrete solutions.
Right: the exponential decay (in $j$) of the near- and far-field relative errors between the solutions on the $j$th prefractal and the $13$th prefractal.
Continuous red lines correspond to standard prefractals and dashed blue lines show the same quantities for thickened prefractals.
}
\label{fig:CantorAllInOne}
\end{figure}

The left panel in Figure~\ref{fig:CantorAllInOne} shows the norms, as defined in Table~\ref{t:PlotNotation}, of the (collocation) BEM solution on $\Gamma_j$ and $\GammajFat$ for $j=0,\ldots,13$.
In all cases the norms quickly settle to an approximately constant value, suggesting that a limiting value as $j\to\infty$ has been reached.
The right panel in Figure~\ref{fig:CantorAllInOne} shows the near- and far-field relative errors for $j=0,\ldots,12$, measured against the solutions for $j=13$.
Standard prefractals seem to give slightly smaller errors than thickened ones. But in both cases the relative errors decay exponentially in $j$, %
 this numerical evidence of collocation BEM convergence, both for standard and thickened prefractals.

\begin{figure}[htb]\centering
\includegraphics[width=.495\textwidth,clip,trim=30 0 20 10]{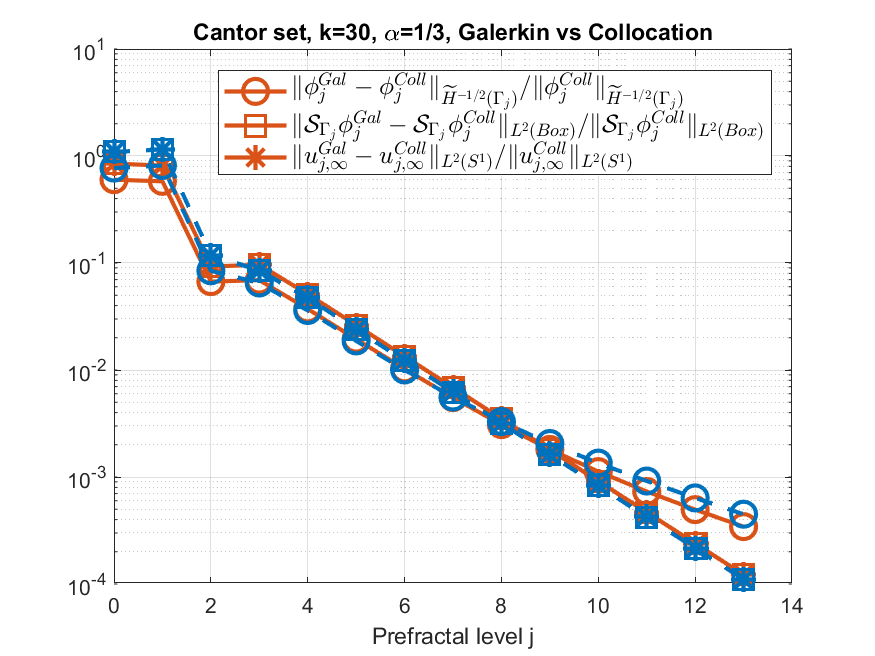}
\includegraphics[width=.495\textwidth,clip,trim=30 0 20 0]{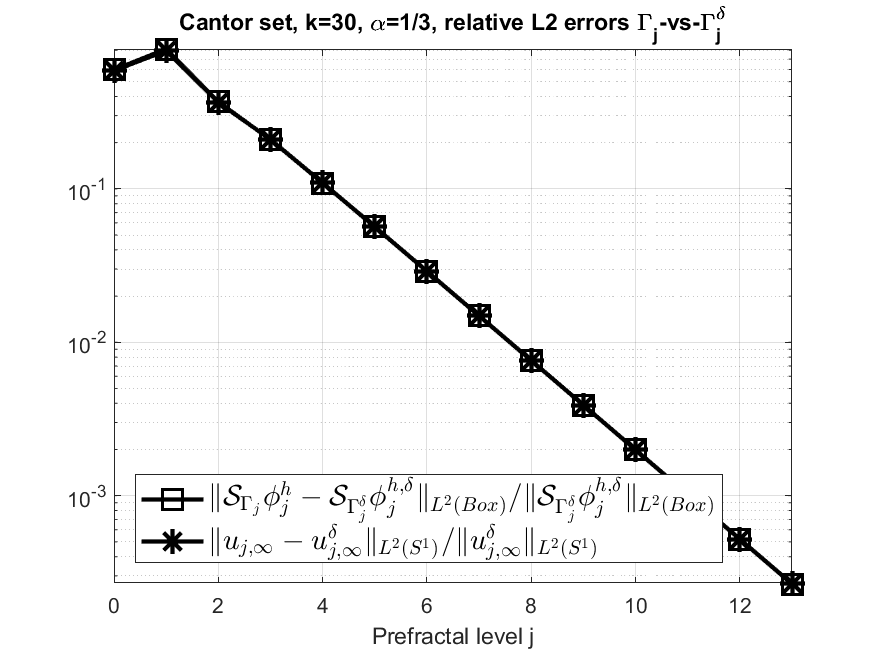}
\caption{Left: validation of our collocation code against a Galerkin implementation for the Cantor set problem in \S\ref{sec:NumCantorSet}. Continuous red lines correspond to standard prefractals and dashed blue lines show the same quantities for thickened prefractals (see Table \ref{t:PlotNotation}). Note that there are 6 lines in total on this graph, but the near-field and far-field relative errors are almost indistinguishable.
Right: exponential decay of the relative difference between the fields scattered by the standard ($\Gamma_j$) and thickened ($\GammajFat$) prefractals for the Cantor set problem in \S\ref{sec:NumCantorSet}.
Here $\phi_j^{h,\delta}$ and $u_{j,\infty}^\delta$ denote the BEM solution on $\Gamma_j^\delta$ and the corresponding far-field pattern.
}
\label{fig:CantorGalerkinVsColl-FatVsThin}
\end{figure}

For this specific 2D problem we have also implemented a Galerkin BEM.
The left panel of Figure \ref{fig:CantorGalerkinVsColl-FatVsThin} demonstrates the close agreement between our collocation solutions and the corresponding Galerkin solution (to which Proposition \ref{prop:CantorSet} applies to prove convergence in the thickened prefractal case). Taken together, since we know from Proposition \ref{prop:CantorSet} that the Galerkin solution on the thickened prefractal sequence converges to the correct limiting solution of the BIE on the Cantor set $\Gamma$, Figures~\ref{fig:CantorAllInOne} and \ref{fig:CantorGalerkinVsColl-FatVsThin} are persuasive numerical evidence that: i) the Galerkin solution on the standard prefractal sequence; and ii) the collocation solutions on both the standard and the thickened prefractal sequences, are all converging to the correct limiting solution for the Cantor set $\Gamma$ as $j\to\infty$.
These conclusions are further supported by the right panel in Figure~\ref{fig:CantorGalerkinVsColl-FatVsThin} which shows that the near and far-field relative differences between the fields computed on the standard and thickened prefractals (using collocation BEM) also decrease exponentially in $j$.
Figure~\ref{fig:CantorAlphaVar} shows how norms of the Cantor set solution, approximated by (collocation BEM) computations on the finest prefractal level, vary as a function of the Cantor set parameter $\alpha$.  %
Recall that $\alpha$ is related to the Hausdorff dimension of the Cantor set $\Gamma$ by $d=\log2/\log(1/\alpha)$.
The strength of the scattered field decreases with decreasing $\alpha$ (decreasing $d$). This is consistent with our earlier theory.
Specifically, let $(\alpha_\ell)_{\ell\in \N_0}\subset (0,1/2)$ be any decreasing sequence such that $\alpha_\ell\to 0$, and let $\Gamma_\ell^C$ denote the corresponding sequence of Cantor sets, and let $\Gamma_\ell^{C,+}:= \bigcup_{i\geq \ell} \Gamma_i^C$. Then, since $\Gamma^{C,+}_{\ell+1}\subset \Gamma^{C,+}_\ell$, it follows from Theorem \ref{thm:MoscoConv} and Proposition \ref{prop:Mosco}(ii) that BVP convergence holds, specifically that the solution for $\Gamma_\ell^{C,+}$ converges to that for $\Gamma^C:= \bigcap_{\ell\in \N_0} \Gamma_\ell^{C,+}$ as $\ell\to\infty$. But $\Gamma^C=\{0,1\}$, since $\{0,1\}\subset \Gamma_\ell^{C,+} \subset [0,\alpha_\ell]\cup [1-\alpha_\ell,1]$ for each $\ell$. Thus, by Lemma \ref{lem:zero} and Proposition \ref{prop:Scatters}, $H^{-1/2}_{\Gamma^C}=\{0\}$ and the field scattered by $\Gamma^C$ is zero. Further, by Lemma~\ref{lem:Sandwich}, since $\Gamma_\ell^C \subset \Gamma_\ell^{C,+}$, the solution for $\Gamma_\ell^C$ also converges to zero as $\ell\to\infty$.

\begin{figure}[htb]\centering
\includegraphics[width=.6\textwidth,clip,trim=30 5 30 0]{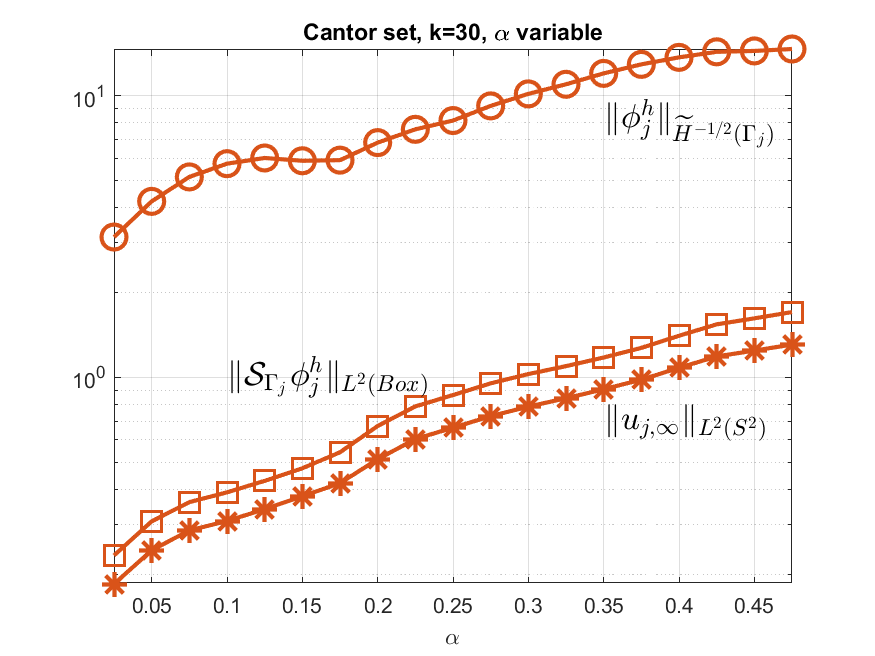}
\caption{
Solution norms for the finest standard prefractal $\Gamma_{13}$, plotted against the Cantor set parameter $\alpha$.
}
\label{fig:CantorAlphaVar}%
\end{figure}

\subsection{Cantor dust}\label{sec:NumCantorDust}

We now make computations for two Cantor dusts, as defined in \S\ref{sec:CantorDust}, with $\alpha=1/3$ and $\alpha=1/10$ respectively, setting
$k=50$ (so that $\lambda\approx0.126$ and there are roughly 11 wavelengths across the diagonal of $\Gamma$),
and choosing $\bd=(0,\frac1{\sqrt2},-\frac1{\sqrt2})$.

Similarly to \S\ref{sec:NumCantorSet} we make (collocation) BEM computations on  both the standard (open) prefractals $\Gamma_j$, defined by \eqref{eq:prefract} with $O=(0,1)^2$
(cf.~\S\ref{sec:CantorDust}),
and the thickened prefractals as defined in \S\ref{sec:CantorDust} with $\delta=\frac14$, which we denote by $\GammajFat$.
As in \S\ref{sec:NumCantorSet} we use BEM meshes with exactly one element per component of each prefractal, giving $N_j=4^j$ DOFs in total. It follows from Proposition \ref{prop:CantorDust} that (Galerkin) BEM convergence is guaranteed for the standard prefractals for $\alpha=1/10$ (since the limiting solution is zero), and for the thickened prefractals for both $\alpha=1/3$ and $\alpha=1/10$.

Figure \ref{fig:CantorDustFields} shows near- and far-field plots of the scattered field for the standard prefractal of level $j=6$ with $N_{6}=4^{6}=4096$ DOFs.

\begin{figure}[t!]\centering
\includegraphics[width=.33\textwidth,clip,trim=40 10 40 10]{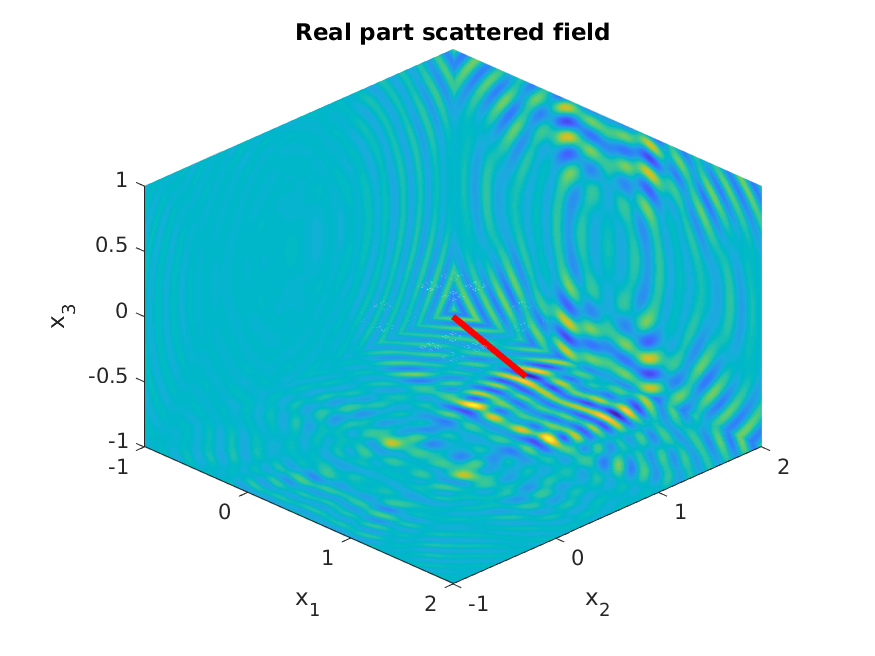}
\includegraphics[width=.33\textwidth,clip,trim=40 10 40 10]{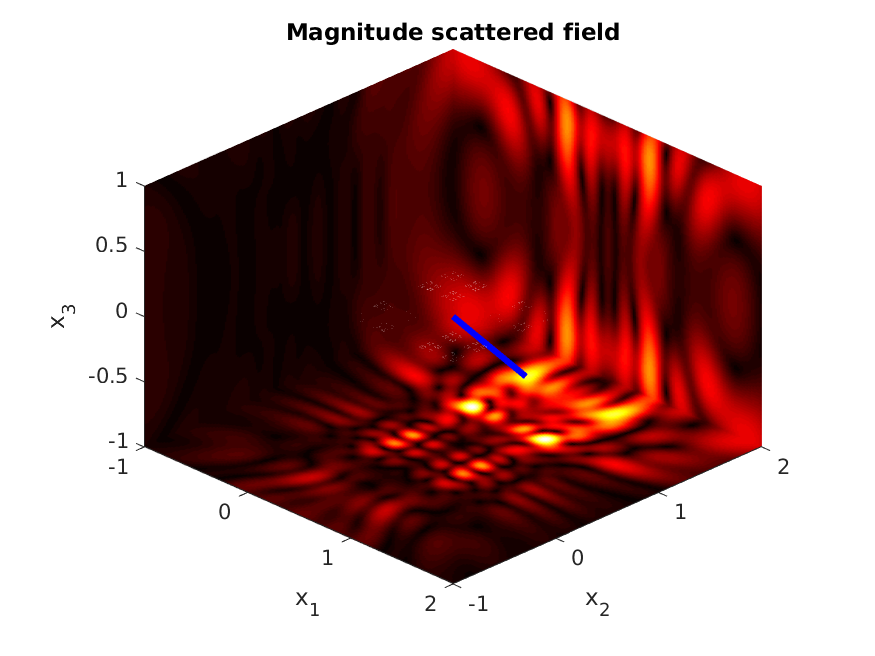}
\includegraphics[width=.32\textwidth,clip,trim=70 30 70 10]{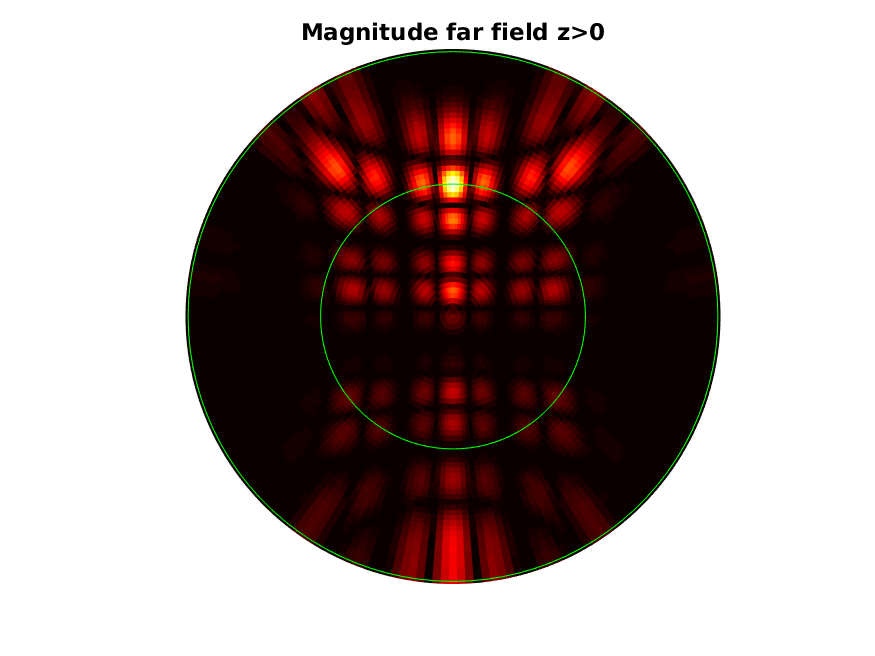}
\caption{The scattered field for the Cantor dust problem in \S\ref{sec:NumCantorDust} with $\alpha=1/3$, computed on $\Gamma_6$ with $N_6=4096$ DOFs.
Left and centre: the real part and magnitude of the scattered field on three faces of the box $(-1,2)\times(-1,2)\times(-1,1)$ (with $\Gamma\subset (0,1)^2\times\{0\}$).
The red/blue segment denotes the direction of the incoming wave.
Right: the magnitude of the far-field pattern $u_{j,\infty}$ on the upper half-sphere $\mathbb S^{2}\cap\{x_3>0\}$; the angular coordinate is the longitude, the radial coordinate the colatitude
(the green circle through the bright spot is the $\pi/4$ parallel).
}
\label{fig:CantorDustFields}
\end{figure}
\begin{figure}[t!]\centering
\includegraphics[width=.49\textwidth,clip,trim=30 0 30 0]{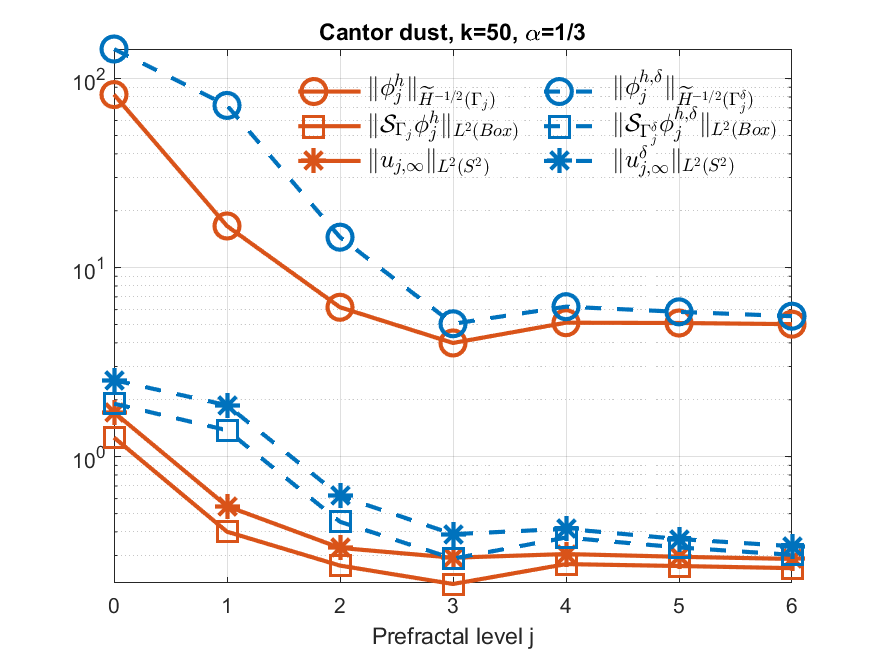}
\includegraphics[width=.49\textwidth,clip,trim=30 0 30 0]{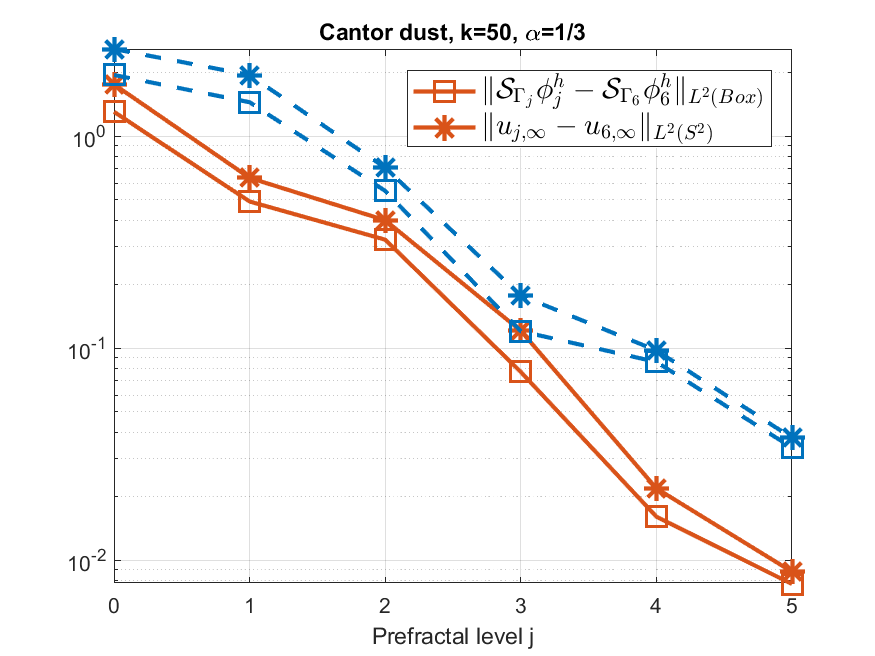}
\includegraphics[width=.49\textwidth,clip,trim=30 0 30 0]{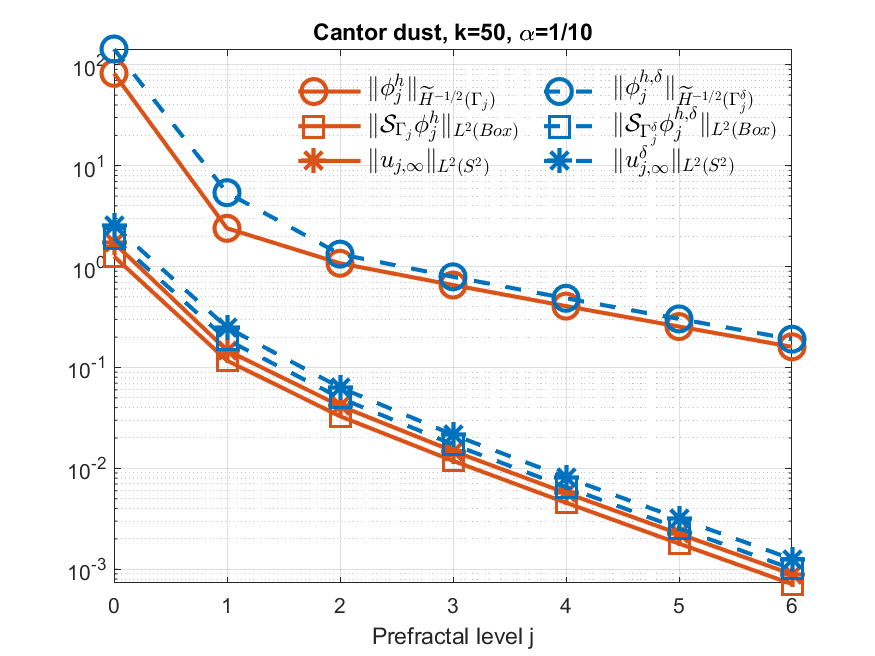}
\includegraphics[width=.49\textwidth,clip,trim=30 0 30 0]{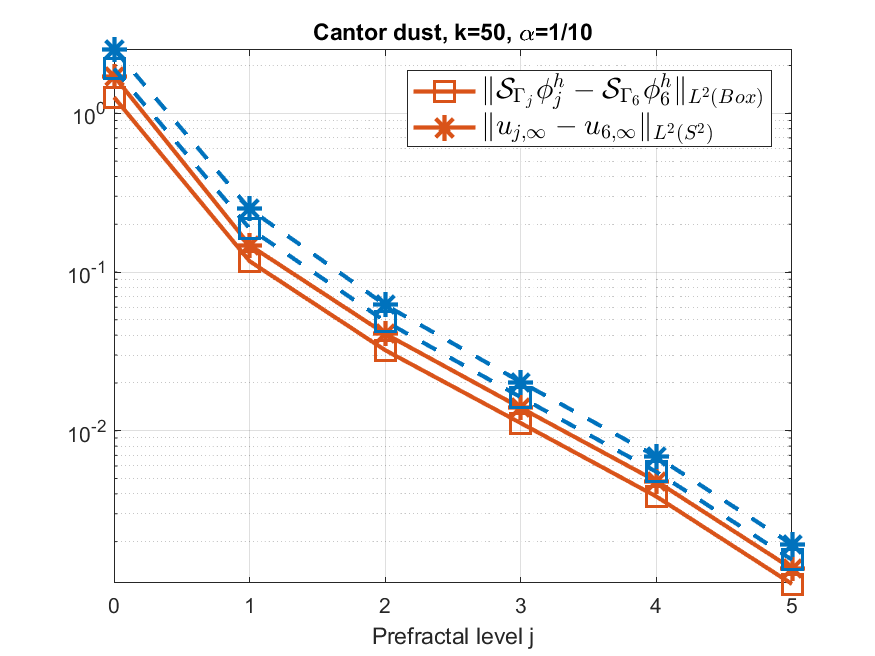}
\caption{Analogue of Figure \ref{fig:CantorAllInOne} for the Cantor dust problem in \S\ref{sec:NumCantorDust}, except that in the right column we show absolute errors.
Top row: $\alpha=1/3$, bottom row: $\alpha=1/10$.
From the left plots we see that, when the prefractal is refined, the solution norms converge to positive values for $\alpha=1/3$ and to $0$ for $\alpha=1/10$, in agreement with theory. }
\label{fig:CantorDust}
\end{figure}

Figure~\ref{fig:CantorDust} shows the solution norms for prefractal levels 0 to 6 and the relative errors against the computations on the finest prefractal.
From the left panels we see that for $\alpha=1/3$ the norms appear to converge to a constant positive value, while for $\alpha=1/10$ they converge (exponentially) to 0, consistent with Proposition~\ref{prop:CantorDust}.
The superior convergence rate in the near- and far-field $L^2$ norms compared to the $\tH^\mhalf$ energy norm, visible for $\alpha=1/10$, is in line with standard superconvergence theory for functionals of a BEM solution---see e.g.\ \cite[\S4.2.5]{sauter-schwab11}.

In the right panels of Figure~\ref{fig:CantorDust} we observe the exponential (in $j$) decay of the errors of near- and far-fields against the solutions on the finest prefractal.
 We have also computed (but do not plot) the differences between standard and thickened prefractals in the near- and far-fields. These behave similarly to those in Figure~\ref{fig:CantorGalerkinVsColl-FatVsThin} (for $n=2$). %
These various numerical experiments, together with validations we have made (see the beginning of \S\ref{sec:Numerics}) of our collocation code against 3D Galerkin code, and the theoretical (Galerkin) BEM convergence results of Proposition \ref{prop:CantorDust}, provide persuasive evidence (cf.~the penultimate paragraph of \S\ref{sec:NumCantorSet}) that our collocation BEM results are converging as $j\to\infty$ to the correct limiting solution for scattering by the Cantor dust $\Gamma$.

\begin{figure}[t!]\centering
\includegraphics[width=.49\textwidth,clip,trim=30 0 30 0]{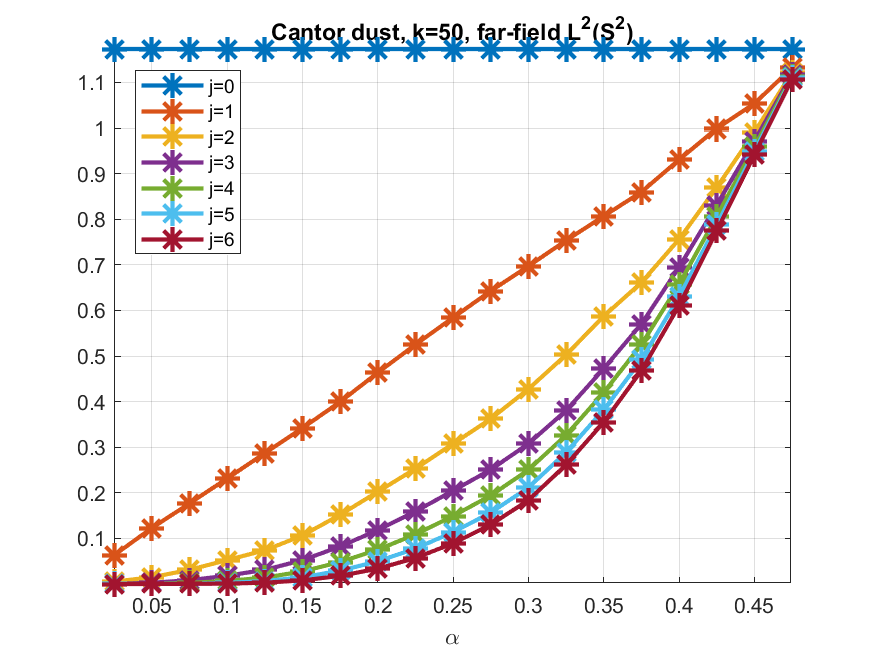}
\includegraphics[width=.49\textwidth,clip,trim=30 0 30 0]{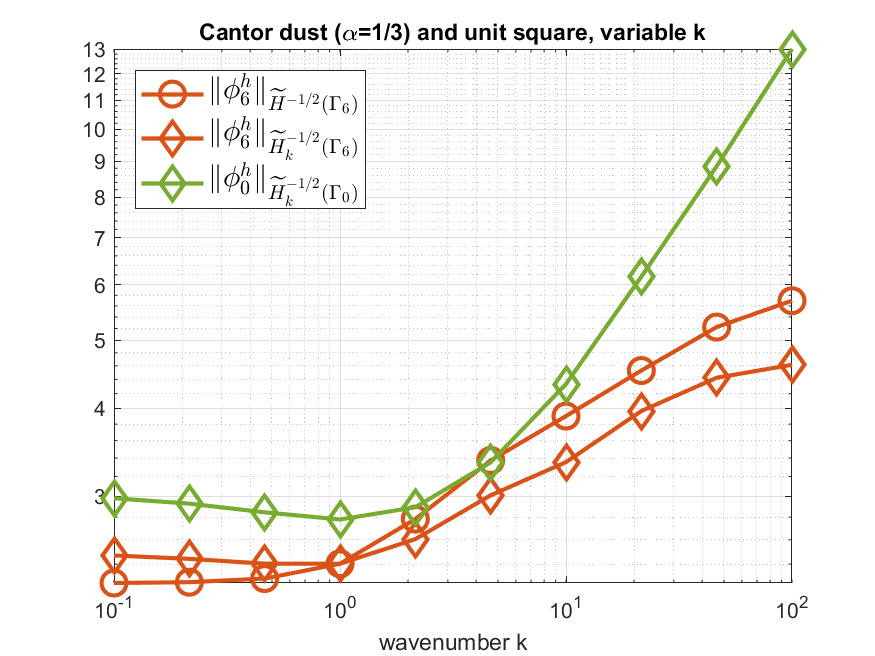}
\caption{Left: the $L^2(\mathbb S^2)$ norms of the far-field pattern for the Cantor dust, plotted against the parameter $\alpha$, %
for standard prefractals of level $j=0,1,\ldots,6$, computed with $N_j=4096$ DOFs for each $j$.
Proposition~\ref{prop:CantorDust} implies that the limit for $j\to\infty$ is $0$ for $\alpha\le1/4$.
Right: solution norms for $\alpha=1/3$ and $j=6$ (in red) and $j=0$ (unit square, in green)
as a function of the wavenumber $k$.%
}
\label{fig:DustAlphaAndK}
\end{figure}

The left panel of
Figure~\ref{fig:DustAlphaAndK} shows how the magnitude of the %
scattered field depends on the parameter $\alpha$, and thus on the Hausdorff dimension $d\!=\!\log{4}/\log(1/\alpha)$, for different prefractal levels $j$. Note that in this experiment we used a fixed total number $N_j=4096$ of DOFs on each prefractal, so that the lower order prefractal solutions are computed more accurately than they would be using our usual prescription $N_j=4^{j}$.
We recall (Proposition \ref{prop:CantorDust}) that in the limit $j\to\infty$ the field should vanish for $\alpha\le1/4$; compare this to the right panel of Figure~\ref{fig:CantorAlphaVar} for the Cantor set ($n=2$), where the limit is non-zero for all $\alpha$.

The right panel of Figure~\ref{fig:DustAlphaAndK} shows the dependence on the wavenumber $k$ of $\|\phi_j^h\|_{\widetilde{H}^{-1/2}(\Gamma_j)}$, for  the largest prefractal level $j=6$ that approximates the fractal limit. We find that $\|\phi_6^h\|_{\widetilde{H}^{-1/2}(\Gamma_6)}$ grows with increasing $k$ like $k^{0.19}$ for the larger values of $k$. In the same panel we also plot $\|\phi_6^h\|_{\widetilde{H}_k^{-1/2}(\Gamma_6)}$ and $\|\phi_0^h\|_{\widetilde{H}_k^{-1/2}(\Gamma_0)}$ against $k$ ($\phi_0^h$ the numerical solution on the screen $\Gamma_0$, i.e. the solution for a unit square screen, this computed with 10000 DOFs corresponding to more than 6 DOFs/wavelength at the highest wavenumber $k=100$). Here  $\|\cdot\|_{\widetilde{H}_k^{-1/2}(\Gamma_j)}$ is a wavenumber-dependent norm on   $\widetilde{H}^{-1/2}(\Gamma_j)$ commonly used in high frequency analysis (see the discussion in \cite[\S2.1]{CoercScreen2}), given by
$$
\|\psi\|^2_{\widetilde{H}_k^{-1/2}(\Gamma_j)} = \int_{\R^{n-1}}(k^2+|\bxi|^2)^{-1/2}|\widehat \psi(\bxi)|^2\, \rd \bxi, \quad \psi\in \widetilde{H}^{-1/2}(\Gamma_j).
$$
Clearly $\|\cdot\|_{\widetilde{H}_k^{-1/2}(\Gamma_j)}$ coincides with the standard norm for $k=1$ and is equivalent to it for any fixed $k$, precisely \cite[Eqn.~(28)]{CoercScreen2}
$$
\min\left(1,k^{-1/2}\right) \|\psi\|_{\widetilde{H}^{-1/2}(\Gamma_j)} \leq \|\psi\|_{\widetilde{H}_k^{-1/2}(\Gamma_j)} \leq \max\left(1,k^{-1/2}\right) \|\psi\|^2_{\widetilde{H}^{-1/2}(\Gamma_j)},
$$
for $\psi\in \widetilde{H}^{-1/2}(\Gamma_j)$.

Plotting these wavenumber-dependent norms enables comparison with the theoretical bounds in \cite{CoercScreen2} that are expressed in terms of these norms. It follows from the coercivity and inhomogeneous term estimates in \cite[Thms.~1.7, Lemma 6.1(i), (27)]{CoercScreen2} that, for every value of $j$, $\|\phi_j\|_{\widetilde{H}_k^{-1/2}(\Gamma_j)}$ ($\phi_j$ the exact solution for the screen $\Gamma_j$) grows with increasing $k$ at a rate no faster than $k^{1/2}$.
In the results in the right panel $\|\phi_0^h\|_{\widetilde{H}^{-1/2}_{k}(\Gamma_0)}$, the solution for a unit square, grows for larger $k$ at a rate $k^{0.48}$, while $\|\phi_6^h\|_{\widetilde{H}^{-1/2}_{k}(\Gamma_6)}$, the solution approximating the fractal limit, grows at a much more modest rate $k^{0.15}$. This suggests that the upper bounds in \cite{CoercScreen2}, which depend only on $k$ and on the diameter of the screen, are sharp for large $k$ in terms of their dependence on screen geometry for a regular screen, but are not sharp for screens with fractal dimension $<n-1$.
\subsection{Sierpinski triangle}\label{sec:NumSierpinski}

We approximate the Sierpinski triangle with the standard prefractals $\Gamma_j$ described in \S\ref{sec:Sierpinski} (to be precise we mesh $\Gamma_j^\circ$, the interior of $\Gamma_j$).
We set $k=40$, so that $\lambda\approx0.157$ and the diameter of $\Gamma$ is approximately 6.4 wavelengths, and consider a downward-pointing incoming wave with $\bd=(0,0,-1)$.

Figure~\ref{fig:SierpinskiFields} shows the near field and the magnitude of the (collocation) BEM solution $\phijBEM$ for prefractal level $j=8$ and $N_8=6561$ DOFs (one per component of $\Gamma_j^\circ$).
We note that $|\phijBEM|$ achieves its maxima at the midpoints of the sides of the triangular holes of side length $1/8$, this size comparable with the wavelength $\lambda$.

\begin{figure}[t!]\centering
\includegraphics[width=.33\textwidth,clip,trim=40 20 40 10]{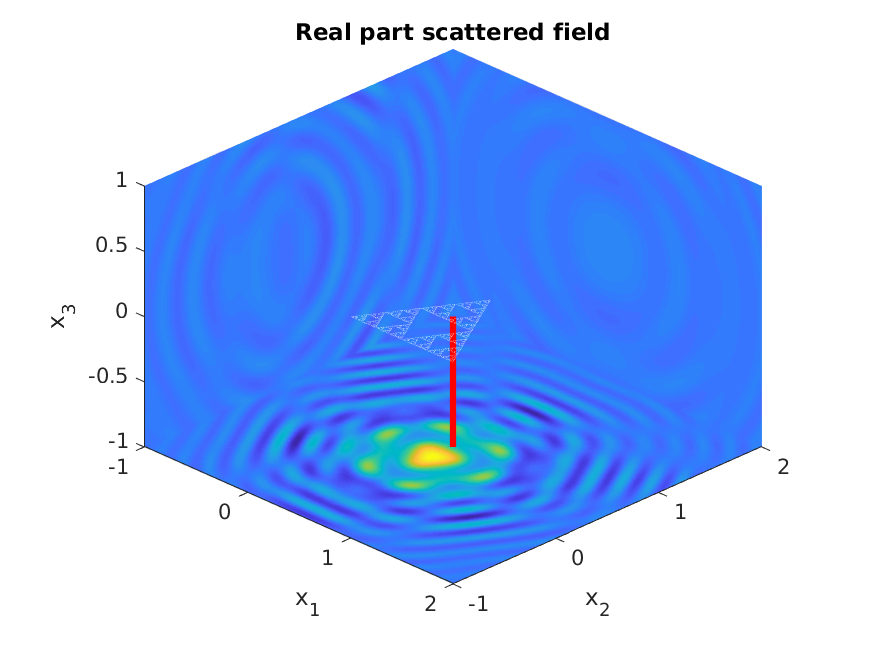}
\includegraphics[width=.33\textwidth,clip,trim=40 20 40 10]{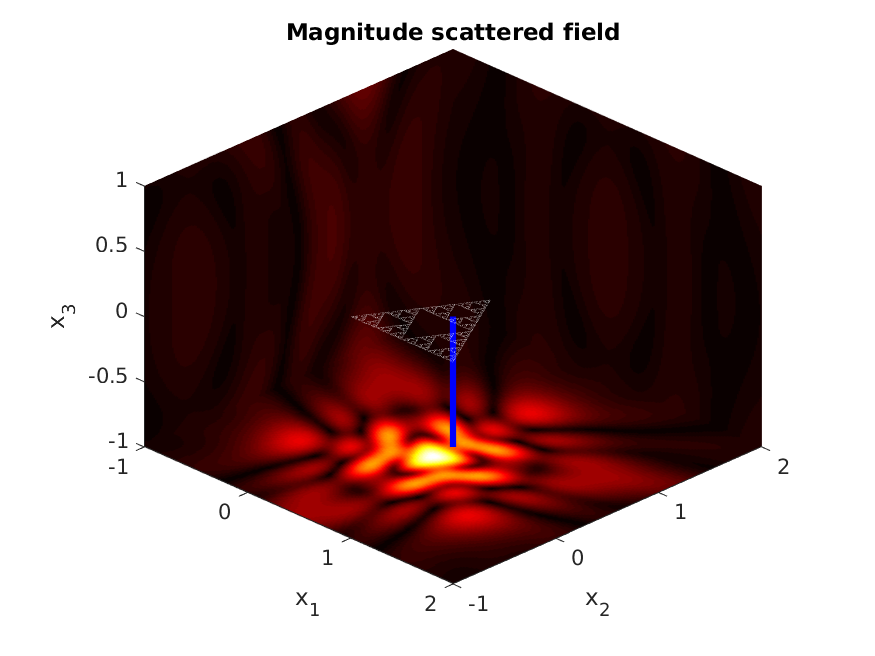}
\includegraphics[width=.32\textwidth,clip,trim=50 30 50 10]{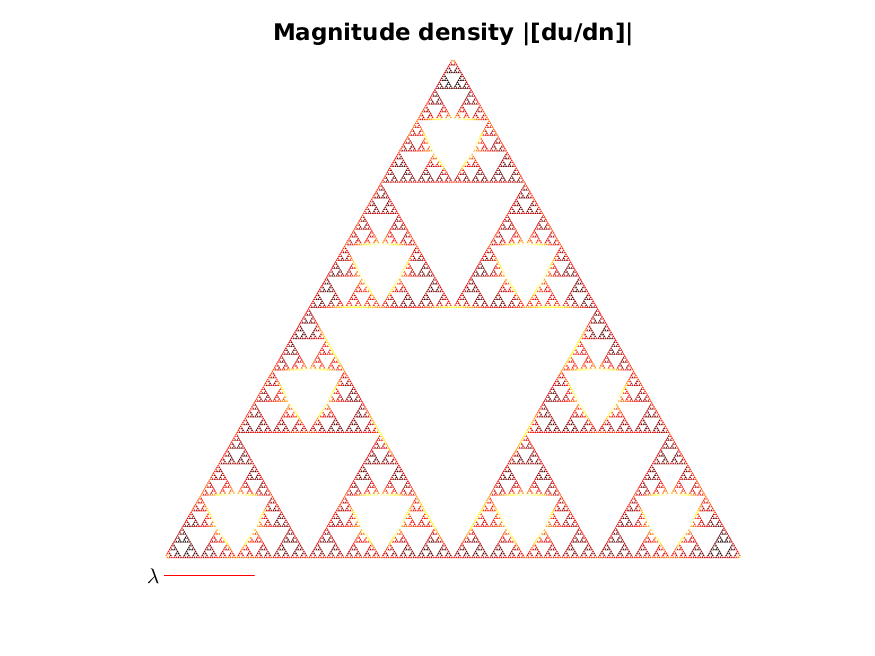}
\caption{Left and centre: the real part and magnitude of the field scattered by the level 8 prefractal approximation of the Sierpinski triangle, for the problem in \S\ref{sec:NumSierpinski}, computed with $N_8=3^8=6561$ DOFs.
Right: the magnitude $|\phi_j^h|$ of the piecewise constant BEM solution. %
Note that the peaks in $|\phi_j^h|$ (shown in yellow) are located close to the midpoints of the sides of the 9 triangular holes of size comparable to the wavelength (red segment).}
\label{fig:SierpinskiFields}
\end{figure}
\begin{figure}[t!]\centering
\includegraphics[width=.49\textwidth,clip,trim=30 0 30 0]{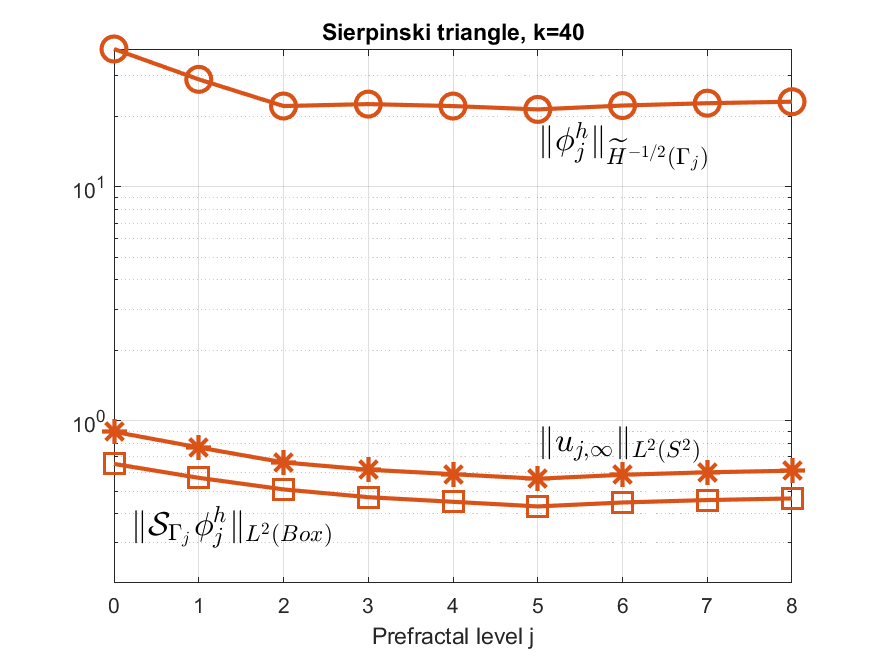}
\includegraphics[width=.49\textwidth,clip,trim=30 0 30 0]{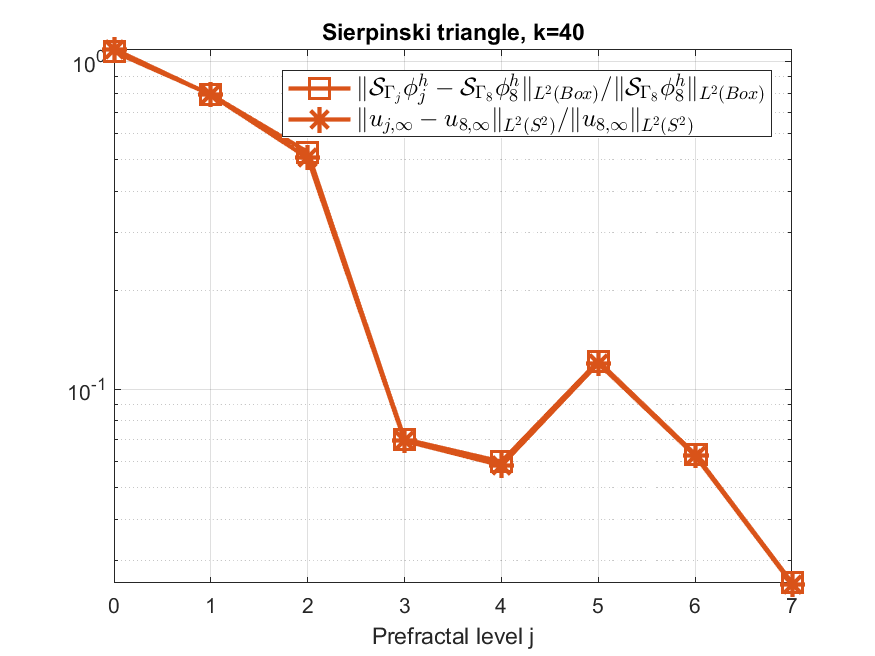}
\caption{Analogue of Figures \ref{fig:CantorAllInOne} and \ref{fig:CantorDust} for the Sierpinski triangle, for the problem of \S\ref{sec:NumSierpinski}, showing results for standard (non-thickened) prefractals only.
}
\label{fig:SierpinskiNorms}
\end{figure}

Figure~\ref{fig:SierpinskiNorms} (left panel) shows the solution norms for prefractal levels 0 to 8.
In these computations we vary from our prescription in the results above of one DOF per component of the prefractal, using a larger number of elements at the lower prefractal levels to ensure that quadrature errors in the evaluation of the coefficients in the linear system are not significant with the 7-point quadrature rule that we use on triangular elements (see \S\ref{sec:collocation}).
Precisely, for prefractals $\Gamma_0$ to $\Gamma_5$ we use a mesh of equilateral triangles of side $h_j=2^{-5}$.
For $\Gamma_5$ to $\Gamma_8$ we use equilateral triangles of size $h_j=2^{-j}$, i.e.\ with one DOF per component of $\Gamma_j^\circ$.
The numbers of DOFs for $j=0,\ldots,8$ are $N_0=1024, N_1=768, N_2=576, N_3=432, N_4=324, N_5=243, N_6=729, N_7=2187, N_8=6561$.
As in previous examples (Figure \ref{fig:CantorAllInOne} and Figure \ref{fig:CantorDust} for $\alpha=1/3$), we observe in the left panel rapid convergence of the norms to positive constant values.
In the right panel we plot relative errors compared to the result for the finest prefractal $\Gamma_8$. As in the right panel of Figure \ref{fig:CantorDust} we see convergence which is exponential in $j$ for $j\geq 5$. But the convergence is not monotonic in $j$ for $j\leq 5$ where we fix $h_j$ as we increase $j$.
\subsection{Classical snowflakes}\label{sec:NumSnowflake}

We now turn to examples in which the limiting screen $\Gamma$ is a bounded open set with fractal boundary. In these cases, as $j\to\infty$ the area of the prefractals must tend to the (non-zero) area of the limiting screen. Thus, in our simulations based on uniform meshes, the cost increase associated with the increasingly fine mesh required to represent the prefractal boundary exactly as $j\to\infty$ will not be compensated by a decrease in the area to be meshed, as was the case for the Cantor sets/dusts and the Sierpinski triangle. This in turn means that our simulations based on uniform meshes and direct solvers will be more expensive than the numerical experiments reported above, limiting the prefractal level that can be attained. More efficient BEM approaches could be developed for these problems using appropriate non-uniform meshes such as those in \cite{Bagnerini13,cefalo2014optimal,lancia2012numerical}, and/or fast iterative matrix-free linear solvers. But for brevity we leave such considerations to future work. %

We consider first the standard Koch snowflake $\Gamma$ with $\beta=\pi/6$ and $\xi=1/3$, in the notation of \S\ref{sec:Snowflake}.
We approximate $\Gamma$ with both the inner (open) and the outer (compact) prefractals $\Gamma_j^-$ and $\Gamma_j^+$ defined in \S\ref{sec:Snowflake}, on which we build uniform meshes conforming to the prefractal geometries (strictly, our mesh is on $(\Gamma_j^+)^\circ$ in the outer prefractal case).
Figure~\ref{fig:SnowflakeFields} shows the scattered field for an incident plane wave with $k=61$ (wavelength $\lambda\approx0.103$) and $\bd=(0,\frac1{\sqrt2},-\frac1{\sqrt2})$, at inner prefractal level 4, computed with $N_4^-=10344$ DOFs.

\begin{figure}[t!]\centering
\includegraphics[width=.33\textwidth,clip,trim=40 20 40 10]{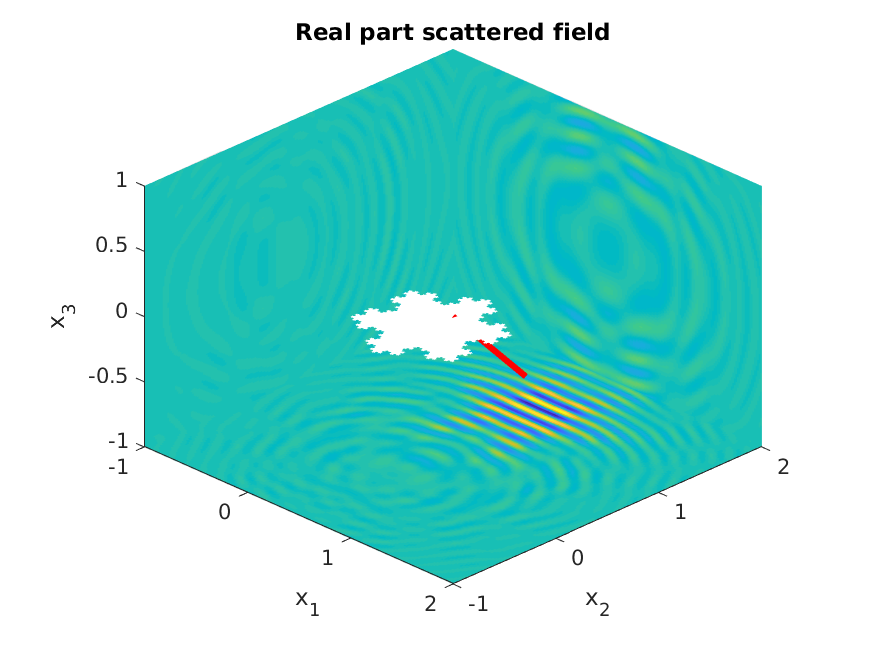}
\includegraphics[width=.33\textwidth,clip,trim=40 20 40 10]{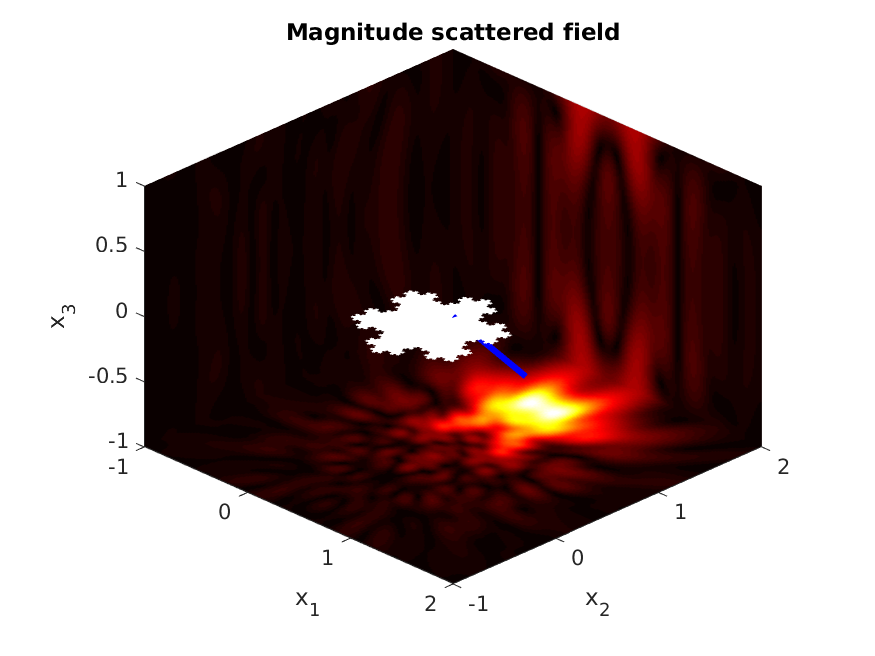}
\includegraphics[width=.32\textwidth,clip,trim=70 30 70 10]{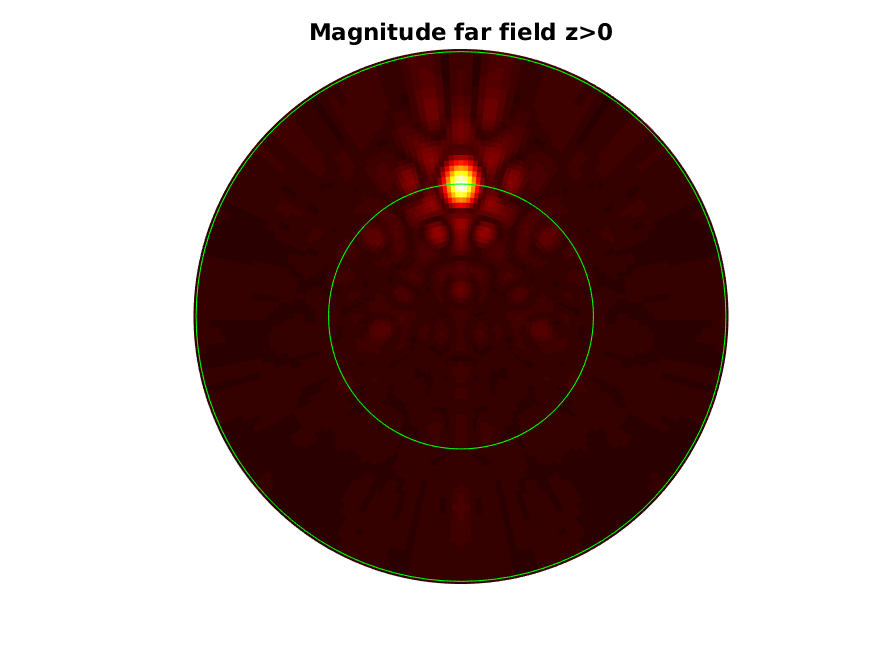}
\caption{The real part, magnitude and far-field pattern of the field scattered by the level 4 inner prefractal approximation $\Gamma_4^-$ of the Koch snowflake, for the problem in \S\ref{sec:NumSnowflake}, computed using a uniform BEM mesh with $N_4^-=10344$ DOFs.}
\label{fig:SnowflakeFields}
\end{figure}

By Proposition~\ref{prop:Snowflake}, which relies on the fact that $\Gamma$ is thick, so that $\widetilde{H}^{-1/2}(\Gamma)=H^{-1/2}_{\overline\Gamma}$, Galerkin BEM solutions for the inner and outer prefractals should both converge to the unique limiting solution on $\Gamma$, provided that $h_j\to 0$ as $j\to\infty$.
We investigate this numerically (for the collocation BEM) in Figure \ref{fig:FractalBEM-SnowflakeComparison-k61} (though our simulations use a fixed $h_j$ on each $\Gamma_j^-$ and a fixed mesh size also on each $\Gamma_j^+$, as is described in more detail below). This figure shows that the alternating inner/outer sequence of prefractal approximations $\Gamma_0^-,\Gamma_0^+,\Gamma_1^-,\Gamma_1^+,\Gamma_2^-,\Gamma_2^+,\Gamma_3^-,\Gamma_3^+,\Gamma_4^-$ has the property that the $H^\mhalf(\R^2)$ norm of the difference between the BEM solutions on consecutive prefractals in the sequence tends to zero monotonically (and approximately exponentially) as one moves along the sequence.
The figure also suggests that plane waves hitting the screen perpendicularly lead to the lowest relative difference
between solutions on inner and outer prefractals,
that grazing incident waves lead to the largest difference, and that the relative errors are rather insensitive to the wavenumber for the values of $k$ investigated.

\begin{figure}[t!]\centering
\includegraphics[width=0.6\textwidth,clip,trim=30 0 30 0]{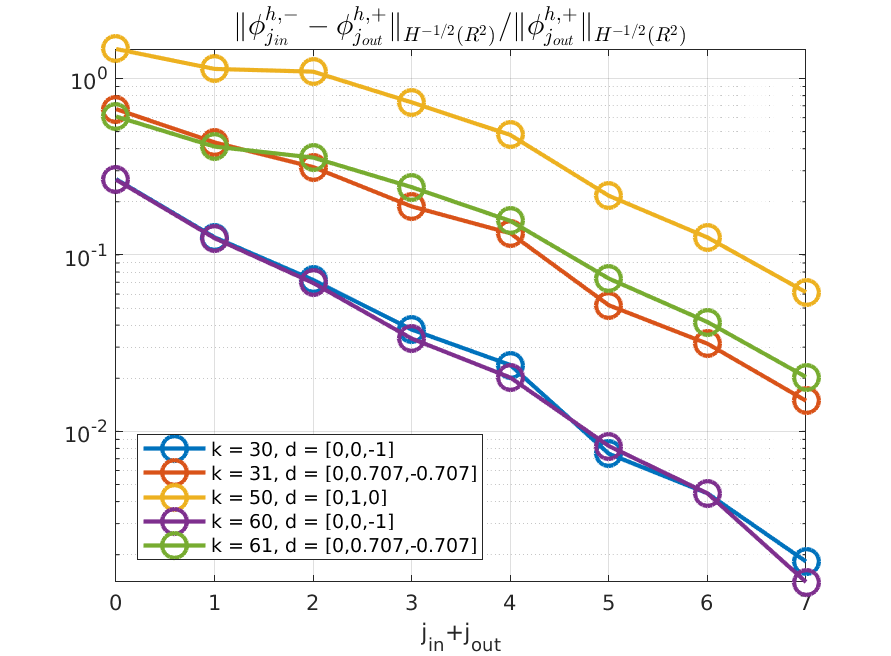}
\caption{The relative difference, measured in the $H^\mhalf$ norm, between the BEM solutions on $\Gamma_{j_{\rm in}}^-$ and on $\Gamma_{j_{\rm out}}^+$, for $j_{\rm in}+j_{\rm out}=0,1,\ldots,7$, with $j_{\rm out}\in\{j_{\rm in}-1,j_{\rm in}\}$,
for the problem described in \S\ref{sec:NumSnowflake} (green curve). We show results also for four other similar problems with different values of $k$ and either vertically, horizontally or obliquely incident plane waves.
}
\label{fig:FractalBEM-SnowflakeComparison-k61}
\end{figure}

For completeness we give a brief explanation of how the results in Figure \ref{fig:FractalBEM-SnowflakeComparison-k61} were computed.
The inner prefractals $\Gamma_j^-$ of levels $0\le j\le4$ are the union of $1$, $12$, $120$, $1128$ and $10344$ equilateral triangles of side $3^{-j}$, respectively.
We mesh them all with equilateral triangles of the same side length $3^{-4}\approx 0.0123$, so that the total numbers of DOFs on the respective meshes $M_j^-$ are $N_0^-=6561=3^8$, $N_1^-=8748=3^6\cdot12$, $N_2^-=9720=3^4\cdot120$, $N_3^-=10152=3^2\cdot1128$, and $N_4^-=10344$.
The outer prefractals $\Gamma_j^+$ of levels $0\le j\le3$ are the union of $6$, $48$, $408$ and $3576$ equilateral triangles of side $3^{-j-\frac12}$, respectively. We mesh them all with equilateral triangles of the same side length $3^{-\frac72}\approx 0.0214$, so that the total numbers of DOFs on the respective meshes $M_j^+$ are $N_0^+=4374=3^6\cdot 6$, $N_1^+=3888= 3^4\cdot 48$, $N_2^+=3672=3^2\cdot 408$ and $N_3^+=3576$.
Figure~\ref{fig:SnowflakeFieldsInOut} shows the real part of the BEM solution $\phijBEM^\pm$ for these prefractals and meshes.

\begin{figure}[t!]\centering
\includegraphics[width=\textwidth,clip,trim= 200 108 150 85]{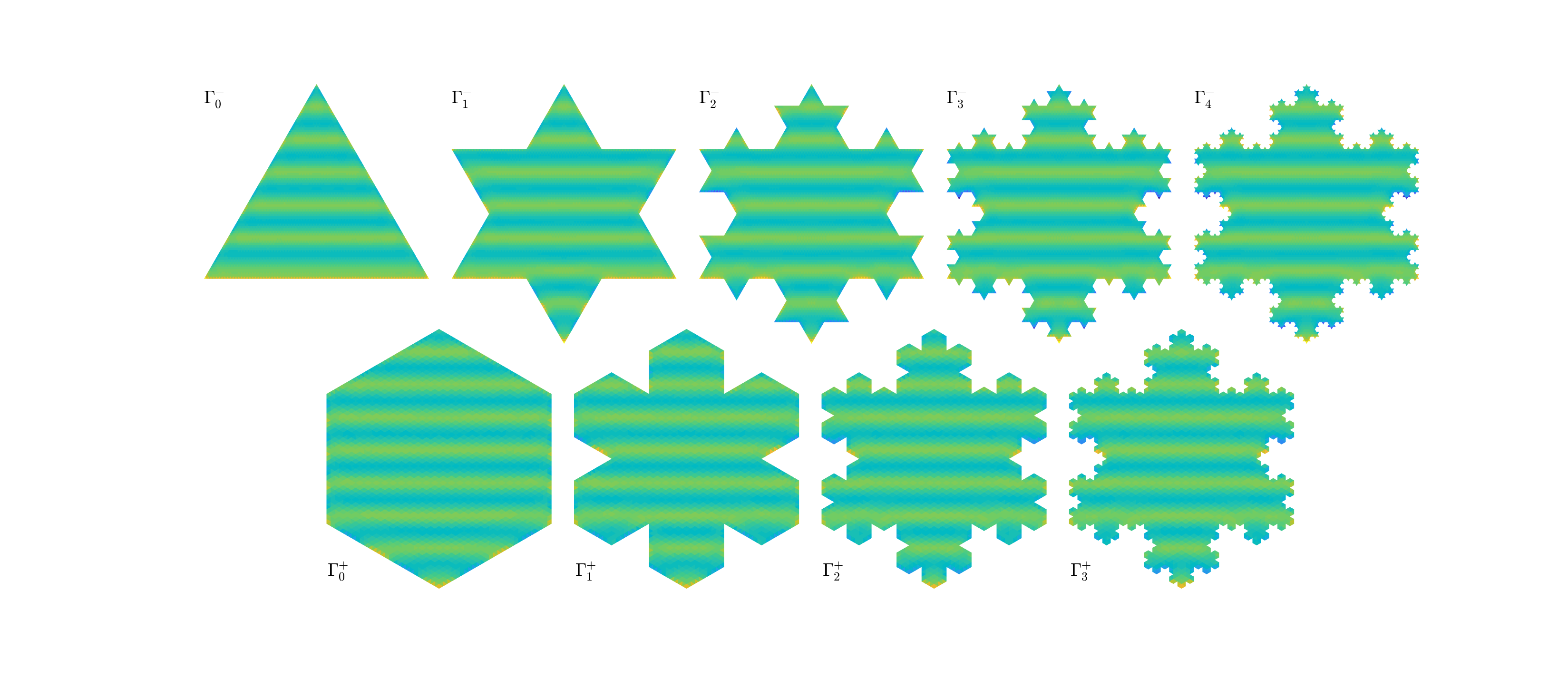}
\caption{The real part of the BEM solutions $\phijBEM^\pm$ on the inner and outer prefractals $\Gamma_0^-,\ldots,\Gamma_4^-$ and $\Gamma_0^+,\ldots,\Gamma_3^+$ of the standard Koch snowflake, for the problem in \S\ref{sec:NumSnowflake}.
All plots are in the same colour scale.}
\label{fig:SnowflakeFieldsInOut}
\end{figure}

To compute an approximation to the $H^\mhalf(\R^2)$ norm of the difference $\phijBEM^--\phijBEM^+$, as plotted in Figure \ref{fig:FractalBEM-SnowflakeComparison-k61}, we first represent
both piecewise-constant fields on the same mesh.
We note that the equilateral triangles of $M_j^+$ are rotated by an angle of $\pi/2$ with respect to those of $M_j^-$ and are larger by a (linear) factor $\sqrt3$.
We define $M_j^*$ to be the smallest uniform mesh that extends $M_j^-$ and covers $\Gamma_j^+$; see Figure~\ref{fig:SnowflakeMeshInOut} for an illustration.
We define two piecewise constant functions $\psi_j^\pm$ on this mesh.
The first, $\psi_j^-$, is simply the zero-extension of $\phijBEM^-$. The second, $\psi_j^+$, is defined from $\phijBEM^+$ as follows: noting that the centre of each element of $M_j^*$ lies on an edge of $M_j^+$, we define the value of $\psi_j^+$ on $T_{j,l}^*\in M_j^*$ to be the average of the values of (the zero-extension of) $\phijBEM^+$ on the two triangles of (the uniform extension of) $M_j^+$ intersecting $T_{j,l}^*$.
The norm $\|\psi_j^--\psi_j^+\|_{H^\mhalf(\R^2)}$ approximates $\|\phijBEM^--\phijBEM^+\|_{H^\mhalf(\R^2)}$ and can be computed using the reaction--diffusion single-layer operator as in the previous examples.
To compute the norm of
$\phi_{j+1}^{h^-}-\phijBEM^+$ we proceed similarly.

\begin{figure}[t!]\centering
\includegraphics[width=0.49\textwidth,clip,trim=60 120 50 130]{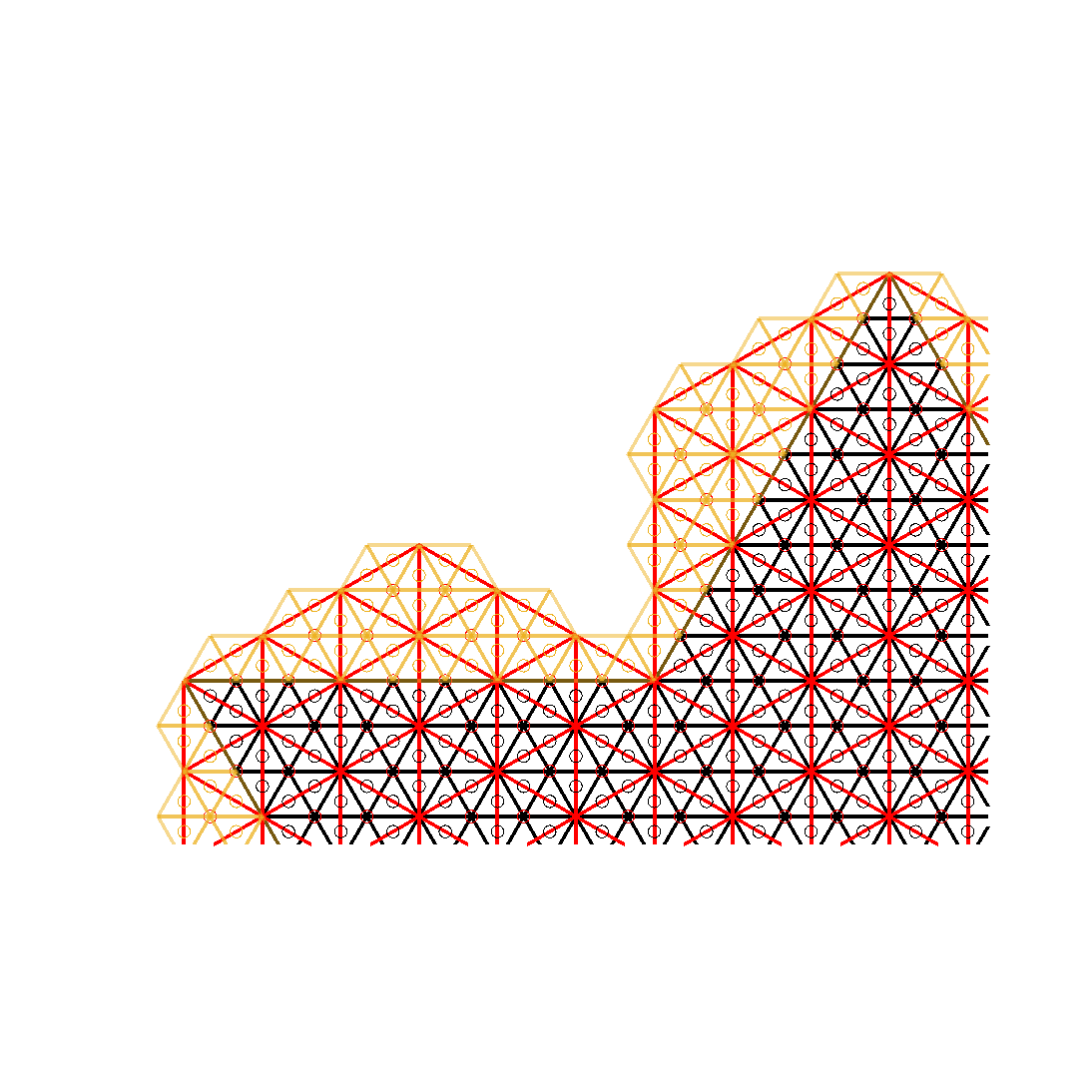}
\includegraphics[width=0.49\textwidth,clip,trim=60 120 50 130]{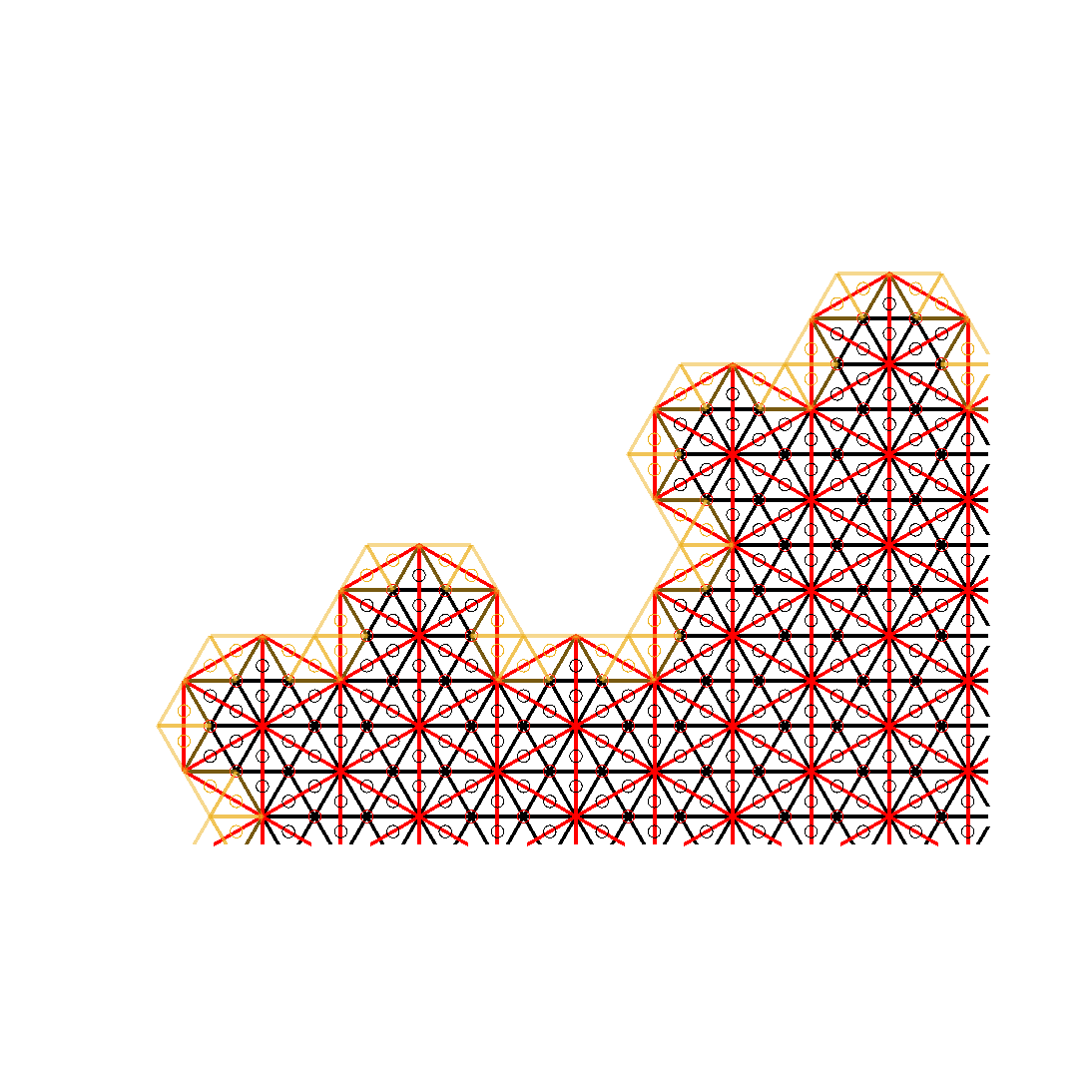}
\caption{Left: a portion of the meshes used to compute the norm of the difference of the BEM solutions ${\phi_2^h}^-$ and ${\phi_2^h}^+$ on the inner and the outer prefractals $\Gamma_2^-$ and $\Gamma_2^+$.
The black triangles are the elements of $M_2^-$ and the %
red triangles are the elements of $M_2^+$. %
The circles denote the centres of the triangles of the same colour.
The extended uniform mesh $M_2^*$ is composed of the black and the yellow triangles, which cover the whole of $\Gamma_2^+$.
Right: the analogous construction for the comparison between the solutions ${\phi_4^h}^-$ and ${\phi_3^h}^+$ on the prefractals $\Gamma_4^-$ and $\Gamma_3^+$.
}
\label{fig:SnowflakeMeshInOut}
\end{figure}

\subsection{Square snowflake}\label{sec:NumSquare}

Finally, we consider the square snowflake $\Gamma$ and the associated sequence of non-nested prefractals $\Gamma_j$ described in \S\ref{sec:Square}.
We choose $k=40$ and $\bd=(0,0,-1)$.
Plots of the resulting near- and far-field solutions for prefractal level $j=3$ are shown in Figure \ref{fig:SquareFields}.
Proposition~\ref{prop:SquareSn} implies that we have (Galerkin) BEM convergence provided $h_j\to0$ as $j\to\infty$. But numerical validation of this convergence is hampered by the fact that the minimal number of DOFs required to discretise the $j$th-level prefractal with a uniform mesh of squares is $16^j$. In order to simulate higher-level prefractals more sophisticated BEM implementations are needed, e.g.\ using fast matrix--vector multiplications, and non-uniform or non-convex meshes.
This will be considered in future work.

\begin{figure}[t!]\centering\
\includegraphics[width=.33\textwidth,clip,trim=40 20 40 10]{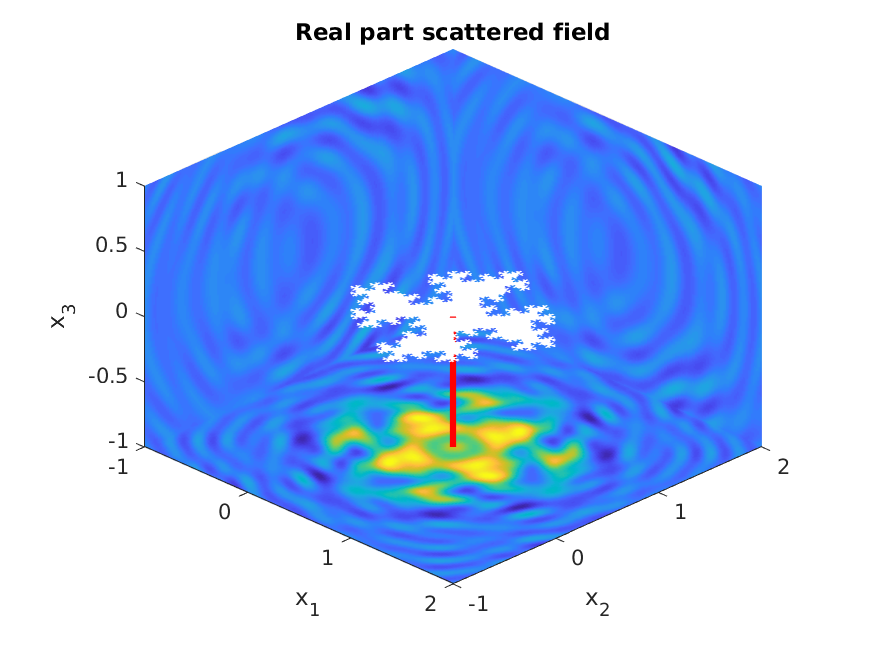}
\includegraphics[width=.33\textwidth,clip,trim=40 20 40 10]{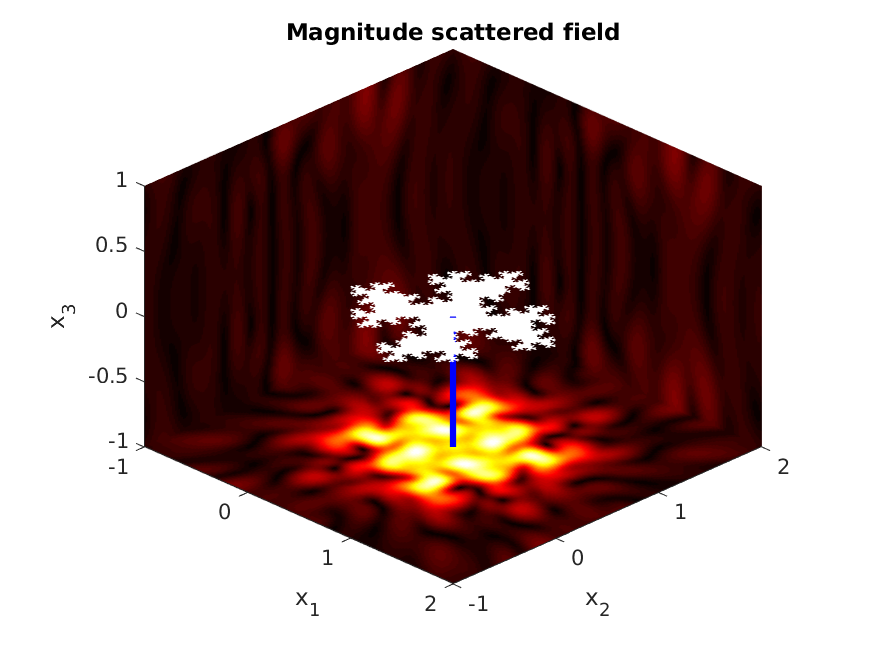}
\includegraphics[width=.31\textwidth,clip,trim=70 30 70 10]{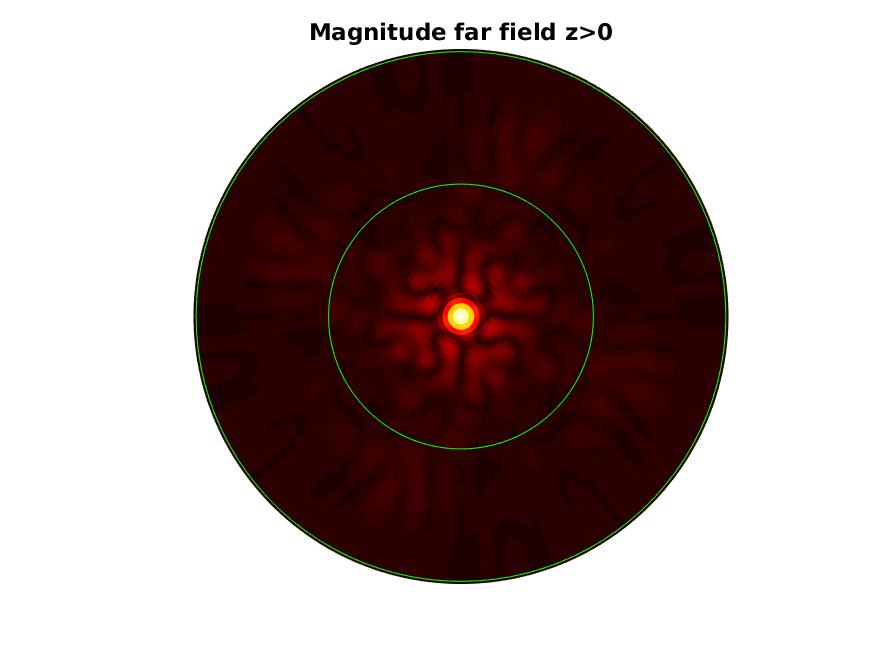}
\caption{Left and centre: the real part and magnitude of the field scattered by the level 3 prefractal approximation of the square snowflake, for the problem in \S\ref{sec:NumSquare}, computed with $N_3=16^3=4096$ DOFs, i.e.\ $\approx10$ DOFs per wavelength.
Right: the far-field pattern for the same problem.}
\label{fig:SquareFields}
\end{figure}

\section{Conclusion and open problems} \label{sec:conclude}
In this paper we have written down in \S\ref{sec:BVPs} well-posed BVP and BIE formulations for scattering by an acoustic screen $\Gamma$ that is either fractal or has fractal boundary, refining, simplifying, and extending (in particular through application of recent results from \cite{caetano2018}) the earlier results of \cite{ScreenPaper}. Generalising the treatment in \cite{ScreenPaper}, we have studied in \S\ref{sec:Conv} (using Mosco convergence for the first time in this context) the convergence of BIE solutions on a sequence of prefractals $\Gamma_j$ to the solution on a limiting set $\Gamma$; in particular we have proved in Theorem \ref{thm:MoscoConv} and Proposition~\ref{prop:Mosco} new convergence results in cases where the sequence of  prefractals $\Gamma_j$ is not monotonically convergent to $\Gamma$, and have applied these results in \S\ref{sec:IFSpart1} to sequences of prefractals generated by general iterated function systems.

But the crucial novelty of the paper has been the analysis and implementation of the BEM in \S\S\ref{sec:Convergence}--\ref{sec:Numerics}. In \S\ref{sec:IFSpart1} we have obtained,  to our knowledge, the first rigorous convergence analysis of a numerical method for scattering by a fractal object and the first proofs of convergence of BEM, applied on a sequence of prefractal sets $\Gamma_j$, to the solution of the BIE on the limiting set $\Gamma$, both in cases when $\Gamma$ is fractal (in particular the fixed point of an IFS of contracting similarities satisfying the standard open set condition, for example the Cantor set or dust, the Sierpinski triangle) and in cases where $\Gamma$ is an open set with fractal boundary (for example the Koch or square snowflake). We have studied these specific cases as examples of applying these new convergence results in \S\ref{sec:examples}, and in numerical experiments in \S\ref{sec:Numerics}. %

These results are, we consider, important first steps in understanding BIEs and their solutions on sets that are fractal or have fractal boundary, and the convergence to these solutions of BEM approximations on sequences of more regular prefractals.
More generally, these results and methods are applicable to the convergence analysis of Galerkin methods for BIEs (and, we anticipate, other integral equations and PDEs) posed on {\em any} rough domain that cannot be discretised exactly but first needs to be approximated by a sequence of more regular sets.

For the specific screen scattering problems that have been the focus of this paper, and for large classes of related problems, there remain many intriguing and important open questions. These include:

\begin{enumerate}
\item What is the regularity of the BIE solution on the limiting screen $\Gamma$, and how does this depend on the fractal dimension of $\Gamma$ or its boundary? %
    In the case when $\Gamma$ is fractal the density results summarised in Lemma \ref{lem:Density} are one step towards addressing this question.
    \item At what rate do BIE solutions on a prefractal sequence $\Gamma_j$ converge to the limiting BIE solution on $\Gamma$?
    \item If the BIE is additionally discretised by BEM, how does the convergence rate depend jointly on the prefractal level $j$ and the discretisation? (The numerical simulations in \S\ref{sec:Numerics} provide some experimental data.)
    \item Extending this question, what is the optimal balance (to achieve an accurate approximation of the solution on $\Gamma$ with least work) between increasing the prefractal level and mesh refinement? For example, for the family of pre-convex meshes \eqref{eq:Mj2}, what is the optimal choice of the parameter $i(j)$?
        \item In the case when $\Gamma$ is a self-similar fractal (is the attractor of an IFS of contracting similarities), can efficient BEM solvers be built making use of the self-similarity?
\end{enumerate}
We hope to address some of these questions in future work, together with exploring more efficient BEM implementations (e.g.\ using locally-refined meshes, fast iterative solvers for structured meshes),
and the extension of the methods developed here to more general problems (e.g.\ curvilinear screens, different boundary conditions, elastic and electromagnetic waves, other PDEs).

\section{Acknowledgements}
We acknowledge support from EPSRC grants EP/S01375X/1 and EP/N019407/1, from GNCS--INDAM, from MIUR through the ``Dipartimenti di Eccellenza'' Programme (2018--2022)--Dept.\ of Mathematics, University of Pavia, and from PRIN for the project ``NA\_FROM-PDEs''. %
The last author acknowledges funding from the ``Bourses de stages \`a l'\'etranger'' scheme of Universit\'e Paris-Saclay, which supported a research visit to UCL in May-August 2018.
The authors are grateful to M. Berry, A.~Buffa, A.~Caetano, R. Capitanelli, R.~Hiptmair, M.~R.~Lancia, U.~Mosco, M.~Negri, G.\ Savar\'e, C.~Schwab, E.~Spence and M.~A.~Vivaldi for instructive discussions.

\appendix

\section{Finite element approximation by piecewise constants} \label{sec:fem}

Let $m\in\N$ and $\Omega \subset \R^m$ be a non-empty bounded open set. Let us say that $M$ is a {\em pre-convex mesh} on $\Omega$ if, for some $N\in \N$, $M=\{T_1,T_2,\ldots,T_N\}$ is a collection of open subsets $T_j\subset \Omega$ such that: (i) the convex hulls $\mathrm{Conv}(T_j)$ are pairwise disjoint; (ii) each $\partial T_j$ has zero $m$-dimensional Lebesgue measure; and (iii)
$\Omega = \left(\,\overline{\bigcup_{j=1}^N T_j}\,\right)^\circ$.  We note that (ii) holds (and this will be the case in the applications that we make) if each $T_j$ is the union of a finite number of pairwise disjoint Lipschitz open sets, in particular if each $T_j$ is convex, in which case we term $M$ a {\em convex mesh}.

In the case %
that $\Omega$ is a curvilinear Lipschitz polygon and each $T_j$ is a curvilinear triangle or quadrilateral, the following lemma is a standard BEM error estimate (e.g.\ \cite[Thm.~10.4, (10.10)]{Steinbach}), except that in the standard versions the explicit constant $\pi^{s-t}$ in \eqref{eq:errest} is replaced by an unknown  constant that depends (in an unspecified way) on the domain and the shape regularity of the elements. The version we prove here, which applies to any bounded open set $\Omega$ and any pre-convex mesh $M$, and provides explicit constants independent of the domain and element shape, should be of some independent interest and is essential for our application to BEM on sequences of prefractals converging to a fractal limit.

\begin{lem} \label{lem:PWc}
Let $m\in\N$ and $N\in \N$, and let $\Omega \subset \R^m$ be a non-empty bounded open set and $M=\{T_1,T_2,\ldots,T_N\}$ be a pre-convex mesh on $\Omega$.
Let $V\subset L^2(\Omega)$ denote the set of piecewise constant functions on $M$, so that $u\in V$ if $u\in L^2(\Omega)$ and $u|_{T_j}$ is constant, for $j=1,\ldots,N$. Let $\Pi:L^2(\Omega)\to V$ be orthogonal projection, so that
 $$
 (\Pi u)|_{T_j} = \frac{1}{|T_j|} \int_{T_j} u(\bx)\, \rd \bx, \quad j=1,\ldots,N.
 $$
Let $h:= \max_{T\in M} \diam(T)$.
Then, for $-1\leq s\leq 0$ and $0\leq t\leq 1$, if $u\in H^t(\Omega)$, then
\begin{equation} \label{eq:errest}
\|u-\Pi u\|_{\tH^s(\Omega)} \leq (h/\pi)^{t-s} \|u\|_{H^t(\Omega)},
\end{equation}
where on the left hand side we extend $u-\Pi u$ from $\Omega$ to $\R^m$ by zero to become an element of $\tH^0(\Omega)\subset \tH^s(\Omega)$.
\end{lem}
\begin{proof} Clearly \rf{eq:errest} holds for $s=t=0$ as $H^0(\Omega)=L^2(\Omega)$, with identical norms, and $\Pi$ is an orthogonal projection operator on $L^2(\Omega)$.
Now suppose that $s=0$ and $t=1$.
In this case, recalling that $\tH^0(\Omega)=\overline{C^\infty_0(\Omega)}^{L^2(\R^m)}=L^2(\Omega)$ with equal norms, and that $\|u\|_{H^1(\Omega)}=\min_{\substack{U\in H^1(\R^m),\;u=U|_\Omega}}\|U\|_{H^1(\R^m)}$, the required bound \eqref{eq:errest} is implied by
\begin{equation} \label{eq:poin}
\|u-\Pi u\|_{L^2(\Omega)} \leq \frac{h}{\pi} \|\nabla U\|_{L^2(\R^m)},
\end{equation}
where $U\in H^1(\R^m)$ is such that $u=U|_{\Omega}$ and $\|u\|_{H^1(\Omega)}=\|U\|_{H^1(\R^m)}$. To see that this bound holds, let $T_j^H$ denote the convex hull of $T_j$, for $j=1,\ldots,N$, which has the same diameter as $T_j$, and let $\Omega^H:= \left(\,\bigcup_{j=1}^N \overline{T^H_j}\,\right)^\circ$, so that $M^H:=\{T_1^H,\ldots,T_N^H\}$ is a convex mesh on $\Omega^H$. Let $V^H\subset L^2(\Omega^H)$ denote the space of piecewise constant functions on $M^H$,  $\Pi^H:L^2(\Omega^H)\to V^H$ be orthogonal projection, and $u^H := U|_{\Omega^H}$. Then $(\Pi^H u^H)|_\Omega \in V$ so that
\begin{eqnarray*} \nonumber
\|u-\Pi u\|^2_{L^2(\Omega)} &\leq & \|u-(\Pi^H u^H)|_\Omega\|^2_{L^2(\Omega)}\\
& \leq &\|u^H-\Pi^H u^H\|^2_{L^2(\Omega^H)}=  \sum_{j=1}^N \|(u^H-\Pi^H u^H)|_{T^H_j}\|^2_{L^2(T^H_j)}.
\end{eqnarray*}
But
$$
\|(u^H-\Pi^H u^H)|_{T^H_j}\|_{L^2(T^H_j)} \leq C \|(\nabla U)|_{T^H_j}\|_{L^2(T_j^H)}
$$
for some $C>0$, by the standard Poincar\'e inequality. Indeed, since $T_j^H$ is convex, this bound holds with $C=h_j/\pi$, where $h_j=\mathrm{diam}(T_j)=\mathrm{diam}(T^H_j)$, which is known \cite{Bebendorf03,PayneWeinberger60} to be the best constant in this inequality for convex domains. The inequality \eqref{eq:poin} follows since the $T_H^j$ are pairwise disjoint.

Thus \eqref{eq:errest} is proved for $s=0$ and $t=0,1$, and by interpolation it extends to $s=0$ and $0< t< 1$. Precisely, let $\|\cdot\|_t$ denote the norm on the $K$-method interpolation space $K_{t,2}((L^2(\Omega),H^1(\Omega)))$, with the normalisation defined in \cite[(7,8)]{InterpolationCWHM}. Then %
 \begin{equation} \label{eq:errest2}
\|u-\Pi u\|_{L^2(\Omega)} \leq (h/\pi)^{t} \|u\|_t \leq  (h/\pi)^{t} \|u\|_{H^t(\Omega)}, \quad 0<t<1,
\end{equation}
where the first inequality follows by interpolation since $K_{t,2}((L^2(\Omega),L^2(\Omega)))=L^2(\Omega)$ with equality of norms given the normalisation \cite[(7,8)]{InterpolationCWHM} (see \cite[Lem.~2.1(iii), Thm.~2.2(i)]{InterpolationCWHM}); the second inequality holds since, again given this normalisation,  the embedding of $H^t(\Omega)$ into $K_{t,2}((L^2(\Omega),H^1(\Omega)))$ has norm $\leq 1$ \cite[Lem.~4.2]{InterpolationCWHM}.

Finally, suppose that $-1\leq s <0$ and $0\leq t\leq 1$. %
Then, as $H^{-s}(\Omega)$ is a unitary realisation of the dual space of $\tH^s(\Omega)$ via an extension of the $L^2(\Omega)$ inner product, if $u\in H^t(\Omega)$ then %
\begin{align*}
\|u-\Pi u\|_{\tH^s(\Omega)} &= \sup_{0\neq v\in H^{-s}(\Omega)} \frac{|(u-\Pi u,v)_{L^2(\Omega)}|}{\|v\|_{H^{-s}(\Omega)}}\\
&= \sup_{0\neq v\in H^{-s}(\Omega)} \frac{|(u-\Pi u,v-\Pi v)_{L^2(\Omega)}|}{\|v\|_{H^{-s}(\Omega)}}%
 \leq  (h/\pi)^{t-s} \|u\|_{H^t(\Omega)},
\end{align*}
where the final bound follows from the Cauchy--Schwarz inequality and two applications of \eqref{eq:errest2}.
\end{proof}

\section{Mollification and Sobolev regularity}\label{sec:Mollification}
Given $m\in\N$, choose $\psi\in C^\infty_0(\R^{m})$ with $\psi(\bx)=0$ for $|\bx|>1$ and $\int_{\R^{m}}\psi(\bx) \rd \bx = 1$, and for $\epsilon>0$ let
$$
\psi_\epsilon(\bx) := \epsilon^m \psi(\bx/\epsilon), \quad \bx\in \R^{m}.
$$
Given $s\in\R$, for $\phi\in H^s(\R^{m})$ and $\epsilon>0$ we define $\phi_\epsilon$, the $\epsilon$-mollification of $\phi$, by $\phi_\epsilon := \psi_\epsilon * \phi$, where
$$
\psi_\epsilon * \phi(\bx) := \int_{\R^{m}}\psi_\epsilon(\bx-\by)\phi(\by)\, \rd \by, \quad \bx \in \R^{m},
$$
in the case that $\phi\in L^2(\Gamma)$, and, in the case that $s<0$, the definition of $\psi_\epsilon * \phi$ is extended to all $\phi\in H^s(\R^m)$ by continuity and density.
It is standard that $\phi_\eps\to \phi$ in $H^s(\R^m)$ as $\epsilon\to0$ (e.g.\ \cite[Exercise 3.17]{McLean}). Moreover, with the Fourier transform normalised as in \S\ref{sec:SobolevSpaces},
we have that
$$
\widehat \psi_\epsilon(\bxi) = \widehat\psi(\epsilon \bxi)
\qquad \text{and}\qquad
\widehat\phi_\epsilon(\bxi)= c_{m} \widehat \psi_\epsilon (\bxi)\widehat \phi(\bxi), \quad \mbox{for } \bxi\in \R^{m},
$$
where $c_{m} := (2\pi)^{m/2}$.
Hence it holds for $t\geq s$ that
\begin{eqnarray*}
\|\phi_\epsilon\|^2_{H^t(\R^{m})} &=& c_{m}^2\int_{\R^{m}}|\widehat\psi(\epsilon \bxi)|^2\, |\widehat \phi(\bxi)|^2(1+\bxi^2)^{t} \, \rd \bxi \leq  c_{m}^2C_\epsilon^2 \,\|\phi\|^2_{H^s(\R^{m})},
\end{eqnarray*}
where
$$
C_\epsilon := \sup_{\bxi\in \R^{m}}\left|\widehat \psi(\epsilon \bxi)(1+\bxi^2)^{(t-s)/2}\right|  = \sup_{\bxi\in \R^{m}}\left|\widehat \psi( \bxi)(1+(\bxi/\epsilon)^2)^{(t-s)/2}\right|.
$$
Since $\psi\in C^\infty_0(\R^{m})$, it holds for $0<\epsilon\le1$ that
$$
C_\epsilon \leq c'_{t-s} \sup_{\bxi\in \R^{m}}\left(\frac{1+\bxi^2/\epsilon^2}{1+\bxi^2}\right)^{(t-s)/2} \leq c'_{t-s}\, \epsilon^{s-t},
\;\,
\text{where}
\;\,
c'_{p}:=\max_{\bxi\in\R^m}|\widehat\psi(\bxi)|(1+\bxi^2)^{p/2},
$$
for $p\geq 0$. Thus, for $t\geq s$, $\epsilon>0$, and $\phi\in H^s(\R^m)$, it holds that
\begin{eqnarray} \label{eq:mollbound}
\|\phi_\epsilon\|_{H^t(\R^{m})}  \leq c_{m}c'_{t-s}\, \epsilon^{s-t} \,\|\phi\|_{H^s(\R^{m})}.
\end{eqnarray}

\bibliographystyle{siamplain}
\bibliography{BEMbib_short2014a}
\end{document}